\documentclass[12pt,a4paper]{article}

\usepackage{amsmath,amstext,amssymb,amscd,euscript}

\newcommand{\Xcomment}[1]{}

\newtheorem{theorem}{Theorem}[section]
\newtheorem{lemma}[theorem]{Lemma}
\newtheorem{corollary}[theorem]{Corollary}

\newtheorem{prop}[theorem]{Proposition}

\makeatletter \@addtoreset{equation}{section} \makeatother

\newenvironment{proof}{\noindent{\bf Proof}\/}%
{\hfill$\qed$\medskip}

\def\qed{ \ \vrule width.1cm height.3cm depth0cm}

\newenvironment{numitem}{\refstepcounter{equation}\begin{enumerate}%
\item[(\thesection.\arabic{equation})]$\quad$}{\end{enumerate}}
\newenvironment{numitem1}{\refstepcounter{equation}\begin{enumerate}%
\item[(\thesection.\arabic{equation})]}{\end{enumerate}}

\newcommand{\refeq}[1]{(\ref{eq:#1})}  

\def\rest#1{_{\,\vrule height 1.6ex width 0.05em depth 0pt\, #1}}

 \makeatletter
\renewcommand{\section}{\@startsection{section}{1}{0pt}%
{-3.5ex plus -1ex minus -.2ex}{2.3ex plus .2ex}%
{\normalfont\Large}}
 \makeatother

 \makeatletter
\renewcommand{\subsection}{\@startsection{subsection}{2}{0pt}%
{-3.0ex plus -1ex minus -.2ex}{1.5ex plus .2ex}%
{\normalfont\normalsize\bf}}
 \makeatother

\def\Rset{{\mathbb R}}
\def\Nset{{\mathbb N}}
\def\Zset{{\mathbb Z}}

\def\Ascr{{\cal A}}
\def\Bscr{{\cal B}}
\def\Cscr{{\cal C}}
\def\Dscr{{\cal D}}
\def\Escr{{\cal E}}

\def\Iscr{{\cal I}}
\def\Kscr{{\cal K}}

\def\Mscr{{\cal M}}
\def\Pscr{{\cal P}}

\def\Rscr{{\cal R}}

\def\frakF{{\mathfrak F}}
\def\frakL{{\mathfrak L}}

\def\frakS{{\mathfrak S}}

\def\tilde{\widetilde}
\def\hat{\widehat}
\def\bar{\overline}

\def\deltain{\delta^{\rm in}}
\def\deltaout{\delta^{\rm out}}
\def\Zin{Z^{\rm in}}
\def\Zout{Z^{\rm out}}

\def\Mlh{M^{\rm lh}}
\def\Muh{M^{\rm uh}}
\def\Mlhmax{M^{\rm lh}_{\rm max}}
\def\Muhmax{M^{\rm uh}_{\rm max}}
\def\Mver{M^{\rm vert}}
\def\Succ{{\rm Succ}}
\def\ISucc{{\rm ISucc}}
\def\Vodd{{\rm V^{\rm odd}}}
\def\Veven{{\rm V^{\rm even}}}

\def\odivide{/}

\begin{document}

 \begin{center}
{\large\bf Planar flows and quadratic relations over semirings%
\footnote[1]{Supported by RFBR grant 10-01-9311-CNRSL\_\,a.}}
 \end{center}

 \begin{center}
{\sc Vladimir~I.~Danilov}\footnote[2] {Central Institute of Economics and
Mathematics of the RAS, 47, Nakhimovskii Prospect, 117418 Moscow, Russia; emails:
danilov@cemi.rssi.ru (V.I.~Danilov); koshevoy@cemi.rssi.ru (G.A.~Koshevoy).},
{\sc Alexander~V.~Karzanov}\footnote[3]{Institute for System Analysis of the
RAS, 9, Prospect 60 Let Oktyabrya, 117312 Moscow, Russia; email:
sasha@cs.isa.ru. Corresponding author.},
{\sc Gleb~A.~Koshevoy}$^2$
\end{center}



 \begin{quote}
 {\bf Abstract.} \small
Adapting Lindstr\"om's well-known construction, we consider a wide class of
functions which are generated by flows in a planar acyclic directed graph whose
vertices (or edges) take weights in an arbitrary commutative semiring. We give
a combinatorial description for the set of ``universal'' quadratic relations
valid for such functions. Their specializations to particular semirings involve
plenty of known quadratic relations for minors of matrices (e.g., Pl\"ucker
relations) and the tropical counterparts of such relations. Also some
applications and related topics are discussed.
 \medskip


  \end{quote}

\parskip=3pt


\section{\Large Introduction}  \label{sec:intr}

In this paper we consider functions which take values in a commutative semiring
and are generated by planar flows. Functions of this sort satisfy plenty of
quadratic relations, and our goal is to describe a combinatorial method to
reveal and prove such relations. One important class consists of quadratic
relations of Pl\"ucker type.

Recall some basic facts concerning Pl\"ucker algebra and Pl\"ucker coordinates.
For a positive integer $n$, let $[n]$ denote the set $\{1,2,\ldots,n\}$.
Consider the $n\times n$ matrix $\bf x$ of indeterminates $x_{ij}$ and the
corresponding commutative polynomial ring $\Zset[\bf x]$. Also consider the
polynomial ring $\Zset[\Delta]$ generated by variables $\Delta_S$ indexed by
the subsets $S\subseteq [n]$. They are linked by the natural ring homomorphism
$\psi:\Zset[\Delta]\to\Zset[\bf x]$ that brings each variable $\Delta_S$ to the
flag minor polynomial for $S$, i.e., to the determinant of the submatrix
$\bf{x}_S$ formed by the column set $S$ and the row set $\{1,\ldots,|S|\}$ of
$\bf x$. An important fact is that the ideal ${\rm ker}(\psi)$ of
$\Zset[\Delta]$ is generated by homogeneous quadrics, each being a certain
integer combination of products $\Delta_S\Delta_{S'}$. They correspond to
quadratic relations on the Pl\"ucker coordinates of an invertible $n\times n$
matrix (regarded as a point of the corresponding flag manifold embedded in an
appropriate projective space); for a survey see, e.g.,~\cite[Ch.~14]{MS}.

There are many quadratic Pl\"ucker relations on flag minors of a matrix whose
entries are assumed to belong to an arbitrary commutative ring $\mathfrak{R}$
(the case $\mathfrak{R}=\Rset$ or $\mathbb{C}$ is most popular). Let $f(S)$
denote the flag minor with a column set $S$ in this matrix.

The simplest examples of Pl\"ucker relations involve triples: for any three
elements $i<j<k$ in $[n]$ and any subset $X\subseteq[n]-\{i,j,k\}$, the flag
minor function $f:2^{[n]}\to \mathfrak{R}$ of an $n\times n$ matrix satisfies
   \begin{equation} \label{eq:AP3}
    f(Xik)f(Xj)-f(Xij)f(Xk)-f(Xjk)f(Xi)=0,
    \end{equation}
where for brevity we write $Xi'\ldots j'$ for $X\cup\{i',\ldots,j'\}$. We
call~\refeq{AP3} the \emph{AP3-relation} (abbreviating ``algebraic Pl\"ucker
relation with triples").  Another well-known special case (in particular,
encountered in a characterization of Grassmannians) involves quadruples
$i<j<k<\ell$; this is of the form
   \begin{equation} \label{eq:AP4}
    f(Xik)f(Xj\ell)-f(Xij)f(Xk\ell)-f(Xi\ell)f(Xjk)=0.
    \end{equation}

A general algebraic Pl\"ucker relation on flag minors of a matrix can be
written as
  \begin{equation} \label{eq:a_pluck}
  \sum_{A\in\Ascr} f(X\cup A)\,f(X\cup(Y-A))-
     \sum_{B\in\Bscr} f(X\cup B)\,f(X\cup(Y-B))=0.
  \end{equation}
Here $X$ and $Y$ are disjoint subsets of $[n]$, and $\Ascr$ and $\Bscr$ are
certain families of $p$-element subsets of $Y$, for some $p$.

In fact, an instance of~\refeq{a_pluck} (such as~\refeq{AP3} or~\refeq{AP4})
represents a class of relations of ``the same type''. More precisely, let
$m:=|Y|$ and define $\gamma_Y$ to be the order preserving bijective map $[m]\to
Y$, i.e., $\gamma_Y(i)<\gamma_Y(j)$ for $i<j$. This gives the families
$\Ascr_0,\Bscr_0$ of $p$-element subsets of the initial interval $[m]$ such
that $\Ascr=\{\gamma_Y(C)\,\colon C\in\Ascr_0\}$ and
$\Bscr=\{\gamma_Y(C)\,\colon C\in\Bscr_0\}$. We call $\Ascr_0$ the
\emph{pattern} of $\Ascr$ and write $\Ascr=\gamma_Y(\Ascr_0)$, and similarly
for $\Bscr_0$ and $\Bscr$. When considering a class of functions $f:2^{[n]}\to
\Rscr$ and speaking of~\refeq{a_pluck} as a ``universal'' (or ``stable'')
relation, we require that~\refeq{a_pluck} be valid for all functions within
this class and depend only on $m$, $p$, and the patterns $\Ascr_0,\Bscr_0$, but
not on $X$ and $Y$. Namely,~\refeq{a_pluck} should hold for any choice of
disjoint $X,Y\subseteq[n]$ with $|Y|=m$ and for the corresponding families
$\Ascr:=\gamma_Y(\Ascr_0)$ and $\Bscr:=\gamma_Y(\Bscr_0)$.

In particular,~\refeq{a_pluck} turns into~\refeq{AP3} when $m=3$, $p=2$,
$\Ascr_0=\{13\}$, $\Bscr_0=\{12,23\}$ and $Y=\{i,j,k\}$, and turns
into~\refeq{AP4} when $m=4$ $p=2$, $\Ascr_0=\{13\}$, $\Bscr_0=\{12,14\}$ and
$Y=\{i,j,k,\ell\}$.

An important fact established by Lindstr\"om~\cite{Li} is that the minors of
many matrices can be expressed in terms of flows in a planar graph. A certain
flow model will play a key role in our description; next we specify the
terminology and notation that we use. (A more general flow model yielding a
generalization of Lindstr\"om's result is given in~\cite{Po,Ta}.)

By a \emph{planar network} we mean a finite directed planar \emph{acyclic}
graph $G=(V,E)$ in which two subsets $S=\{s_1,\ldots,s_n\}$ and
$T=\{t_1,\ldots,t_{n'}\}$ of vertices are distinguished, called the sets of
\emph{sources} and \emph{sinks}, respectively. We assume that these vertices,
also called \emph{terminals}, lie in the \emph{boundary} of a compact convex
region in the plane, which we denote by $O$ and sometimes conditionally call a
``circumference'', and the remaining part of the graph lies inside $O$. The
terminals appear in $O$ in the cyclic order $s_n,\ldots,s_1,t_1,\ldots,t_{n'}$
clockwise (with possibly $s_1=t_1$ or $s_n=t_{n'}$), and for convenience we say
that the sources and sinks lie in the ``lower'' and ``upper'' halves of $O$,
respectively, and that the indices in each set grow ``from left to right''.

Two important particular cases are: the (square) \emph{grid} $\Gamma_{n,n'}$
and the \emph{half-grid} $\Gamma^\triangle_n$, where the vertices in the former
are the integer points $(i,j)\in\Rset^2$ with $1\le i\le n$, $1\le j\le n'$,
the vertices in the latter are the integer points $(i,j)$ with $1\le j\le i\le
n$, and the edges in both cases are all possible ordered pairs of the form
$((i,j),(i-1,j))$ or $((i,j),(i,j+1))$. The sources are the points
$s_i:=(i,1)$, whereas the sinks are the points $t_j:=(1,j)$ in the former case,
and the $t_j:=(j,j)$ in the latter case. The graphs $\Gamma_{5,4}$ and
$\Gamma^\triangle_4$ are illustrated in Fig.~\ref{fig:Gamma}.

\begin{figure}[htb]
 \begin{center}
 \unitlength=0.8mm
\begin{picture}(120,35)
  \put(0,-5){\begin{picture}(50,35)   
 \put(0,5){\circle{2}}
 \put(10,5){\circle{2}}
 \put(20,5){\circle{2}}
 \put(30,5){\circle{2}}
 \put(40,5){\circle{2}}
 \put(0,15){\circle{2}}
 \put(10,15){\circle*{1.7}}
 \put(20,15){\circle*{1.7}}
 \put(30,15){\circle*{1.7}}
 \put(40,15){\circle*{1.7}}
 \put(0,25){\circle{2}}
 \put(10,25){\circle*{1.7}}
 \put(20,25){\circle*{1.7}}
 \put(30,25){\circle*{1.7}}
 \put(40,25){\circle*{1.7}}
 \put(0,35){\circle{2}}
 \put(10,35){\circle*{1.7}}
 \put(20,35){\circle*{1.7}}
 \put(30,35){\circle*{1.7}}
 \put(40,35){\circle*{1.7}}
 \put(39,5){\vector(-1,0){8}}
 \put(29,5){\vector(-1,0){8}}
 \put(19,5){\vector(-1,0){8}}
 \put(9,5){\vector(-1,0){8}}
 \put(39,15){\vector(-1,0){8}}
 \put(29,15){\vector(-1,0){8}}
 \put(19,15){\vector(-1,0){8}}
 \put(9,15){\vector(-1,0){8}}
 \put(39,25){\vector(-1,0){8}}
 \put(29,25){\vector(-1,0){8}}
 \put(19,25){\vector(-1,0){8}}
 \put(9,25){\vector(-1,0){8}}
 \put(39,35){\vector(-1,0){8}}
 \put(29,35){\vector(-1,0){8}}
 \put(19,35){\vector(-1,0){8}}
 \put(9,35){\vector(-1,0){8}}
 \put(40,6){\vector(0,1){8}}
 \put(40,16){\vector(0,1){8}}
 \put(40,26){\vector(0,1){8}}
 \put(30,6){\vector(0,1){8}}
 \put(30,16){\vector(0,1){8}}
 \put(30,26){\vector(0,1){8}}
 \put(20,6){\vector(0,1){8}}
 \put(20,16){\vector(0,1){8}}
 \put(20,26){\vector(0,1){8}}
 \put(10,6){\vector(0,1){8}}
 \put(10,16){\vector(0,1){8}}
 \put(10,26){\vector(0,1){8}}
 \put(0,6){\vector(0,1){8}}
 \put(0,16){\vector(0,1){8}}
 \put(0,26){\vector(0,1){8}}
 \put(0,1){\makebox(0,0)[cc]{$s_1$}}
 \put(10,1){\makebox(0,0)[cc]{$s_2$}}
 \put(20,1){\makebox(0,0)[cc]{$s_3$}}
 \put(30,1){\makebox(0,0)[cc]{$s_4$}}
 \put(40,1){\makebox(0,0)[cc]{$s_5$}}
 \put(-4,7){\makebox(0,0)[cc]{$t_1$}}
 \put(-4,17){\makebox(0,0)[cc]{$t_2$}}
 \put(-4,27){\makebox(0,0)[cc]{$t_3$}}
 \put(-4,37){\makebox(0,0)[cc]{$t_4$}}
   \end{picture}}
   \put(80,-5){\begin{picture}(40,35)  
 \put(0,5){\circle{2}}
 \put(10,5){\circle{2}}
 \put(20,5){\circle{2}}
 \put(30,5){\circle{2}}
 \put(10,15){\circle{2}}
 \put(20,15){\circle*{1.7}}
 \put(30,15){\circle*{1.7}}
 \put(20,25){\circle{2}}
 \put(30,25){\circle*{1.7}}
 \put(30,35){\circle{2}}
 \put(29,5){\vector(-1,0){8}}
 \put(19,5){\vector(-1,0){8}}
 \put(9,5){\vector(-1,0){8}}
 \put(29,15){\vector(-1,0){8}}
 \put(19,15){\vector(-1,0){8}}
 \put(29,25){\vector(-1,0){8}}
 \put(30,6){\vector(0,1){8}}
 \put(30,16){\vector(0,1){8}}
 \put(30,26){\vector(0,1){8}}
 \put(20,6){\vector(0,1){8}}
 \put(20,16){\vector(0,1){8}}
 \put(10,6){\vector(0,1){8}}
 \put(0,1){\makebox(0,0)[cc]{$s_1$}}
 \put(10,1){\makebox(0,0)[cc]{$s_2$}}
 \put(20,1){\makebox(0,0)[cc]{$s_3$}}
 \put(30,1){\makebox(0,0)[cc]{$s_4$}}
 \put(-4,7){\makebox(0,0)[cc]{$t_1$}}
 \put(7,17){\makebox(0,0)[cc]{$t_2$}}
 \put(17,27){\makebox(0,0)[cc]{$t_3$}}
 \put(27,37){\makebox(0,0)[cc]{$t_4$}}
   \end{picture}}
     \end{picture}
  \end{center}
\caption{The grid $\Gamma_{5,4}$ (left) and the half-grid $\Gamma^\triangle_4$
(right).} \label{fig:Gamma}
  \end{figure}

In what follows, the collection of pairs $(I\subseteq[n],I'\subseteq[n'])$ with
equal sizes: $|I|=|I'|$, is denoted by $\Escr^{n,n'}$.  By an
$(I|I')$-\emph{flow} we mean a collection $\phi$ of \emph{pairwise (vertex)
disjoint directed paths} in $G$ going from the source set $S_I:=\{s_i\,\colon
i\in I\}$ to the sink set $T_{I'}:=\{t_j\,\colon j\in I'\}$. The set of
$(I|I')$-flows in $G$ is denoted by $\Phi_{I|I'}^G$, or simply by
$\Phi_{I|I'}$.

Let $w:V\to\mathfrak{R}$ be a \emph{weighting} on the vertices of $G$
(alternatively, one can consider a weighting on the edges; see the end of this
section). We associate to $w$ the function $f=f_w$ on $\Escr^{n,n'}$ defined by
  \begin{equation} \label{eq:alg_f}
  f(I|I'):=\sum\nolimits_{\phi\in\Phi_{I|I'}}\prod\nolimits_{v\in V_\phi} w(v),
  \qquad (I,I')\in\Escr^{n,n'},
  \end{equation}
where $V_\phi$ is the set of vertices occurring in a flow $\phi$. (When $G$ has
no flow for some $(I,I')$, we set $f(I|I'):=0$.) We refer to $f$ obtained in
this way as an \emph{algebraic flow-generated function}, or an
\emph{AFG-function} for short.

When an $(I|I')$-flow $\phi$ enters the first $|I|=:k$ sinks (i.e., $I'=[k]$),
we say that $\phi$ is a \emph{flag flow} for $I$. Accordingly, we use the
abbreviated notation $\Phi_I$ for $\Phi_{I|\,[k]}$, and $f_w(I)$ for
$f_w(I|\,[k])$. When we are interested in the flag case only, $f_w$ is regarded
as a function on the set $2^{[n]}$ of subsets of $[n]$.

Lindstr\"om~\cite{Li} showed that if $M$ is the $n'\times n$ matrix whose
entries $m_{ji}$ are defined as $\sum_{\phi\in\Phi_{\{i\}|\{j\}}}\prod_{v\in
V_\phi} w(v)$, then for any $(I,I')\in\Escr^{n,n'}$, the minor of $M$ with the
column set $I$ and the row set $I'$ is equal to the value $f(I|I')$ as
in~\refeq{alg_f}. A converse property is known to be valid for the totally
nonnegative matrices (see~\cite{Br}): the minors of such a matrix can be
expressed as above via flows for some planar network and weighting. (Recall
that a real matrix is called \emph{totally nonnegative} (\emph{totally
positive}) if all minors in it are nonnegative (resp., positive).) Moreover, we
show in the Appendix that a similar property holds for any matrix over a field.

Another important application of the flow model concerns tropical analogues of
the above quadratic relations. In this case the flow-generated function $f=f_w$
determined by a weighting $w$ on $V$ is defined by
   \begin{equation} \label{eq:trop_f}
  f(I| I'):=\max_{\phi\in\Phi_{I|I'}}\left(\sum\nolimits_{v\in V_\phi} w(v)\right),
  \qquad (I,I')\in\Escr^{n,n'}.
  \end{equation}
Here $w$ is assumed to take values in a totally ordered abelian group
$\mathfrak{L}$ (usually one deals with $\mathfrak{L}=\Rset$ or $\Zset$). The
expression for $f$ in~\refeq{trop_f} is the tropicalization of that
in~\refeq{alg_f}, and $f$ is said to be a \emph{tropical flow-generated
function}, or a \emph{TFG-function}. Some appealing properties of such
functions and related objects in the flag flow case are demonstrated
in~\cite{DKK1,DKK2} (where real-valued tropical functions are considered but
everywhere $\Rset$ can be replaced by $\mathfrak{L}$). In particular, a
TFG-function $f$ satisfies the tropical analog of~\refeq{AP3}, or the
\emph{TP3-relation}: for $i<j<k$ and $X\subseteq[n]-\{i,j,k\}$,
   \begin{equation} \label{eq:TP3}
    f(Xik)+f(Xj)=\max\{f(Xij)+f(Xk),f(Xjk)+f(Xi)\}.
    \end{equation}

In this paper we combine both cases, the algebraic and tropical ones, by
considering functions taking values in an arbitrary \emph{commutative semiring}
$\mathfrak{S}$, a set equipped with two associative and commutative binary
operations $\oplus$ (addition) and $\odot$ (multiplication) satisfying the
distributive law $a\odot(b\oplus c)=(a\odot b)\oplus(a\odot c)$. Sometimes we
will assume that $\mathfrak{S}$ contains neutral elements $\underline 0$ (for
addition) and/or $\underline 1$ (for multiplication). Two special cases are:
(i) a commutative ring (in which case $\underline 0\in\mathfrak{S}$ and each
element has an additive inverse); (ii) a commutative semiring with division (in
which case $\underline 1\in\mathfrak{S}$ and each element has a multiplicative
inverse). Examples of~(ii) are: the set $\Rset_{>0}$ of positive reals (with
$\oplus=+$ and $\odot=\cdot$), and the above-mentioned tropicalization of a
totally ordered abelian group $\mathfrak{L}$, denoted as $\mathfrak{L}^{\rm
trop}$ (with $\oplus=\max$ and $\odot=+$). The set $2\Zset_{>0}$ of positive
even integers (with usual addition and multiplication) gives an example of a
commutative semiring having neither $\underline 0$ nor $\underline 1$.

Extending~\refeq{alg_f} and~\refeq{trop_f}, the flow-generated function $f=f_w$
determined by a weighting $w:V\to\mathfrak{S}$ is defined by
  $$
  f(I| I'):=\bigoplus\nolimits_{\phi\in\Phi_{I|I'}} w(\phi),
  \qquad (I,I')\in\Escr^{n,n'},
  $$
where $w(\phi)$ denotes the weight $\odot(w(v)\,\colon v\in V_\phi)$ of a flow
$\phi$. We call $f$ an \emph{FG-function} (abbreviating ``flow-generated
function''), and denote the set of these functions by ${\bf
FG}_{n,n'}(\mathfrak{S})$. \medskip

\noindent\textbf{Remark 1.} Note that an $(I|I')$-flow in $G$ may not exist,
making $f(I|I')$ undefined if $\mathfrak{S}$ does not contain $\underline 0$
(e.g., in the tropical case). To overcome this trouble, we may formally extend
$\mathfrak{S}$, when needed, by adding an ``extra neutral'' element $\ast$,
setting $\ast\oplus a=a$ and $\ast\odot a=\ast$ for all $a\in \mathfrak{S}$. In
the extended semiring $\mathfrak{\hat S}$, one defines $f(I|I'):=\ast$ whenever
$\Phi_{I|I'}=\emptyset$.
\medskip

As before, we write $f(I)$ for $f(I|\,[|I|])$ in the flag flow case. Then a
direct analogue of the general Pl\"ucker relation~\refeq{a_pluck} for
$\mathfrak{S}$ is viewed as
  \begin{equation} \label{eq:S_pluck}
  \bigoplus_{A\in\Ascr} \left(f(X\cup A)\odot
      f(X\cup(Y-A))\right)
 =\bigoplus_{B\in\Bscr} \left(f(X\cup B)\odot
   f(X\cup(Y-B))\right).
  \end{equation}

\noindent \textbf{Definition 1.} Let $p<m\le n$ and let $\Ascr_0,\Bscr_0$ be
two families of $p$-element subsets of $[m]$. If~\refeq{S_pluck} (with $f$
determined as above) holds for any commutative semiring $\mathfrak{S}$, acyclic
directed graph $G$, weighting $w$, disjoint subsets $X$ and $Y$, and the
families $\Ascr:=\gamma_Y(\Ascr_0)$ and $\Bscr:=\gamma_Y(\Bscr_0)$, then we
call~\refeq{S_pluck} a \emph{stable quadratic relation of Pl\"ucker type}, or a
\emph{PSQ-relation} for short, and say that it is induced by the patterns
$\Ascr_0,\Bscr_0$. \medskip

Note that in general we admit that $\Ascr_0$ or $\Bscr_0$ can contain multiple
members. In other words, one may assume that for $m,p$ fixed, the pairs of
patterns inducing PSQ-relations constitute an abelian group under the
operations $(\Ascr_0,\Bscr_0)+(\Ascr'_0,\Bscr'_0):=(\Ascr_0\sqcup \Ascr'_0,
\Bscr_0\sqcup \Bscr'_0)$ and
$(\Ascr_0,\Bscr_0)-(\Ascr'_0,\Bscr'_0):=(\Ascr_0\sqcup \Bscr'_0, \Bscr_0\sqcup
\Ascr'_0)$, where $\sqcup$ denotes the disjoint set union). For this reason, we
will write $\Ascr_0\Subset \binom{[m]}{p}$ and $\Ascr\Subset \binom{Y}{p}$
(with symbol $\Subset$ rather than $\subseteq$), and similarly for $\Bscr_0$
and $\Bscr$.
\smallskip

Next, a Pl\"ucker relation on minors of a matrix deals with flag minors and is
{\em homogeneous}, in the sense that the pairs of minor sizes in all products
are the same. However, there are quadratic relations involving non-flag and
non-homogeneous minors. One relation of this sort is expressed by Dodgson's
condensation formula~\cite{Do}:
  \begin{equation} \label{eq:Dodg_orig}
  f(iX|\,i'X')\,f(Xk|X'k')=
  f(iXk|\,i'X'k')\,f(X|X')+f(iX|X'k')\,f(Xk|\,i'X'),
  \end{equation}
where $f(I|I')$ stands for the minor of a matrix with the column set $I$ and
the row set $I'$, ~$k-i=k'-i'>0$, ~$X$ is the interval $[i+1..k-1]$ (from $i+1$
to $k-1$) and $X'$ is the interval $[i'+1..k'-1]$.

This inspires a study of a larger class of quadratic relations on
flow-generated functions over commutative semirings. Now for $\mathfrak{S},G,w$
as before, we deal with the function $f=f_w$ on $\Escr^{n,n'}$, and consider
disjoint $X,Y\subseteq[n]$ and disjoint $X',Y'\subseteq[n']$. An identity of
our interest is of the form
  \begin{multline} \label{eq:gen_sq}
  \bigoplus\nolimits_{(A,A')\in\Ascr} \left(f(XA|X'A')\odot
    f(X\bar A | X'\bar A\,')\right)     \\
 =\bigoplus\nolimits_{(B,B')\in\Bscr} \left(f(XB |
  X'B')\odot
   f(X\bar B| X'\bar B\,')\right).
  \end{multline}
Here, to simplify notation, we write $KL$ for the union $K\cup L$ of disjoint
sets $K,L$, denote the complement $Y-C$ of $C\subseteq Y$ by $\bar C$, and the
complement $Y'-C'$ of $C'\subseteq Y'$ by $\bar C\,'$. The families
$\Ascr,\Bscr$ consist of certain pairs $(C\subseteq Y,C'\subseteq Y')$
(admitting multiplicities). As before, we are interested in ``universal''
relations, and for this reason, consider the patterns $\Ascr_0,\Bscr_0$ formed
by the pairs $(A_0\subseteq[m],B_0\subseteq[m'])$ such that
$\Ascr=\gamma_{Y,Y'}(\Ascr_0)$ and $\Bscr=\gamma_{Y,Y'}(\Bscr_0)$, where
$m:=|Y|$, ~$m':=|Y'|$, and $\gamma_{Y,Y'}$ is the bi-component order preserving
bijective map of $[m]\sqcup[m']$ to $Y\sqcup Y'$.
Observe that~\refeq{S_pluck} is a special case of~\refeq{gen_sq} with
$X'=\{1,2,\ldots,|X|+r\}$ and $Y'=\{|X|+r+1,\ldots,|X|+m-r\}$, where
$r:=\min\{p,m-p\}$.
\medskip

\noindent \textbf{Definition~2.} When~\refeq{gen_sq} holds for fixed
$\Ascr_0,\Bscr_0$ as above and any corresponding $\mathfrak{S},G,w,X,Y,X',Y'$
and the families $\Ascr:=\gamma_{Y,Y'}(\Ascr_0)$ and
$\Bscr:=\gamma_{Y,Y'}(\Bscr_0)$, we call~\refeq{gen_sq} a (general)
\emph{stable quadratic relation}, or an \emph{SQ-relation}, and say that it is
induced by the patterns $\Ascr_0,\Bscr_0$.\medskip

To distinguish between the general and Pl\"ucker cases, we will refer to
$\Ascr,\Bscr$ in Definition~2 as \emph{2-families}, and to $\Ascr_0,\Bscr_0$ as
\emph{2-patterns}, whereas analogous objects in Definition~1 will be called
\emph{1-families} and \emph{1-patterns}. \smallskip

The goal of this paper is to describe a relatively simple combinatorial method
of characterizing the patterns $\Ascr_0,\Bscr_0$ inducing SQ-relations (in
particular, PSQ-relations). In fact, our method generalizes a flow rearranging
approach used in~\cite{DKK1} for proving the TP3-relation for TFG-functions.
The method consists in reducing to a certain combinatorial problem, and as a
consequence, provides an ``efficient'' procedure to recognize whether or not a
pair $\Ascr,\Bscr$ of 2-families yields an SQ-relation.

The main result obtained on this way is roughly as follows. We associate to a
pair $(C\subseteq[m],C'\subseteq[m'])$ a certain set $\Mscr(C,C')$ of perfect
matchings on $[m]\sqcup[m']$. Given a pair $\Ascr_0,\Bscr_0$ of 2-patterns for
$m,m'$, define $\Mscr(\Ascr_0)$ to be the collection of such matchings over all
members of $\Ascr_0$ (counting multiplicities), and define $\Mscr(\Bscr_0)$ in
a similar way. We say that $\Ascr_0,\Bscr_0$ are \emph{balanced} if the
families $\Mscr(\Ascr_0)$ and $\Mscr(\Bscr_0)$ are equal, and show
(Theorem~\ref{tm:main}) that \smallskip

\emph{2-patterns $\Ascr_0,\Bscr_0$ induce an SQ-relation if and only if they
are balanced.}
 \smallskip

Our approach to handling flows and reducing the problem to examining certain
collections of matchings is close in essence to a lattice paths method
elaborated in Fulmek and Kleber~\cite{FK} and Fulmek~\cite{Flm} to generate
quadratic identities on Schur functions. The latter method is based on the
Gessel--Viennot interpretation~\cite{GV} of semistandard Young tableaux by use
of ``flows'' in a special directed graph, and~\cite{Flm,FK} give sufficient
conditions on quadratic identities for Schur functions, formulated just in
terms of relations on matchings. \smallskip

The paper is organized as follows. Section~\ref{sec:flow} describes properties
of certain pairs of flows, called \emph{double flows}, which lie in the
background of the method. Section~\ref{sec:balan} states the main result
(Theorem~\ref{tm:main}) and proves sufficiency, claiming that all balanced
families $\Ascr,\Bscr$ give SQ-relations. Section~\ref{sec:relat} is devoted to
illustrations of the method; it demonstrates a number of examples of
SQ-relations, including rather wide classes (a majority concerns the flag flow
case). Section~\ref{sec:(i)-(ii)} proves the other direction of
Theorem~\ref{tm:main}, which is more intricate. Moreover, we show that if
2-patterns $\Ascr_0,\Bscr_0$ are not balanced, then for any corresponding
$X,Y,X',Y'$, one can construct a planar network $G$ with integer weights $w$
such that the FG-function $f_w$ violates relation~\refeq{gen_sq}. As a
consequence, \emph{for $\Ascr,\Bscr$ fixed, validity of~\refeq{gen_sq} for all
commutative semirings $\frakS$ is equivalent to its validity for
$\frakS=\Zset$.} (This matches the so-called \emph{transfer principle} for
semirings; see, e.g., \cite[Sec.~3]{AGG}.) Section~\ref{sec:Schur} is devoted
to applications to Schur functions. Section~\ref{sec:laurent} contains a short
discussion on nice additional properties (namely, the existence and an explicit
construction of a so-called basis for $\bf{FG}_{n,n'}(\frakS)$, the Laurent
phenomenon for FG-functions, and some others) in the case when $\frakS$ is a
commutative semiring \emph{with division}; this extends corresponding results
from~\cite{DKK1}. This section is concluded with a (rather routine) proof of
the assertion that the function of minors of any $n'\times n$ matrix $A$ over a
commutative ring obeys all SQ-relations concerning $n,n'$
(Proposition~\ref{pr:matrix-SQ}). Finally, in the Appendix we consider an
arbitrary matrix $M$ over a field and explain how to construct a pair $(G,w)$
of which flows generate the function $f$ of minors of $M$ (thus showing that
$f$ is an FG-function).
 \medskip

We have mentioned above that, instead of a weighting on the vertices of a graph
$G$ in question, one can consider a weighting on the edges. However, this does
not affect the problem and our results in essence. When an edge $e$ is endowed
with a weight, one can split $e$ into two edges in series and transfer the
weight into the intermediate vertex, yielding an equivalent flow model (up to
assigning the weight to each old vertex to be the ``neutral element for
multiplication''). Throughout the paper (except for Section~\ref{sec:Schur}) we
prefer to deal with a weighting on vertices for technical reasons.


\section{\Large Flows and double flows}  \label{sec:flow}

As before, let $G=(V,E)$ be an (acyclic) planar network with sources
$s_1,\ldots,s_n$ and sinks $t_1,\ldots,t_{n'}$. In this section we describe
ideas and tools behind the method of constructing 2-patterns $\Ascr_0,\Bscr_0$
that ensure validity of~\refeq{gen_sq} for all flow-generated functions $f=f_w$
on $\Escr^{n,n'}$ determined by weightings $w:V\to \mathfrak{S}$, where
$\mathfrak{S}$ is an arbitrary commutative semiring.

First of all we specify some terminology and notation. By a \emph{path} in a
digraph (directed graph) we mean a sequence $P=(v_0,e_1,v_1,\ldots,e_k,v_k)$
such that each $e_i$ is an edge connecting vertices $v_{i-1},v_i$. An edge
$e_i$ is called \emph{forward} if it is directed from $v_{i-1}$ to $v_i$,
denoted as $e_i=(v_{i-1},v_i)$, and \emph{ backward} otherwise (when
$e_i=(v_i,v_{i-1})$). The path $P$ is called {\em directed} if it has no
backward edge, and {\em simple} if all vertices $v_i$ are distinct. When $k>0$,
$v_0=v_k$, and all $v_1,\ldots,v_k$ are distinct, ~$P$ is called a \emph{simple
cycle}, or a \emph{circuit}. The sets of vertices and edges of $P$ are denoted
by $V_P$ and $E_P$, respectively.

Consider an $(I|I')$-flow $\phi$ in $G$, where $(I,I')\in\Escr^{n,n'}$. It
consists of pairwise disjoint directed paths going from the source set $S_I$ to
the sink set $T_{I'}$. Since $G$ is acyclic, these paths are simple, and in
view of the ordering of sources and sinks in the boundary $O$, the path in
$\phi$ beginning at $i$th source in $S_I$ enters $i$th sink in $T_{I'}$
(counting ``from left to right''). Equivalently (when $s_1\ne t_1$ and $s_n\ne
t_{n'}$) we may think of $\phi$ as an induced subgraph of $G$ satisfying:
$\deltaout_\phi(s_i)=1$ and $\deltain_\phi(s_i)=0$ if $i\in I$;
$\deltaout_\phi(t_j)=0$ and $\deltain_\phi(t_j)=1$ if $j\in I'$; and
$\deltaout_\phi(v)=\deltain_\phi(v)\in\{0,1\}$ for the other vertices $v$ of
$G$. Here $\deltaout_\phi(v)$ (resp., $\deltain_\phi(v)$) denotes the number of
edges in $\phi$ leaving (resp., entering) a vertex $v$. Also we denote
$\deltaout_\phi(v)+\deltain_\phi(v)$ by $\delta_\phi(v)$.

Our approach is based on examining certain pairs of flows in $G$ and
rearranging them to form some other pairs. To simplify technical details, it is
convenient to modify the original network $G$ as follows. Let us split each
vertex $v\in V$ into two vertices $v',v''$ (placing them in a small
neighborhood of $v$ in the plane) and connect them by edge $e_v=(v',v'')$,
called a \emph{split-edge}. Each edge $(u,v)$ of $G$ is replaced by an edge
going from $u''$ to $v'$; we call it an \emph{ordinary} edge. Also for each
$s_i\in S$, we add a new source $\hat s_i$ and the edge $(\hat s_i,s'_i)$, and
for each $t_j\in T$, add a new sink $\hat t_j$ and the edge $(t''_j,\hat t_j)$;
we refer to such edges as \emph{extra} ones. The picture illustrates the
transformation for the half-grid $\Gamma^\triangle_3$.

 \begin{center}
 \unitlength=1.0mm
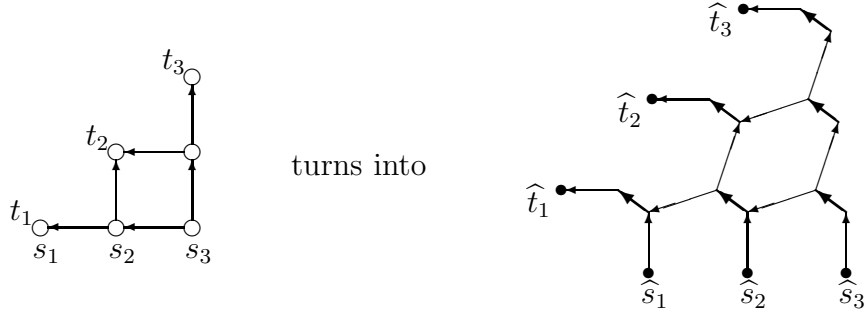
  \begin{picture}(120,40)
   \put(0,5){\begin{picture}(35,28)
 \put(10,5){\circle{2}}
 \put(20,5){\circle{2}}
 \put(30,5){\circle{2}}
 \put(20,15){\circle{2}}
 \put(30,15){\circle{2}}
 \put(30,25){\circle{2}}
 \put(19,5){\vector(-1,0){8}}
 \put(29,5){\vector(-1,0){8}}
 \put(29,15){\vector(-1,0){8}}
 \put(20,6){\vector(0,1){8}}
 \put(30,6){\vector(0,1){8}}
 \put(30,16){\vector(0,1){8}}
 \put(9,1){$s_1$}
 \put(19,1){$s_2$}
 \put(29,1){$s_3$}
 \put(6,6){$t_1$}
 \put(16,16){$t_2$}
 \put(26,26){$t_3$}
  \end{picture}}
 \put(43,17){turns into}
   \put(70,4){\begin{picture}(50,32)
 \put(20,0){\circle*{1.5}}
 \put(33,0){\circle*{1.5}}
 \put(46,0){\circle*{1.5}}
 \put(8.5,11){\circle*{1.5}}
 \put(20.5,23){\circle*{1.5}}
 \put(32.5,35){\circle*{1.5}}
 \put(20,0){\vector(0,1){8}}
 \put(33,0){\vector(0,1){8}}
 \put(46,0){\vector(0,1){8}}
 \put(16,11){\vector(-1,0){7}}
 \put(28,23){\vector(-1,0){7}}
 \put(40,35){\vector(-1,0){7}}
 \put(29,11){\vector(-3,-1){9}}
 \put(42,11){\vector(-3,-1){9}}
 \put(41,23){\vector(-3,-1){9}}
 \put(29,11){\vector(1,3){3}}
 \put(42,11){\vector(1,3){3}}
 \put(41,23){\vector(1,3){3}}
{\thicklines
 \put(20,8){\vector(-4,3){4}}
 \put(33,8){\vector(-4,3){4}}
 \put(46,8){\vector(-4,3){4}}
 \put(32,20){\vector(-4,3){4}}
 \put(45,20){\vector(-4,3){4}}
 \put(44,32){\vector(-4,3){4}}
 }
 \put(19,-4){$\hat s_1$}
 \put(32,-4){$\hat s_2$}
 \put(45,-4){$\hat s_3$}
 \put(4,8){$\hat t_1$}
 \put(16,20){$\hat t_2$}
 \put(28,32){$\hat t_3$}
  \end{picture}}
 \end{picture}
  \end{center}

Note that the new (modified) graph is again acyclic, but it need not be planar in
general (e.g., a local non-planarity arises when the original graph has a vertex
$v$ with four incident edges $e_1,e_2,e_3,e_4$, in this order clockwise, such
that $e_1,e_3$ enter and $e_2,e_4$ leave $v$); nevertheless, the latter fact will
cause no trouble to us. We denote this graph by $\hat G=(\hat V,\hat E)$, and
take $\hat S:=\{\hat s_1,\ldots,\hat s_n\}$ and $\hat T:=\{\hat t_1,\ldots,\hat
t_{n'}\}$ as the sets of sources and sinks in it, respectively. As before,
sources and sinks are also called \emph{terminals}. Clearly for any $i\in[n]$ and
$j\in[n']$, there is a natural 1--1 correspondence between the directed paths
from $s_i$ to $t_j$ in $G$ and the ones from $\hat s_i$ to $\hat t_j$ in $\hat
G$. This is extended to a 1--1 correspondence between flows, and for $(I,I')\in
\Escr^{n,n'}$, we keep notation $\Phi_{I|I'}$ for the set of flows in $\hat G$
going from $\hat S_I:=\{\hat s_i \,\colon i\in I\}$ to $\hat T_{I'}:=\{\hat
t_j\,\colon j\in I'\}$. (When needed, a weighting $w$ on the vertices $v$ of the
initial $G$ is transferred to the split-edges of ${\hat G}$, namely, by setting
$w(e_v):=w(v)$. Then corresponding flows in both networks have equal weights,
which are the $\odot$-products of the weights of vertices or split-edges in the
flows. This implies that the functions on $\Escr^{n,n'}$ generated by
corresponding flows coincide.)

The digraph ${\hat G}$ possesses the following useful properties:
  \begin{numitem1}
(a) Each non-terminal vertex is incident with exactly one split-edge, and if
$e=(u,v)$ is a split-edge, then  $\deltaout_{\hat G}(u)=1$ and $\deltain_{\hat
G}(v)=1$; (b) Each source (sink) has exactly one leaving edge and no entering
edge (resp., one entering edge and no leaving edge).
  \label{eq:1edge}
  \end{numitem1}

Consider disjoint subsets $X,Y\subseteq[n]$ and disjoint subsets
$X',Y'\subseteq[n']$. Let $m:=|Y|$ and $m':=|Y|$. Consider a pair $(A\subseteq
Y,A'\subseteq Y')$ satisfying
  $$
  |X|+|A|=|X'|+|A'| \quad \mbox{and} \quad |X|+|\bar A|=|X'|+|\bar A\,'|,
  $$
as before denoting by $\bar A$ and $\bar A\,'$ the sets $Y-A$ and $Y'-A'$,
respectively.
\medskip

\noindent\textbf{Remark 2.} The above equalities are necessary for the
existence of an $(XA|X'A')$-flow and an $(X\bar A|X'\bar A\,')$-flow (as
before, we write $XA$ for $X\cup A$, and so on). They imply
  \begin{equation} \label{eq:proper}
\mbox{(i)}\;\;  2|X|+|Y|=2|X'|+|Y'|\quad \mbox{and}\quad
  \mbox{(ii)}\;\; |Y|-|Y'|=2(|A|-|A'|).
  \end{equation}

We say that $X,Y,X',Y'$ satisfying~(i) are \emph{consistent} and refer to a
pair $(A\subseteq Y,A'\subseteq Y')$ satisfying~(ii) as being \emph{proper} for
$(Y,Y')$. The set of proper pairs for $(Y,Y)$ is denoted by $\Pi_{Y,Y'}$. For
brevity we write $\Pi_{m,m'}$ for $\Pi_{[m],[m']}$.

Consider an $(XA|X'A')$-flow $\phi$ and a $(X\bar A|X\bar A\,')$-flow $\phi'$
in ${\hat G}$; we call the pair $(\phi,\phi')$ a \emph{double flow} for
$(A,A')$. Our method will rely on two lemmas. Hereinafter we write
$C\triangle\, D$ for the symmetric difference $(C-D)\cup(D-C)$ of sets $C,D$.
  \begin{lemma} \label{lm:sumFFp}
~$E_\phi\triangle\, E_{\phi'}$ is partitioned into the edge sets of pairwise
disjoint circuits $C_1,\ldots,C_d$ (for some $d$) and simple paths $P_1,\ldots,
P_p$ (with $p=\frac12(m+m')$), where each $P_i$ connects either $\hat S_{A}$
and $\hat S_{\bar A}$, or $\hat S_{A}$ and $\hat T_{A'}$, or $\hat S_{\bar A}$
and $\hat T_{\bar A\,'}$, or $\hat T_{A'}$ and $\hat T_{\bar A\,'}$. In each of
these circuits and paths, the edges of $\phi$ and the edges of $\phi'$ have
opposed directions (say, the former edges are forward and the latter ones are
backward).
  \end{lemma}
  \begin{proof}
~Observe that a vertex $v$ of ${\hat G}$ satisfies: (i) $\delta_{\phi}(v)=1$
and $\delta_{\phi'}(v)=0$ if $v\in\hat S_{A}\cup \hat T_{A'}$; (ii)
$\delta_{\phi}(v)=0$ and $\delta_{\phi'}(v)=1$ if $v\in\hat S_{\bar A}\cup \hat
T_{\bar A\,'}$; (iii) $\delta_{\phi}(v)=\delta_{\phi'}(v)=1$ if $v\in \hat
S_X\cup \hat T_{X'}$; and (iv) $\delta_{\phi}(v),\delta_{\phi'}(v)\in\{0,2\}$
otherwise. This together with~\refeq{1edge} implies that any vertex $v$ is
incident with 0, 1 or 2 edges in $E_\phi\triangle\, E_{\phi'}$, and the number
of such edges is equal to 1 if and only if $v\in \hat S_{A}\cup\hat S_{\bar A}
\cup \hat T_{A'}\cup \hat T_{\bar A\,'}$. (This is where we essentially use the
transformation of $G$ into $\hat G$.) Hence the weakly connected components of
the subgraph of ${\hat G}$ induced by $E_\phi\triangle\, E_{\phi'}$ are
circuits, say, $C_1,\ldots,C_d$, and simple paths $P_1,\ldots,P_p$, each of the
latter connecting two terminals in $\hat S_{A}\cup\hat S_{\bar A} \cup \hat
T_{A'}\cup \hat T_{\bar A\,'}$.

Consider consecutive edges $e,e'$ in a circuit $C_i$ or a path $P_j$. If both
$e,e'$ belong to the same flow among $\phi,\phi'$, then, obviously, they have the
same direction in this circuit/path. Suppose $e,e'$ belong to different flows. In
view of~\refeq{1edge}, the common vertex $v$ of $e,e'$ is not a terminal and it
is incident with a split-edge $e''$. Clearly $e''$ belongs to both $\phi,\phi'$,
and therefore $e''\ne e,e'$. Then either both $e,e'$ enter $v$ or both leave $v$,
so they are directed differently along the circuit/path containing them. This
yields the second assertion in the lemma.

Finally, consider a path $P_j=(v_0,e_1,v_1,\ldots,e_r,v_r)$ as above, and
suppose that some of its ends, say, $v_0$, belongs to $\hat S_{A}$. Then the
extra edge $e_1$ is contained in $\phi$ and leaves the source $v_0$. If $v_r\in
\hat S_A$, then the extra edge $e_r$ is in $\phi$ as well and leaves the source
$v_r$; so $e_1,e_r$ are directed differently along $P_j$, contradicting the
argument above. And if $v_r\in \hat T_{\bar A\,'}$, then $e_r$ belongs to
$\phi'$ and enters the sink $v_r$; so $e_1,e_r$ have the same direction along
$P_j$, again obtaining a contradiction. Thus, $P_j$ connects $\hat S_A$ and
$\hat S_{\bar A}\cup \hat T_{A'}$. Similarly, any path $P_j$ neither has both
ends in exactly one of $\hat S_{\bar A},\hat T_{A'},\hat T_{\bar A\,'}$, nor
connects $\hat S_{\bar A}$ and $\hat T_{A'}$.
  \end{proof}

Figure~\ref{fig:double} illustrates an example of ${\hat G},\phi,\phi'$ and
indicates $E_\phi\sqcup E_{\phi'}$ and $E_\phi\triangle\, E_{\phi'}$.

\begin{figure}[htb]

 \begin{center}
 \unitlength=1.0mm
  \begin{picture}(145,53)
   \put(0,10){\begin{picture}(20,43)  
 \put(0,0){\circle{2}}
 \put(8,0){\circle{2}}
 \put(16,0){\circle{2}}
 \put(0,38){\circle{2}}
 \put(12,38){\circle{2}}
 \put(18,38){\circle{2}}
 \put(0,1){\vector(0,1){28.5}}
 \put(8,0){\vector(2,3){4}}
 \put(16,0){\vector(-2,3){4}}
 \put(12,6){\vector(0,1){6}}
 \put(12,12){\vector(1,1){6}}
 \put(12,12){\vector(-1,1){6}}
 \put(6,18){\vector(1,1){6}}
 \put(18,18){\vector(-1,1){6}}
 \put(12,24){\vector(0,1){6}}
 \put(12,30){\vector(-1,0){12}}
 \put(12,30){\vector(0,1){7}}
 \put(0,30){\vector(0,1){7}}
 \put(0,-5){$\hat s_1$}
 \put(8,-5){$\hat s_2$}
 \put(16,-5){$\hat s_3$}
 \put(0,40){$\hat t_1$}
 \put(9,40){$\hat t_2$}
 \put(18,40){$\hat t_3$}
  \end{picture}}
 \put(7,-2){(a)}
   \put(32,10){\begin{picture}(20,43)  
 \put(0,0){\circle{2}}
 \put(8,0){\circle{2}}
 \put(16,0){\circle{2}}
 \put(0,38){\circle{2}}
 \put(12,38){\circle{2}}
 \put(18,38){\circle{2}}
 \put(0,1){\vector(0,1){28.5}}
 \put(16,0){\vector(-2,3){4}}
 \put(12,6){\vector(0,1){6}}
 \put(12,12){\vector(1,1){6}}
 \put(18,18){\vector(-1,1){6}}
 \put(12,24){\vector(0,1){6}}
 \put(12,30){\vector(0,1){7}}
 \put(0,30){\vector(0,1){7}}
  \end{picture}}
 \put(39,-2){(b)}
   \put(64,10){\begin{picture}(20,43)  
 \put(0,0){\circle{2}}
 \put(8,0){\circle{2}}
 \put(16,0){\circle{2}}
 \put(0,38){\circle{2}}
 \put(12,38){\circle{2}}
 \put(18,38){\circle{2}}
 \put(8,0){\vector(2,3){4}}
 \put(12,6){\vector(0,1){6}}
 \put(12,12){\vector(-1,1){6}}
 \put(6,18){\vector(1,1){6}}
 \put(12,24){\vector(0,1){6}}
 \put(12,30){\vector(-1,0){12}}
 \put(0,30){\vector(0,1){7}}
  \end{picture}}
 \put(69,-2){(c)}
   \put(96,10){\begin{picture}(20,43)  
 \put(0,0){\circle{2}}
 \put(8,0){\circle{2}}
 \put(16,0){\circle{2}}
 \put(0,38){\circle{2}}
 \put(12,38){\circle{2}}
 \put(18,38){\circle{2}}
 \put(0,1){\vector(0,1){28.5}}
 \put(8,0){\vector(2,3){4}}
 \put(16,0){\vector(-2,3){4}}
 \put(11.5,6){\vector(0,1){6}}
 \put(12.5,6){\vector(0,1){6}}
 \put(12,12){\vector(1,1){6}}
 \put(12,12){\vector(-1,1){6}}
 \put(6,18){\vector(1,1){6}}
 \put(18,18){\vector(-1,1){6}}
 \put(11.5,24){\vector(0,1){6}}
 \put(12.5,24){\vector(0,1){6}}
 \put(12,30){\vector(-1,0){12}}
 \put(12,30){\vector(0,1){7}}
 \put(-0.5,30){\vector(0,1){7}}
 \put(0.5,30){\vector(0,1){7}}
  \end{picture}}
 \put(103,-2){(d)}
   \put(128,10){\begin{picture}(20,43)  
 \put(0,0){\circle{2}}
 \put(8,0){\circle{2}}
 \put(16,0){\circle{2}}
 \put(0,38){\circle{2}}
 \put(12,38){\circle{2}}
 \put(18,38){\circle{2}}
 \put(0,1){\vector(0,1){28.5}}
 \put(8,0){\vector(2,3){4}}
 \put(16,0){\vector(-2,3){4}}
 \put(12,12){\vector(1,1){6}}
 \put(12,12){\vector(-1,1){6}}
 \put(6,18){\vector(1,1){6}}
 \put(18,18){\vector(-1,1){6}}
 \put(12,30){\vector(-1,0){12}}
 \put(12,30){\vector(0,1){7}}
  \put(1,7){$P'$}
  \put(15,3){$P$}
 \end{picture}}
 \put(135,-2){(e)}
 \end{picture}
  \end{center}
\caption{(a) ${\hat G}$; \;\;(b) $\phi$; \;\;(c) $\phi'$; \;\;(d) $E_\phi\sqcup
E_{\phi'}$; \;\;(e) $E_\phi\triangle\, E_{\phi'}$} \label{fig:double}
  \end{figure}

Next we explain how to rearrange a double flow $(\phi,\phi')$ for $(A,A')$ so
as to obtain a double flow for another pair $(B,B')\in\Pi_{Y,Y'}$. Let
$P_1,\ldots,P_p$ be the paths as in Lemma~\ref{lm:sumFFp}, where
$p=\frac12(m+m')$. We denote the set of these paths by $\Pscr(\phi,\phi')$. For
a path $P\in\Pscr(\phi,\phi')$, let $\pi(P)$ denote the pair of elements in
$Y\sqcup Y'$ corresponding to the end vertices of $P$. We observe from
Lemma~\ref{lm:sumFFp} that $\pi(P)$ belongs to one of $A\times \bar A$,
$A\times A'$, $A'\times \bar A\,'$, $\bar A\times \bar A\,'$ (considering
$\pi(P)$ up to reversing). Define
  $$
  M(\phi,\phi'):=\{\pi(P)\,\colon P\in\Pscr(\phi,\phi')\}.
 $$
This set of pairs forms a \emph{perfect matching} on $Y\sqcup Y'$ (i.e., each
element of the latter set is contained in exactly one pair).

   %
   \begin{lemma} \label{lm:path_switch}
Choose an arbitrary subset $M_0\subseteq M(\phi,\phi')$. Define $Z:=\cup(\pi\in
M_0)\cap Y$, ~$Z':=\cup(\pi\in M_0)\cap Y'$, ~$B:=A\triangle\, Z$, and
$B':=A'\triangle\, Z'$. Define $U:=\cup(E_{P}\,\colon
P\in\Pscr(\phi,\phi'),~\pi(P)\in M_0)$. Then there are a unique
$(XB|X'B')$-flow $\psi$ and a unique $(X\bar B|X'\bar B\,')$-flow $\psi'$ such
that $E_\psi=E_\phi\triangle\,U$ and $E_{\psi'}=E_{\phi'}\triangle\, U$. In
particular, $E_\psi\sqcup E_{\psi'}=E_\phi\sqcup E_{\phi'}$.
  \end{lemma}
  \begin{proof}
~By Lemma~\ref{lm:sumFFp}, each path $P\in \Pscr(\phi,\phi')$ is a
concatenation of directed paths $Q_1,\ldots,Q_r$ (considered up to reversing),
where consecutive $Q_j,Q_{j+1}$ are contained in different flows among
$\phi,\phi'$ and either both leave or both enter their common vertex.
Therefore, exchanging the pieces $Q_j$ in $\phi,\phi'$, we obtain an
$(XC|X'C')$-flow $\alpha$ and an $(X\bar C|X'\bar C\,')$-flow $\alpha'$ such
that $E_{\alpha}=E_\phi\triangle\,E_{P}$ and $E_{\alpha'}=E_{\phi'}\triangle\,
E_{P}$, where $C:=A\triangle\,(\pi\cap Y)$ and $C':=A'\triangle\,(\pi\cap Y')$.

Doing so for all $P\in \Pscr(\phi,\phi')$ with $\pi(P)\in M_0$, we obtain flows
$\psi,\psi'$ satisfying the desired properties, taking into account that the
paths in $\Pscr(\phi,\phi')$ are pairwise disjoint. The uniqueness of
$\psi,\psi'$ is obvious.
  \end{proof}

Note that $M(\psi,\psi')=M(\phi,\phi')$ and $\Pscr(\psi,\psi')=
\Pscr(\phi,\phi')$, and the transformation of $\psi,\psi'$ by use of the paths
in $\Pscr(\psi,\psi')$ related to $M_0$ returns the flows $\phi,\phi'$.

Figure~\ref{fig:exchange} illustrates flows $\psi,\psi'$ created from
$\phi,\phi'$ in Fig.~\ref{fig:double}. Here the left fragment shows
$\psi,\psi'$ when the exchange is performed with respect to the (single) path
$P$ in $\Pscr(\phi,\phi')$ connecting the sources $\hat s_2$ and $\hat s_3$,
and the right fragment shows those for the path $P'$ connecting the source
$\hat s_1$ and the sink $\hat t_2$ (see~(e) in Fig.~\ref{fig:double}).

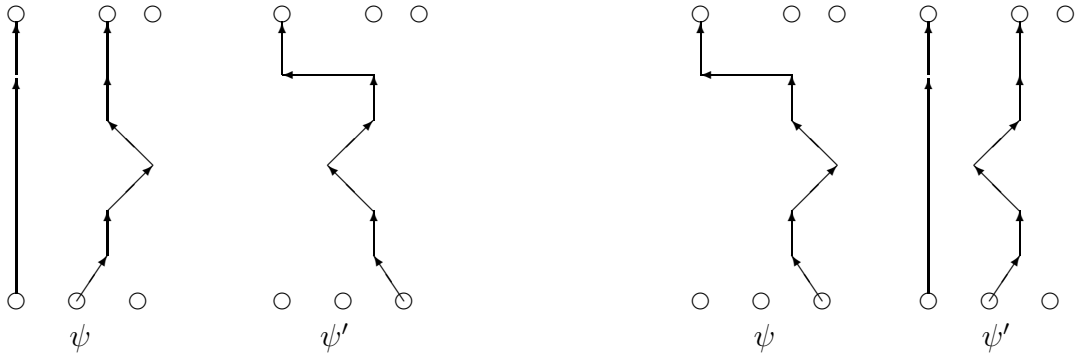
\begin{figure}[htb]

 \begin{center}
 \unitlength=1.0mm
  \begin{picture}(140,45)
   \put(0,3){\begin{picture}(20,40)  
 \put(0,0){\circle{2}}
 \put(8,0){\circle{2}}
 \put(16,0){\circle{2}}
 \put(0,38){\circle{2}}
 \put(12,38){\circle{2}}
 \put(18,38){\circle{2}}
 \put(0,1){\vector(0,1){28.5}}
 \put(8,0){\vector(2,3){4}}
 \put(12,6){\vector(0,1){6}}
 \put(12,12){\vector(1,1){6}}
 \put(18,18){\vector(-1,1){6}}
 \put(12,24){\vector(0,1){6}}
 \put(12,30){\vector(0,1){7}}
 \put(0,30){\vector(0,1){7}}
  \end{picture}}
 \put(7,-3){$\psi$}
   \put(35,3){\begin{picture}(20,40)  
 \put(0,0){\circle{2}}
 \put(8,0){\circle{2}}
 \put(16,0){\circle{2}}
 \put(0,38){\circle{2}}
 \put(12,38){\circle{2}}
 \put(18,38){\circle{2}}
 \put(16,0){\vector(-2,3){4}}
 \put(12,6){\vector(0,1){6}}
 \put(12,12){\vector(-1,1){6}}
 \put(6,18){\vector(1,1){6}}
 \put(12,24){\vector(0,1){6}}
 \put(12,30){\vector(-1,0){12}}
 \put(0,30){\vector(0,1){7}}
  \end{picture}}
 \put(40,-3){$\psi'$}
   \put(90,3){\begin{picture}(20,40)  
 \put(0,0){\circle{2}}
 \put(8,0){\circle{2}}
 \put(16,0){\circle{2}}
 \put(0,38){\circle{2}}
 \put(12,38){\circle{2}}
 \put(18,38){\circle{2}}
 \put(16,0){\vector(-2,3){4}}
 \put(12,6){\vector(0,1){6}}
 \put(12,12){\vector(1,1){6}}
 \put(18,18){\vector(-1,1){6}}
 \put(12,24){\vector(0,1){6}}
 \put(0,30){\vector(0,1){7}}
 \put(12,30){\vector(-1,0){12}}
  \end{picture}}
 \put(97,-3){$\psi$}
   \put(120,3){\begin{picture}(20,40)  
 \put(0,0){\circle{2}}
 \put(8,0){\circle{2}}
 \put(16,0){\circle{2}}
 \put(0,38){\circle{2}}
 \put(12,38){\circle{2}}
 \put(18,38){\circle{2}}
 \put(0,1){\vector(0,1){28.5}}
 \put(12,30){\vector(0,1){7}}
 \put(8,0){\vector(2,3){4}}
 \put(12,6){\vector(0,1){6}}
 \put(12,12){\vector(-1,1){6}}
 \put(6,18){\vector(1,1){6}}
 \put(12,24){\vector(0,1){6}}
 \put(0,30){\vector(0,1){7}}
  \end{picture}}
 \put(127,-3){$\psi'$}
 \end{picture}
  \end{center}
\caption{Creating $\psi,\psi'$ from $\phi,\phi'$ in Fig.~\ref{fig:double}: by
use of $P$ (left); by use of $P'$ (right).}
 \label{fig:exchange}
  \end{figure}

In the next section we will use the fact that, although the modified graph
${\hat G}$ may not be planar, its subgraph $\phi\cup\phi'$ is planar.

To see this, consider a non-terminal vertex $v$ in the initial graph $G$ that
belongs to both flows $\phi,\phi'$. Let $a,a'$ be the edges of $\phi$
(concerning $G$) entering and leaving $v$, respectively, and let $b,b'$ be
similar edges for $\phi'$. The only situation when the subgraph $\phi\cup\phi'$
(concerning $\hat G$) is not locally planar in a small neighborhood of the
split-edge $e_v$ is that all $a,a',b,b'$ are different and follow in this order
(clockwise or counterclockwise) around $v$. We assert that this is not the
case. Indeed, $a,a'$ belong to a directed path $P$ in ${\hat G}$ from a source
$\hat s_i$ to a sink $\hat t_{i'}$, and $b,b'$ belong to a directed path $Q$
from $\hat s_j$ to $\hat t_{j'}$. From the facts that the initial graph $G$ is
planar and acyclic and that the edges $a,a',b,b'$ occur in this order around
$v$ one can conclude that the paths $P,Q$ can meet only at $v$. This implies
that the terminals $s_i,t_{i'},s_j,t_{j'}$ are different and follow in this
cyclic order in the boundary $O$; a contradiction. Thus, $\phi\cup\phi'$ is
planar, as required.
\medskip

\noindent\textbf{Remark 3.} ~In the definition of FG-functions one can
equivalently consider only the acyclic digraphs $G$ having the additional
property that each edge of $G$ belongs to at least one directed path going from
a source to a sink. Arguing as above, one easily shows that for any vertex $v$
of such a $G$, the edge direction (to $v$ or from $v$) changes at most twice
when we move around $v$. Then the modified graph $\hat G$ is automatically
planar, and so is $\phi\cup\phi'$.
\medskip

  \Xcomment{
\noindent\textbf{Remark 4.} ~For a double flow $(\phi,\phi')$ for $(A,A')$, let
$C_1,\ldots,C_d$ be the circuits in $E_\phi\triangle\,E_{\phi'}$ (cf.
Lemma~\ref{lm:sumFFp}) and let $\xi(\phi,\phi')$ denote the 0,1,2-function on
$\hat E$ whose value on an edge $e$ is equal to the number of occurrences of
$e$ in $\phi,\phi'$. For an arbitrary subset $K\subseteq[d]$, replace $E_\phi$
by $E_\phi\triangle\,\Cscr$, and $E_{\phi'}$ by $E_{\phi'}\triangle\,\Cscr$,
where $\Cscr:=\cup(E_{C_i}\,\colon i\in K)$. This gives a double flow
$(\alpha,\alpha')$ for $(A,A')$ such that
$\xi(\alpha,\alpha')=\xi(\phi,\phi')$. As a result, there are exactly $2^d$
double flows for $(A,A')$ having the same ``characteristic function'' $\xi$.
  }


\section{\Large Balanced families}  \label{sec:balan}

In this section we use the above observations and results to construct families
involved in stable quadratic relations.

Consider the same objects as before: consistent sets $X,Y,X',Y'$ and a proper
pair $(A,A')$ for $(Y,Y')$ (obeying~\refeq{proper}), a double flow
$(\phi,\phi')$ for $(A,A')$, and the perfect matching $M=M(\phi,\phi')$ on
$Y\sqcup Y'$, referring to the members of $M$ as \emph{couples}. We denote the
set of double flows for $(A,A')$ by $\Dscr(A,A')$ (when $X,Y,X',Y'$ are fixed).

We associate to $(A,A')$ the set $\Mscr(A,A')$ (or $\Mscr_{Y,Y'}(A,A)$) of
feasible perfect matchings $M$ on $Y\sqcup Y'$ defined as follows. Let us think
that the elements of $Y$ and $Y'$ are placed, respectively, on the lower half
and on the upper half of a circumference $O$, in the increasing order from left
to right. Also let us call the elements (points) of $A\sqcup A'$ \emph{white},
and the elements of $\bar A\sqcup\bar A\,'$ \emph{black}. Then a perfect
matching $M$ on $Y\sqcup Y'$ is called \emph{feasible} for $(A,A')$ if:
  \begin{numitem1}
  \begin{itemize}
\item[(i)] When both elements of a couple $\pi\in M$ lie either in the lower half of $O$ or in
the upper half of $O$, these elements have different colors;
\item[(ii)] When the elements of $\pi$ lie in different halves, these elements have the same color;
\item[(iii)] $M$ is planar, in the sense that the chords of $O$ connecting the couples in $M$
are pairwise not intersecting.
 \end{itemize}
 \label{eq:match1}
 \end{numitem1}

Observe that for $(\phi,\phi')\in\Dscr(A,A')$, the matching $M=M(\phi,\phi')$
is feasible. This follows from Lemma~\ref{lm:sumFFp}, taking into account that
the subgraph $\phi\cup\phi'$ of $\hat G$ is planar and that the paths in
$\Pscr(\phi,\phi')$ are pairwise disjoint. A priori any matching in
$\Mscr(A,A')$ may be expressed as $M(\phi,\phi')$ for some
$(\phi,\phi')\in\Dscr(A,A')$.

We refer to a triple $(A,A',M)$ with $(A,A')\in\Pi_{Y,Y'}$ and
$M\in\Mscr(A,A')$ as a \emph{configuration}. For a 2-family
$\Ascr\Subset\Pi_{Y,Y'}$, we define $\Kscr(\Ascr)$ to be the family of all
configurations $(A,A',M)$ (with possible multiplicities) arising when $(A,A')$
runs over $\Ascr$.

The \emph{exchange operation} applied to a configuration $(A,A',M)$ and to a
chosen subset $ M_0\subseteq M$ makes the pair $(B,B')$ defined by
$B:=A\triangle\,(\cup(\pi\in M_0)\cap Y)$ and $B':=A\triangle\,(\cup(\pi\in
M_0)\cap Y')$; in other words, we change the colors of both elements in each
couple $\pi\in M_0$ (cf. Lemma~\ref{lm:path_switch}). Then $M$ becomes a
feasible matching for $(B,B')$, and the exchange operation applied to the
configuration $(B,B',M)$ and the same $M_0$ returns $(A,A')$.
\medskip

 \noindent\textbf{Definition.}
We say that two 2-families $\Ascr,\Bscr\Subset \Pi_{Y,Y'}$ are \emph{balanced}
if there exists a bijection between $\Kscr(\Ascr)$ and $\Kscr(\Bscr)$ such that
the corresponding configurations $(A,A',M)$ and $(B,B',M')$ have the same
matching: $M=M'$. (We rely on the simple fact that for any two configurations
$(A,A',M)$ and $(B,B',M)$, the pair $(B,B')$ can be obtained from $(A,A')$ by
the exchange operation w.r.t. some $ M_0\subseteq M$.) Equivalently,
$\Ascr,\Bscr$ are balanced if for each planar perfect matching $M$ on $Y\sqcup
Y'$, the number of times $M$ occurs in sets $\Mscr_{Y,Y'}(A,A')$ among
$(A,A')\in\Ascr$ is equal to a similar number for sets $\Mscr_{Y,Y'}(B,B')$
among $(B,B')\in\Bscr$. This can be written as
  $$
  \Mscr(\Ascr)=\Mscr(\Bscr),
  $$
where for $\Cscr\Subset\Pi_{Y,Y'}$, ~$\Mscr(\Cscr)$ denotes the family
consisting of matchings $M$ taken with the multiplicities equal to the number
of $(C,C')\in\Cscr$ such that $M\in\Mscr(C,C')$. Clearly $\Ascr,\Bscr$ are
balanced if and only if their 2-patterns $\Ascr_0,\Bscr_0$ are
balanced.\medskip

Our main result is the following
  \begin{theorem} \label{tm:main}
~Let $\Ascr_0,\Bscr_0\Subset\Pi_{m,m'}$. The following statements are
equivalent:

{\rm (i)} ~\refeq{gen_sq} is a stable quadratic relation, where
$\Ascr=\gamma_{Y,Y'}(\Ascr_0)$ and $\Bscr=\gamma_{Y,Y'}(\Bscr_0)$;

{\rm (ii)} ~$\Ascr_0,\Bscr_0$ are balanced.
  \end{theorem}

Part (i)$\Rightarrow$(ii) of this theorem will be shown in
Section~\ref{sec:(i)-(ii)}. In its turn, part (ii)$\Rightarrow$(i) can be
immediately proved by relying on the lemmas from the previous section.
  \begin{prop} \label{pr:balance-sq}
Let $\Ascr_0,\Bscr_0\Subset\Pi_{m,m'}$ be balanced. Then
identity~\refeq{gen_sq} holds for any consistent sets $X,Y,X',Y'$ (concerning
$n,n',m,m'$ as above; cf.~\refeq{proper}(i)), the families
$\Ascr=\gamma_{Y,Y'}(\Ascr_0)$ and $\Bscr=\gamma_{Y,Y'}(\Bscr_0)$, and any
FG-function $f$ on $\Escr^{n,n'}$ (concerning arbitrary $G,w,\mathfrak{S}$ as
above).
  \end{prop}
  \begin{proof}
~For corresponding $G,w,\mathfrak{S},X,Y,X',Y'$, consider the FG-function
$f=f_w$ on $\Escr^{n,n'}$. The summand concerning $(A,A')\in\Ascr$ in the
L.H.S. of~\refeq{gen_sq} can be expressed via double flows as follows:
  \begin{multline} \label{eq:ff-zeta}
  f(XA|X'A')\odot f(X\bar A|X'\bar A\,') \\
 =\left(\bigoplus\nolimits_{\phi\in\Phi_{XA|X'A'}} w(\phi)\right)
 \odot \left(\bigoplus\nolimits_{\phi'\in\Phi_{X\bar A|X'\bar A\,'}} w(\phi')\right)
    \qquad\qquad\qquad\qquad \\
  =\bigoplus\nolimits_{(\phi,\phi')\in\Dscr(A,A')} w(\phi)\odot w(\phi')\qquad\qquad\qquad\qquad\\
  =\bigoplus\nolimits_{M\in\Mscr(A,A')}
  \bigoplus\nolimits_{(\phi,\phi')\in\Dscr(A,A')\,\colon M(\phi,\phi')=M} w(\phi)\odot w(\phi').
  \end{multline}
The summand concerning $(B,B')\in\Bscr$ in the L.H.S. of~\refeq{gen_sq} is
expressed similarly.

Consider a configuration $(A,A',M)\in\Kscr(\Ascr)$ and suppose $(\phi,\phi')$
is a double flow for $(A,A')$ such that $M(\phi,\phi')=M$ (if it exists). Since
$\Ascr,\Bscr$ are balanced, $(A,A',M)$ is bijective to some configuration
$(B,B',M)$ in $\Bscr$. Since $M$ is a feasible matching for both $(A,A')$ and
$(B,B')$, one can see from conditions~\refeq{match1}(i),(ii) that $(B,B')$ is
obtained from $(A,A')$ by the exchange operation w.r.t. some $M_0\subseteq M$.
Then transforming $(\phi,\phi')$ by use of the paths $P\in\Pscr(\phi,\phi')$
with $\pi(P)\in M_0$, as described in Lemma~\ref{lm:path_switch}, we obtain a
double flow $(\psi,\psi')$ for $(B,B')$ such that $E_\psi\sqcup
E_{\psi'}=E_\phi\sqcup E_{\phi'}$, and therefore $w(\psi)\odot
w(\psi')=w(\phi)\odot w(\phi')$. Moreover, $(\phi,\phi')\mapsto (\psi,\psi')$
gives a bijection between all double flows involved in the configurations in
$\Kscr(\Ascr)$ and those in $\Kscr(\Bscr)$. Now the desired
equality~\refeq{gen_sq} follows by considering the last term
in~\refeq{ff-zeta}.
  \end{proof}

The rest of this section is devoted to additional conventions and
illustrations.

Let $M$ be a planar perfect matching on $Y\sqcup Y'$. Sometimes it will be
convenient to assume that all couples $\pi\in M$ are ordered: if $\pi$ consists
of elements $i,j$, we may write $\pi=ij$ if: either (a) $i,j\in Y$ and $i<j$,
or (b) $i,j\in Y'$ and $i<j$, or (c) $i\in Y$ and $j\in Y'$. We call a couple
$\pi$ in these cases \emph{lower horizontal}, \emph{upper horizontal}, and
\emph{vertical}, respectively. The subsets of such couples in $M$ are denoted
by $\Mlh$, $\Muh$, and $\Mver$, respectively. When $\pi=ij$ is horizontal, we
denote the interval $\{i,i+1,\ldots,j\}$ by $[\pi]$. The fact that $M$ is
planar (cf.~\refeq{match1}(iii)) implies that
  \begin{numitem1}
the set $\Mlh$ is \emph{nested}, which means that for any $\pi,\pi'\in\Mlh$,
the intervals $[\pi]$ and $[\pi']$ are either disjoint or one includes the
other; also for each $\pi\in\Mlh$, all elements of $Y$ within $[\pi]$ are
covered by couples in $\Mlh$; similar properties hold for $\Muh$ and $Y'$.
  \label{eq:nest}
  \end{numitem1}

A proper pair $(A,A')\in \Pi_{Y,Y'}$ along with a feasible matching $M$ for it
can be illustrated by use of either a \emph{circular} diagram or a
\emph{two-level} diagram; the couples in the former are connected by
straight-line segments, and those in the latter by straight-line segments or by
arcs. See the picture where $Y=\{1,2,3,4\}$, $Y'=\{1',2'\}$, $A=\{1,3\}$,
$A'=\{1'\}$, $\Mlh=\{34\}$, $\Muh=\emptyset$, and $\Mver=\{11',22'\}$.

 \begin{center}
 \unitlength=1mm
\begin{picture}(120,40)
  \put(0,0){\begin{picture}(50,35)   
 \qbezier(20,0)(2,2)(0,20)
 \qbezier(20,0)(38,2)(40,20)
 \qbezier(0,20)(2,38)(20,40)
 \qbezier(20,40)(38,38)(40,20)
 \put(3,10){\circle{2}}
 \put(37,10){\circle*{2}}
 \put(10,3){\circle*{2}}
 \put(25,1){\circle{2}}
 \put(12.3,38){\circle{2}}
 \put(37,30){\circle*{2}}
{\thicklines
 \put(37,10){\line(-4,-3){12}}
 \put(3,10){\line(1,3){9.2}}
 \put(10,3){\line(1,1){27}}
 }
 \put(0,7){1}
 \put(7,0){2}
 \put(27,-1.5){3}
 \put(38,7){4}
 \put(10,39){$1'$}
 \put(38,31){$2'$}
  \multiput(-10,20)(10,0){6}%
{\put(0,0){\line(1,0){6}}}
 \end{picture}}
 \put(60,20){or}
  \put(80,0){\begin{picture}(50,35)   
 \put(0,10){\circle{2}}
 \put(20,10){\circle{2}}
 \put(10.5,30){\circle{2}}
 \put(10,10){\circle*{2}}
 \put(30,10){\circle*{2}}
 \put(20,30){\circle*{2}}
 \put(21,10){\line(1,0){9}}
 \put(0.5,10.5){\line(1,2){9.5}}
 \put(10,10){\line(1,2){10}}
  \put(0,5){1}
 \put(10,5){2}
 \put(20,5){3}
 \put(30,5){4}
 \put(10,33){$1'$}
 \put(20,33){$2'$}
 \end{picture}}
 \end{picture}
  \end{center}

Recall that in the flag flow case we deal with 1-patterns on $[m]$ and
1-families on $Y\subseteq[n]$ (with $|Y|=m$), which are formed by $p$-element
subsets in these sets (cf. Definition~1 in the Introduction). They are
equivalent, respectively, to 2-patterns on $([m],[m'])$ and 2-families on
$(Y,Y')$, where $|Y'|=m'=|p-(m-p)|$. Theorem~\ref{tm:main} implies the
following criterion on Pl\"ucker type relations.
     \begin{corollary} \label{cor:Pluck}
Let $\Ascr_0,\Bscr_0\Subset\binom{[m]}{p}$. Then~\refeq{S_pluck} is a
PSQ-relation (where $\Ascr=\gamma_{Y}(\Ascr_0)$ and
$\Bscr=\gamma_{Y}(\Bscr_0)$) if and only if the 1-patterns $\Ascr_0,\Bscr_0$
are balanced.
  \end{corollary}

Here the notion of \emph{balanced} 1-families $\Ascr,\Bscr$ (in particular,
1-patterns) comes from the one given for 2-families and is specified as
follows. Let $q:=m-p$ and assume, w.l.o.g., that $p\ge q$. A \emph{feasible
matching} for a set $A\in\binom{Y}{p}$ (or for the partition $(A,\bar A)$ of
$Y$) is now defined to be a set $\tilde M$ of pairs (couples) in $Y$ such that
    \begin{numitem1}
  \begin{itemize}
\item[(i)] $|\tilde M|=q$, the couples in $\tilde M$ are mutually disjoint,
and $|\pi\cap A|=|\pi\cap\bar A|=1$ for each $\pi\in \tilde M$;
\item[(ii)] $\tilde M$ is \emph{nested}, and for each $\pi\in \tilde M$, all elements
of $[\pi]$ are covered by $\tilde M$;
  \end{itemize}
   \label{eq:match2}
   \end{numitem1}
cf.~\refeq{nest}. In other words, $\tilde M$ is just the set $\Mlh$ in the
corresponding planar perfect matching $M$ for $Y\sqcup Y'$. In view of $|A|=p$,
~$|\bar A|=q$ and $|Y'|=p-q$, we have $\Muh=\emptyset$ and $|\Mver|=p-q$. In
particular, the elements of $Y'$ are colored white, provided that the elements
of $A$ and $\bar A$ are white and black, respectively.

For illustrations in the flag case, we will use \emph{flat} (or
\emph{one-level}) \emph{diagrams}. An example of such a diagram and its
corresponding two-level diagram are drawn in the picture; here $Y=[7]$,
$A=\{1,3,5,6\}$, $\tilde M=\{14,23,67\}$, and $Y'=\{1'\}$.
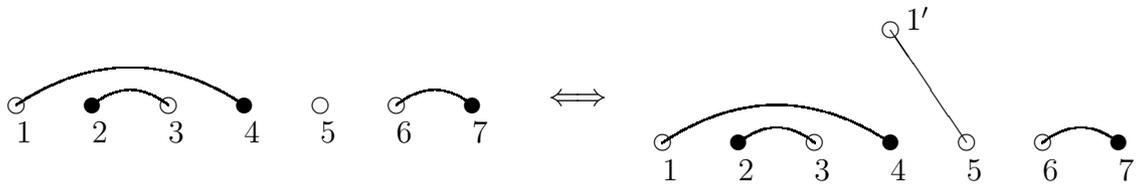
 \begin{center}
 \unitlength=1mm
\begin{picture}(150,23)
  \put(0,5){\begin{picture}(60,10)   
 \put(0,5){\circle{2}}
 \put(20,5){\circle{2}}
 \put(40,5){\circle{2}}
 \put(50,5){\circle{2}}
 \put(10,5){\circle*{2}}
 \put(30,5){\circle*{2}}
 \put(60,5){\circle*{2}}
 \put(0,0){1}
 \put(10,0){2}
 \put(20,0){3}
 \put(30,0){4}
 \put(40,0){5}
 \put(50,0){6}
 \put(60,0){7}
 \qbezier(10,5)(15,9)(20,5)
 \qbezier(50,5)(55,9)(60,5)
 \qbezier(0,5)(15,15)(30,5)
   \end{picture}}
 \put(70,10){$\Longleftrightarrow$}
  \put(85,0){\begin{picture}(60,10)   
 \put(0,5){\circle{2}}
 \put(20,5){\circle{2}}
 \put(40,5){\circle{2}}
 \put(50,5){\circle{2}}
 \put(10,5){\circle*{2}}
 \put(30,5){\circle*{2}}
 \put(60,5){\circle*{2}}
 \put(0,0){1}
 \put(10,0){2}
 \put(20,0){3}
 \put(30,0){4}
 \put(40,0){5}
 \put(50,0){6}
 \put(60,0){7}
 \qbezier(10,5)(15,9)(20,5)
 \qbezier(50,5)(55,9)(60,5)
 \qbezier(0,5)(15,15)(30,5)
 \put(30,20){\circle{2}}
 \put(40,5){\line(-2,3){10}}
 \put(32,20){$1'$}
   \end{picture}}

 \end{picture}
  \end{center}


\section{\Large Examples of stable quadratic relations}  \label{sec:relat}

In this section we illustrate the method described in the previous section by
demonstrating several classes of stable quadratic relations on FG-functions.
According to Proposition~\ref{pr:balance-sq}, once we are able to show that one
or another pair of families $\Ascr,\Bscr$ is balanced, we can declare that
relation~\refeq{gen_sq} involving these families is stable.

As before, when visualizing a proper pair $(C\subseteq[m],C'\subseteq[m'])$
(i.e., satisfying~\refeq{proper}(ii)), we will refer to the elements of $C$ and
$C'$ as white, and to the elements of their complements $\bar C=[m]-C$ and
$\bar C\,'=[m']-C'$ as black.

Most of examples below (namely, those in items~1--5) concern PSQ-relations for
flag-flow-determined functions $f:2^{[n]}\to\mathfrak{S}$. In these cases, we
deal with 1-patterns $\Ascr_0,\Bscr_0\subseteq \binom{[m]}{p}$ for some $p<m$
and set $q:=m-p$. Also, considering one or another white-black partition
$(C,\bar C)$ of $[m]$ (with $|C|=p$) and a feasible matching $M$ for it, we
illustrate the configuration $(C,M)$ by a flat diagram (introduced in the end
of the previous section). The set of feasible matchings for $(C,\bar C)$ is
denoted by $\Mscr(C)$.
\medskip

\noindent \textbf{1.} When $m=3$ and $p=2$, the collection $\binom{[m]}{p}$
consists of three 2-element sets $C$, namely, $12, 13, 23$, and their
complements $\bar C$ are the 1-element sets $3,2,1$, respectively. Since $q=1$,
a feasible matching consists of a unique couple. The sets 12 and 23 admit only
one feasible matching each, namely, $\Mscr(12)=\{\{23\}\}$ and
$\Mscr(23)=\{\{12\}\}$, whereas 13 has two feasible matchings, namely,
$\Mscr(13)=\{\{12\},\{23\}\}$. Therefore, the 1-patterns $\Ascr_0:=\{13\}$ and
$\Bscr_0:=\{12,23\}$ are balanced. The corresponding configurations and
bijection are illustrated in the picture.

 \begin{center}
  \unitlength=1mm
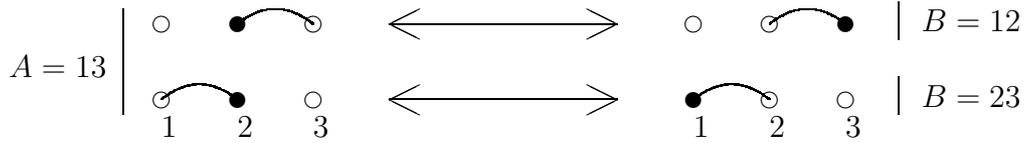
  \begin{picture}(140,20)
  \put(20,5){\circle{2}}
  \put(20,15){\circle{2}}
  \put(40,5){\circle{2}}
  \put(40,15){\circle{2}}
  \put(90,15){\circle{2}}
  \put(100,5){\circle{2}}
  \put(100,15){\circle{2}}
  \put(110,5){\circle{2}}
  \put(30,5){\circle*{2}}
  \put(30,15){\circle*{2}}
  \put(90,5){\circle*{2}}
  \put(110,15){\circle*{2}}
  \multiput(50,5)(0,10){2}%
{\put(0,0){\line(1,0){30}}
  \put(0,0){\line(2,1){4}}
  \put(0,0){\line(2,-1){4}}
  \put(30,0){\line(-2,1){4}}
  \put(30,0){\line(-2,-1){4}}}
  \multiput(20,5)(0,0){1}%
{\qbezier(0,0)(5,4)(10,0)}
  \multiput(30,15)(0,0){1}%
{\qbezier(0,0)(5,4)(10,0)}
  \multiput(90,5)(0,0){1}%
{\qbezier(0,0)(5,4)(10,0)}
  \multiput(100,15)(0,0){1}%
{\qbezier(0,0)(5,4)(10,0)}
  \put(20,0){1}
  \put(30,0){2}
  \put(40,0){3}
  \put(90,0){1}
  \put(100,0){2}
  \put(110,0){3}
  \put(15,3){\line(0,1){14}}
  \put(0,8){$A=13$}
  \put(117,3){\line(0,1){5}}
  \put(120,4){$B=23$}
  \put(117,13){\line(0,1){5}}
  \put(120,14){$B=12$}
  \end{picture}
   \end{center}

This gives rise to the \emph{P3-relation} (generalizing AP3- and
TP3-relations~\refeq{AP3},\refeq{TP3}): for a triple $i<j<k$ (forming $Y$) and
$X\subseteq[n]-\{i,j,k\}$,
   \begin{equation} \label{eq:SP3}
    f(Xik)\odot f(Xj)=(f(Xij)\odot f(Xk))\oplus(f(Xjk)\odot f(Xi)).
    \end{equation}

 \medskip
\noindent \textbf{2.} Let $p=q=2$. Take the 1-patterns $\Ascr_0:=\{13\}$ and
$\Bscr_0:=\{12,14\}$ in $\binom{[4]}{2}$. One can see that each of 12 and 14
admits a unique feasible matching: $\Mscr(12)=\{\{14,23\}\}$ and
$\Mscr(14)=\{\{12,34\}\}$, whereas $\Mscr(13)$ consists of two feasible
matchings: just those $\{14,23\}$ and $\{12,34\}$. Thus, $\Ascr_0,\Bscr_0$ are
balanced; see the picture where the couples (arcs) involved in the
corresponding exchange operations are marked with crosses.

 \begin{center}
  \unitlength=1mm
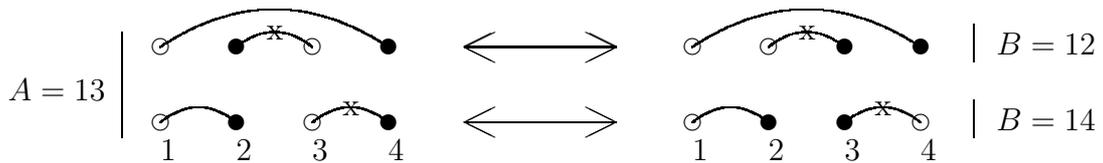
  \begin{picture}(150,20)
  \put(20,5){\circle{2}}
  \put(20,15){\circle{2}}
  \put(40,5){\circle{2}}
  \put(40,15){\circle{2}}
  \put(90,15){\circle{2}}
  \put(90,5){\circle{2}}
  \put(100,15){\circle{2}}
  \put(120,5){\circle{2}}
  \put(30,5){\circle*{2}}
  \put(30,15){\circle*{2}}
  \put(50,5){\circle*{2}}
  \put(50,15){\circle*{2}}
  \put(100,5){\circle*{2}}
  \put(110,5){\circle*{2}}
  \put(110,15){\circle*{2}}
  \put(120,15){\circle*{2}}
  \multiput(60,5)(0,10){2}%
{\put(0,0){\line(1,0){20}}
  \put(0,0){\line(2,1){4}}
  \put(0,0){\line(2,-1){4}}
  \put(20,0){\line(-2,1){4}}
  \put(20,0){\line(-2,-1){4}}}
  \multiput(20,5)(0,0){1}%
{\qbezier(0,0)(5,4)(10,0)}
  \multiput(40,5)(0,0){1}%
{\qbezier(0,0)(5,4)(10,0)}
  \multiput(30,15)(0,0){1}%
{\qbezier(0,0)(5,4)(10,0)}
  \multiput(90,5)(0,0){1}%
{\qbezier(0,0)(5,4)(10,0)}
  \multiput(110,5)(0,0){1}%
{\qbezier(0,0)(5,4)(10,0)}
  \multiput(100,15)(0,0){1}%
{\qbezier(0,0)(5,4)(10,0)}
  \multiput(20,15)(0,0){1}%
{\qbezier(0,0)(15,10)(30,0)}
  \multiput(90,15)(0,0){1}%
{\qbezier(0,0)(15,10)(30,0)}
  \put(44,6){x}
  \put(34,16){x}
  \put(114,6){x}
  \put(104,16){x}
  \put(20,0){1}
  \put(30,0){2}
  \put(40,0){3}
  \put(50,0){4}
  \put(90,0){1}
  \put(100,0){2}
  \put(110,0){3}
  \put(120,0){4}
  \put(15,3){\line(0,1){14}}
  \put(0,8){$A=13$}
  \put(127,3){\line(0,1){5}}
  \put(130,4){$B=14$}
  \put(127,13){\line(0,1){5}}
  \put(130,14){$B=12$}
  \end{picture}
   \end{center}

As a consequence, we obtain the \emph{P4-relation} (generalizing~\refeq{AP4}
and its tropical counterpart): for $i<j<k<\ell$ and
$X\subseteq[n]-\{i,j,k,\ell\}$,
   \begin{equation} \label{eq:SP4}
    f(Xik)\odot f(Xj\ell)=(f(Xij)\odot f(Xk\ell))\oplus(f(Xi\ell)\odot f(Xjk)).
    \end{equation}

 \medskip
\noindent \textbf{3.} As one more illustration of the method, let us consider
one particular case for $m=5$ and $p=3$. Put $\Ascr_0:=\{135\}$ and
$\Bscr_0:=\{234,125,145\}$. One can check that $\Mscr(234)=\{\{12,45\}\}$,
~$\Mscr(125):=\{\{14,23\},\{23,45\}\}$, ~$\Mscr(145)=\{\{12,34\},\{25,34\}\}$,
and that $\Mscr(135)$ consists just of the five matchings occurring in those
three collections. Therefore, $\Ascr_0,\Bscr_0$ are balanced. The corresponding
configurations and bijection are shown in the picture.

 \begin{center}
  \unitlength=1mm
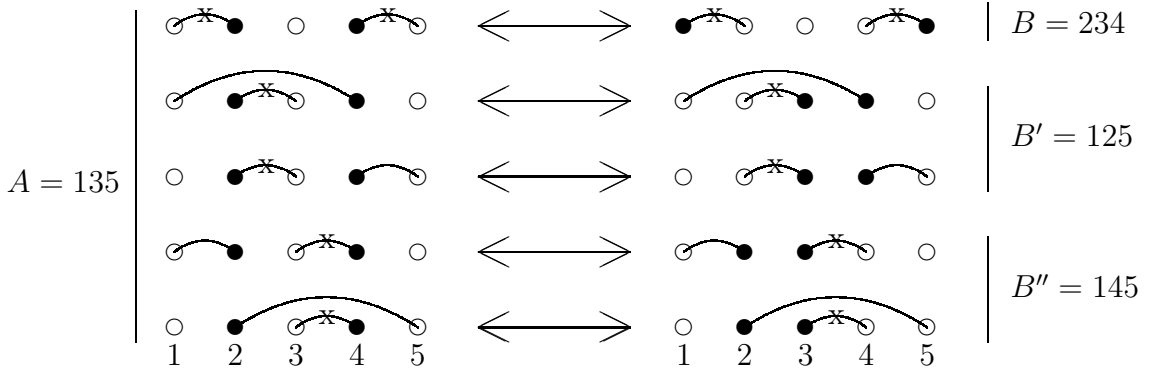
  \begin{picture}(140,47)
\multiput(20,5)(0,10){5}%
  {\put(0,0){\circle{2}}
  \put(8,0){\circle*{2}}
  \put(16,0){\circle{2}}
  \put(24,0){\circle*{2}}
  \put(32,0){\circle{2}}}
\multiput(87,5)(0,10){2}%
  {\put(0,0){\circle{2}}
  \put(8,0){\circle*{2}}
  \put(16,0){\circle*{2}}
  \put(24,0){\circle{2}}
  \put(32,0){\circle{2}}}
\multiput(87,25)(0,10){2}%
  {\put(0,0){\circle{2}}
  \put(8,0){\circle{2}}
  \put(16,0){\circle*{2}}
  \put(24,0){\circle*{2}}
  \put(32,0){\circle{2}}}
\multiput(87,45)(0,10){1}%
  {\put(0,0){\circle*{2}}
  \put(8,0){\circle{2}}
  \put(16,0){\circle{2}}
  \put(24,0){\circle{2}}
  \put(32,0){\circle*{2}}}
  \multiput(60,5)(0,10){5}%
{\put(0,0){\line(1,0){20}}
  \put(0,0){\line(2,1){4}}
  \put(0,0){\line(2,-1){4}}
  \put(20,0){\line(-2,1){4}}
  \put(20,0){\line(-2,-1){4}}}
  \multiput(36,5)(67,0){2}%
{\qbezier(0,0)(4,3)(8,0)
  \put(3,0.5){x}}
  \multiput(20,15)(67,0){2}%
{\qbezier(0,0)(4,3)(8,0)}
  \multiput(36,15)(67,0){2}%
{\qbezier(0,0)(4,3)(8,0)
  \put(3,0.5){x}}
  \multiput(28,25)(67,0){2}%
{\qbezier(0,0)(4,3)(8,0)
  \put(3,0.5){x}}
  \multiput(44,25)(67,0){2}%
{\qbezier(0,0)(4,3)(8,0)}
  \multiput(28,35)(67,0){2}%
{\qbezier(0,0)(4,3)(8,0)
  \put(3,0.5){x}}
  \multiput(20,45)(67,0){2}%
{\qbezier(0,0)(4,3)(8,0)
  \put(3,0.5){x}}
  \multiput(44,45)(67,0){2}%
{\qbezier(0,0)(4,3)(8,0)
  \put(3,0.5){x}}
  \multiput(28,5)(67,0){2}%
{\qbezier(0,0)(12,8)(24,0)}
  \multiput(20,35)(67,0){2}%
{\qbezier(0,0)(12,8)(24,0)}
  \multiput(-1,0)(67,0){2}%
  {\put(20,0){1}
  \put(28,0){2}
  \put(36,0){3}
  \put(44,0){4}
  \put(52,0){5}}
  \put(15,3){\line(0,1){44}}
  \put(-2,23){$A=135$}
  \put(127,3){\line(0,1){14}}
  \put(130,9){$B''=145$}
  \put(127,23){\line(0,1){14}}
  \put(130,29){$B'=125$}
  \put(127,43){\line(0,1){5}}
  \put(130,44){$B=234$}
  \end{picture}
   \end{center}

This implies a particular PSQ-relation on quintuples: for $i<j<k<\ell<r$ and
$X\subseteq [n]-\{i,j,k,\ell,r\}$,
   \begin{multline*}
    f(Xikr)\odot f(Xj\ell)
    =(f(Xjk\ell)\odot f(Xir))\oplus(f(Xijr)\odot f(Xk\ell))\\
      \oplus(f(Xi\ell r)\odot f(Xjk)).\quad
    \end{multline*}

 \noindent \textbf{4.}
The next illustration concerns a wide class of balanced 1-patterns for $m>p\ge
m-p=:q$; it includes the 1-patterns indicated in items 1 and 2 as special
cases.

The 1-pattern $\Ascr_0$ contains a distinguished set $A_0\in\binom{[m]}{p}$.
Fix a nonempty subset $Z\subseteq\bar A_0$ and consider the collection
   \begin{equation} \label{eq:Cscript}
   \Cscr:=\{C\subset[m] \,\colon |C|=p,\;C\cap \bar A_0=Z\}.
   \end{equation}
For a subset $C\subseteq[n]$, we will denote by $\Sigma(C)$ the number
$\sum(i\in C)$. Now define
   \begin{gather}
   \Ascr_0:=\{A_0\}\cup\{ A\in\Cscr\,\colon \Sigma(A)-\Sigma(A_0)+|Z|\;\; \mbox{odd}\}
   \quad\mbox{and} \nonumber \\
   \Bscr_0:=\{ B\in\Cscr\,\colon\Sigma(B)-\Sigma(A_0)+|Z|\;\; \mbox{even}\}.
   \label{eq:AA4}
   \end{gather}
In particular, $\Ascr_0\cap\Bscr_0=\emptyset$. We assert the following:
   \begin{lemma} \label{lm:AA4}
The pair $\Ascr_0,\Bscr_0$ in~\refeq{AA4} is balanced.
  \end{lemma}
  \begin{proof}
~Consider $C\in\Cscr$ and $M\in\Mscr(C)$. We describe a rule which associates
to $(C,M)$ another configuration $(D,M)$, aiming to obtain the desired
bijection between $\Kscr(\Ascr_0)$ and $\Kscr(\Bscr_0)$.

Since $M$ is feasible and $p\ge q$, we have $|M|=q=|\bar A_0|$. This implies
that exactly one of the two cases is possible: (i) there is a couple $\pi\in M$
with both elements in $A_0$; and (ii) each $\pi\in M$ satisfies $|\pi\cap
A_0|=1$ (whence $|\pi\cap \bar A_0|=1$ and $M$ covers $\bar A_0$).
\smallskip

In case~(i), take the couple $\pi=ij\in M$ ($i<j$) such that $i,j\in A_0$ and
$i$ is minimum under this property. Let $D:=C\triangle\, \pi$. Then $D\cap \bar
A_0=C\cap \bar A_0=Z$, whence $D\in\Cscr$. Also the interval $[\pi]$ is
partitioned into couples (cf.~\refeq{match2}(ii)), implying that $j-i$ is odd.
Hence the numbers $\Sigma(C)-\Sigma(A_0)+|Z|$ and $\Sigma(D)-\Sigma(A_0)+|Z|$
have different parity, and therefore, $C,D$ belong to different collections
among $\Ascr_0,\Bscr_0$. Obviously, $M$ is a feasible matching for $D$, the
couple $\pi$ satisfies the above minimality property for $D$, and applying to
$D$ the exchange operation w.r.t. $\pi$ returns $C$. We associate the
configurations $(C,M)$ and $(D,M)$ to each other.
\smallskip

In case (ii), each couple of $M$ has one element in $A_0$ and the other in
$\bar A_0$. Let $M_Z$ be the set of $\pi\in M$ such that $\pi\cap Z\ne
\emptyset$ and let $Q:=\cup(\pi\in M_Z$). Then the set $D:=C\triangle\, Q$
satisfies $|D|=p$ and $D\cap \bar A_0=\emptyset$. This means that $D=A_0$. Also
$M\in\Mscr(A_0)$ and $C=A_0\triangle\, Q$. Since $j-i$ is odd for each $ij\in
M$, the numbers $\Sigma(C)-\Sigma(A_0)$ and $Z$ have the same parity. So $C\in
\Bscr_0$. We associate $(C,M)$ and $(A_0,M)$ to each other. Conversely, for
each $M'\in\Mscr(A_0)$, let $C':=A_0\triangle\, Q'$, where $Q':=\cup(\pi\in
M'\,\colon \pi\cap Z\ne\emptyset)$. Then $Q'\cap \bar A_0=Z$, ~$C'\in \Bscr_0$,
and the above construction associates $(C',M')$ with $(A_0,M')$.
\smallskip

This gives the desired bijection between $\Kscr(\Ascr_0)$ and $\Kscr(\Bscr_0)$.
  \end{proof}

\noindent \textbf{Remark 4.} ~(i) Consider $m=3$, $p=2$, $A_0:=12$, and
$Z:=\{3\}$. Then the collection $\Cscr$ as in~\refeq{Cscript} consists of the
sets $13,23$, and we have $\Ascr_0=\{12,23\}$ and $\Bscr_0=\{13\}$. These
1-patterns correspond to those in item~1. (ii) When $m=4$, $p=2$, $A_0:=12$,
and $Z:=\{3\}$, we obtain $\Ascr_0=\{12,23\}$ and $\Bscr_0=\{13\}$. They
generate the same PSQ-relations as the balanced 1-patterns $\{13\}$,
$\{12,14\}$ in item~2.
\medskip

The 1-patterns $\Ascr_0,\Bscr_0$ as in~\refeq{AA4} give rise to the following
PSQ-relation on FG-functions $f$; for brevity, it is exposed when
$\mathfrak{S}$ is a ring. Let $I,J\subset[n]$ and $|I|\ge |J|$. Fix $Z\subseteq
J-I$. Then
  \begin{equation} \label{eq:Fl-manif}
 f(I)f(J) =\sum_K (-1)^{a+{\rm Inv}((I-K)\cup Z,(J-Z)\cup K)}
 f((I-K)\cup Z)\, f((J-Z)\cup K),
   \end{equation}
where: the sum is over all $K\subseteq I-J$ with $|K|=|Z|$; ~${\rm Inv}(I',J')$
denotes the number of pairs $(i,j)\in I'\times J'$ with $i>j$ (inversions); and
$a:=|Z|+{\rm Inv}(I-J,J-I)$. In this case one should set $X:=I\cap J$,
$Y:=I\triangle\,J$, $m:=|Y|$, $p:=|I-J|$, and $A_0:=I-J$.

Relations similar to~\refeq{Fl-manif} (but possibly given in a different form)
appear in a characterization of flag manifolds ${\rm Fl}={\rm
Fl}^{d_1,\ldots,d_r}(\mathbb{C}^n)$, where $d_1<\ldots<d_r\le n$;
cf.~\cite[Ch.9]{Fu}. In this case one should take all subsets $I,J\subseteq[n]$
and $Z\subseteq J-I$ with $|I|=d_i$, $|J|=d_j$, $i\ge j$; then~\refeq{Fl-manif}
generate the ideal of polynomials with zero values on ${\rm Fl}$ canonically
embedded in the corresponding product of projective spaces.
  \Xcomment{
As is mentioned in~\cite{LNT}, certain special cases of~\refeq{Fl-manif} admit
analogues (extensions) in the quantized Pl\"ucker algebra (in particular, when:
(i) $Z=J-I$; (ii) $|I|-1\ge |J|+1$ and $|Z|=1$; (iii) $I,J$ are such that $J-I$
has a (unique) partition $(J_1,J_2)$ satisfying $J_1< I-J< J_2$ (such $I,J$ are
called \emph{weakly separated}), and $Z=J_1$ ; here we write $A< B$ if $i< j$
for all $i\in A$ and $j\in B$).
  }

 \medskip
 \noindent \textbf{5.}
One more representable class of balanced 1-patterns for $p<m\le n$ with $p\ge
m-p=:q$ is obtained by slightly modifying the previous construction.

Fix a set $Z\subset[m]$ with $0<|Z|\le q-1$ and a subset $Z'\subseteq Z$. Form
the collection
   $$
   \Cscr:=\{C\subset[m]\, \colon |C|=p,\; C\cap Z=Z'\}.
   $$
Partition $\Cscr$ into two 1-patterns
   \begin{equation} \label{eq:AA5}
\Ascr_0:=\{A\in\Cscr\,\colon \Sigma(A)\;\;\mbox{odd}\}\quad \mbox{and}\quad
 \Bscr_0:=\{A\in\Cscr\,\colon \Sigma(A)\;\;\mbox{even}\}.
   \end{equation}

 \begin{lemma} \label{lm:AA5}
The pair $\Ascr_0,\Bscr_0$ in~\refeq{AA5} is balanced.
  \end{lemma}
  \begin{proof}
~Let $C\in\Cscr$ and $M\in\Mscr(C)$. Since $|M|=q>|Z|$, there exists a couple
$\pi=ij\in M$ ($i<j$) with both elements in $[m]-Z$; take such a $\pi$ so that
$i$ be minimum. Form $D:=C\triangle\,\pi$. Then $C,D$ belong to different
1-patterns among $\Ascr_0,\Bscr_0$ (since $j-i$ is odd), and we can associate
$(C,M)$ and $(D,M)$ to each other, taking into account that $M\in\Mscr(D)$ and
that the choice of $\pi$ depends only on $M$.
  \end{proof}

This lemma gives rise to the corresponding class of PSQ-relations; we omit it
here. (In fact, such relations can be derived from those in item~4 when
$\mathfrak{S}$ is a ring.) \medskip

In the rest of this section we give simple examples of balanced families in the
non-flag case. Now we deal with disjoint sets $X,Y\subseteq [n]$ and disjoint
sets $X',Y'\subseteq[n']$, denote $m:=|Y|$ and $m':=|Y'|$, and consider
2-patterns formed by proper pairs for $([m],[m'])$ (cf.~\refeq{proper}).
Corresponding matchings will be illustrated by use of \emph{two-level diagrams}
(see the end of Section~\ref{sec:balan}) in which the white/black elements of
$[m]$ (resp., $[m']$) are disposed in the lower (resp., upper) horizontal line.
 \medskip

 \noindent \textbf{6.}
The picture below shows an example of balanced homogeneous 2-patterns
$\Ascr_0,\Bscr_0$.
 \begin{center}
  \unitlength=1mm
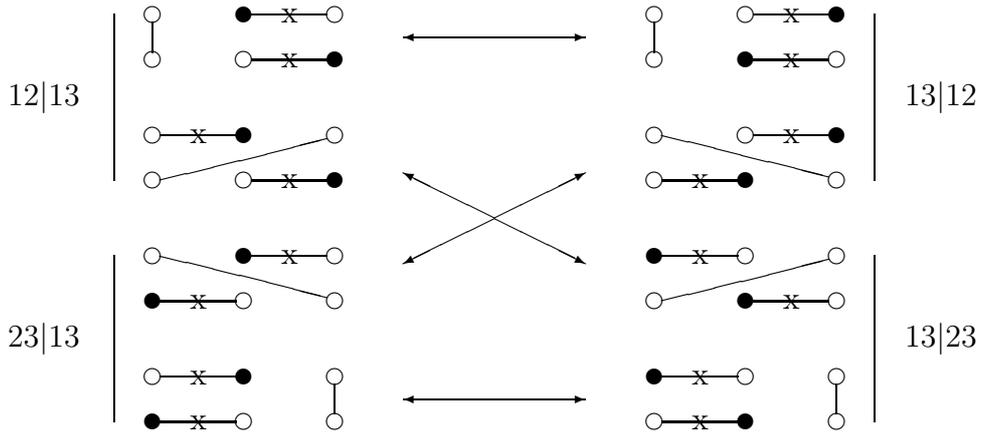
  \begin{picture}(140,54)
\multiput(20,0)(0,16){2}%
  {\put(0,0){\circle*{2}}
  \put(12,0){\circle{2}}
  \put(24,0){\circle{2}}
  \put(0,6){\circle{2}}
  \put(12,6){\circle*{2}}
  \put(24,6){\circle{2}}}
\multiput(20,32)(0,16){2}%
  {\put(0,0){\circle{2}}
  \put(12,0){\circle{2}}
  \put(24,0){\circle*{2}}
  \put(0,6){\circle{2}}
  \put(12,6){\circle*{2}}
  \put(24,6){\circle{2}}}
\multiput(86,0)(0,16){2}%
  {\put(0,0){\circle{2}}
  \put(12,0){\circle*{2}}
  \put(24,0){\circle{2}}
  \put(0,6){\circle*{2}}
  \put(12,6){\circle{2}}
  \put(24,6){\circle{2}}}
\multiput(86,32)(0,16){2}%
  {\put(0,0){\circle{2}}
  \put(12,0){\circle*{2}}
  \put(24,0){\circle{2}}
  \put(0,6){\circle{2}}
  \put(12,6){\circle{2}}
  \put(24,6){\circle*{2}}}
   \put(15,0){\line(0,1){22}}
   \put(15,32){\line(0,1){22}}
   \put(115,0){\line(0,1){22}}
   \put(115,32){\line(0,1){22}}
  \put(1,10){$23|13$}
  \put(1,42){$12|13$}
  \put(119,10){$13|23$}
  \put(119,42){$13|12$}
  \put(21,0){\line(1,0){10}}
  \put(25,-1){x}
  \put(21,6){\line(1,0){10}}
  \put(25,5){x}
  \put(21,16){\line(1,0){10}}
  \put(25,15){x}
  \put(21,38){\line(1,0){10}}
  \put(25,37){x}
  \put(33,22){\line(1,0){10}}
  \put(37,21){x}
  \put(33,32){\line(1,0){10}}
  \put(37,31){x}
  \put(33,48){\line(1,0){10}}
  \put(37,47){x}
  \put(33,54){\line(1,0){10}}
  \put(37,53){x}
  \put(44,1){\line(0,1){4}}
  \put(20,49){\line(0,1){4}}
  \put(21,22){\line(4,-1){22}}
  \put(21,32){\line(4,1){22}}
  \put(87,0){\line(1,0){10}}
  \put(91,-1){x}
  \put(87,6){\line(1,0){10}}
  \put(91,5){x}
  \put(99,16){\line(1,0){10}}
  \put(103,15){x}
  \put(99,38){\line(1,0){10}}
  \put(103,37){x}
  \put(87,22){\line(1,0){10}}
  \put(91,21){x}
  \put(87,32){\line(1,0){10}}
  \put(91,31){x}
  \put(99,48){\line(1,0){10}}
  \put(103,47){x}
  \put(99,54){\line(1,0){10}}
  \put(103,53){x}
  \put(110,1){\line(0,1){4}}
  \put(86,49){\line(0,1){4}}
  \put(87,38){\line(4,-1){22}}
  \put(87,16){\line(4,1){22}}
  \put(65,3){\vector(1,0){12}}
  \put(65,3){\vector(-1,0){12}}
  \put(65,51){\vector(1,0){12}}
  \put(65,51){\vector(-1,0){12}}
  \put(65,27){\vector(2,1){12}}
  \put(65,27){\vector(-2,-1){12}}
  \put(65,27){\vector(2,-1){12}}
  \put(65,27){\vector(-2,1){12}}
  \end{picture}
   \end{center}
Here $m=m'=3$, ~$\Ascr_0$ consists of the pairs $12|13$ and $23|13$, and
$\Bscr_0$ consists of the pairs $13|12$ and $13|23$ (indicated by light
circles); we write $a|b$ for $(a,b)$. The feasible matchings are indicated by
line segments, and the couples involved in the corresponding exchange
operations are marked with crosses.

These 2-patterns give rise to the SQ-relation
  \begin{multline*}
\left(f(Xij|X'i'k')\odot f(Xk|X'j')\right) \oplus \left(f(Xjk|X'i'k')\odot f(Xi|X'j')\right) \\
  =(f(Xik|X'i'j')\odot f(Xj|X'k'))\oplus (f(Xik|X'j'k')\odot f(Xj|X'i')),
  \end{multline*}
where $i<j<k$ and $i'<j'<k'$ (this is rather trivial for minors of a matrix).
\medskip

 \noindent \textbf{7.}
One of the simplest examples of balanced non-homogeneous 2-patterns is formed
by $\Ascr_0=\{1|1\}$ and $\Bscr_0=\{2|1,12|12\}$, in case $m=m'=2$. See the
picture:

 \begin{center}
  \unitlength=1mm
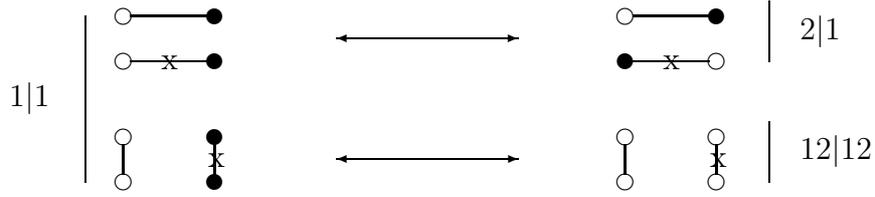
  \begin{picture}(130,25)
  \put(20,0){\circle{2}}
  \put(32,0){\circle*{2}}
  \put(20,6){\circle{2}}
  \put(32,6){\circle*{2}}
  \put(20,16){\circle{2}}
  \put(32,16){\circle*{2}}
  \put(20,22){\circle{2}}
  \put(32,22){\circle*{2}}
  \put(86,0){\circle{2}}
  \put(98,0){\circle{2}}
  \put(86,6){\circle{2}}
  \put(98,6){\circle{2}}
  \put(86,16){\circle*{2}}
  \put(98,16){\circle{2}}
  \put(86,22){\circle{2}}
  \put(98,22){\circle*{2}}
   \put(15,0){\line(0,1){22}}
   \put(105,0){\line(0,1){8}}
   \put(105,16){\line(0,1){8}}
  \put(5,10){$1|1$}
  \put(109,3){$12|12$}
  \put(109,19){$2|1$}
  \put(21,16){\line(1,0){10}}
  \put(25,15){x}
  \put(21,22){\line(1,0){10}}
  \put(20,1){\line(0,1){4}}
  \put(32,1){\line(0,1){4}}
  \put(31.2,2){x}
  \put(87,16){\line(1,0){10}}
  \put(91,15){x}
  \put(87,22){\line(1,0){10}}
  \put(86,1){\line(0,1){4}}
  \put(98,1){\line(0,1){4}}
  \put(97.2,2){x}
  \put(60,3){\vector(1,0){12}}
  \put(60,3){\vector(-1,0){12}}
  \put(60,19){\vector(1,0){12}}
  \put(60,19){\vector(-1,0){12}}
  \end{picture}
   \end{center}

\noindent This gives the following SQ-relation similar to Dodgson's
condensation formula for minors of a matrix~(cf.~\refeq{Dodg_orig}): for $i<k$
and $i'<k'$,
  \begin{multline} \label{eq:dodgson}
  f(Xi|X'i')\odot f(Xk|X'k') \\
=(f(Xik|X'i'k')\odot f(X|X')) \oplus (f(Xk|X'i')\odot f(Xi|X'k')).
   \end{multline}

Another simple non-homogeneous example is analogous to the row decomposition
(by row 2) of the determinant of a $3\times 3$ matrix. Here $m=m'=3$,
~$\Ascr_0=\{13|13\}$ and $\Bscr_0=\{12|13,23|13,123|123\}$; see the picture:

 \begin{center}
  \unitlength=1mm
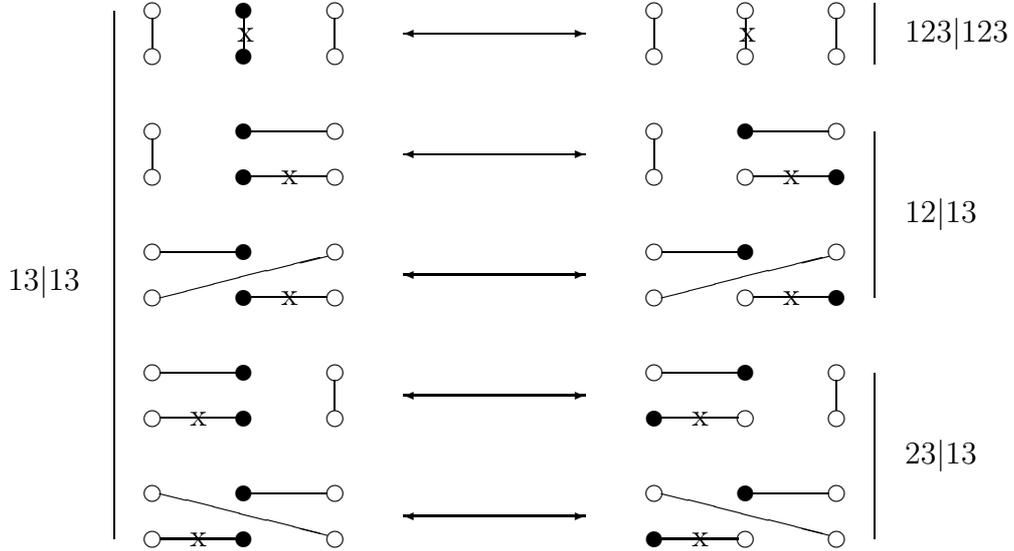
  \begin{picture}(140,70)
\multiput(20,0)(0,16){5}%
  {\put(0,0){\circle{2}}
  \put(12,0){\circle*{2}}
  \put(24,0){\circle{2}}
  \put(0,6){\circle{2}}
  \put(12,6){\circle*{2}}
  \put(24,6){\circle{2}}}
\multiput(86,0)(0,16){2}%
  {\put(0,0){\circle*{2}}
  \put(12,0){\circle{2}}
  \put(24,0){\circle{2}}
  \put(0,6){\circle{2}}
  \put(12,6){\circle*{2}}
  \put(24,6){\circle{2}}}
\multiput(86,32)(0,16){2}%
  {\put(0,0){\circle{2}}
  \put(12,0){\circle{2}}
  \put(24,0){\circle*{2}}
  \put(0,6){\circle{2}}
  \put(12,6){\circle*{2}}
  \put(24,6){\circle{2}}}
  \put(86,64){\circle{2}}
  \put(98,64){\circle{2}}
  \put(110,64){\circle{2}}
  \put(86,70){\circle{2}}
  \put(98,70){\circle{2}}
  \put(110,70){\circle{2}}
   \put(15,0){\line(0,1){70}}
   \put(115,0){\line(0,1){22}}
   \put(115,32){\line(0,1){22}}
   \put(115,63){\line(0,1){8}}
  \put(1,33){$13|13$}
  \put(119,10){$23|13$}
  \put(119,42){$12|13$}
  \put(119,66){$123|123$}
\multiput(20,0)(66,0){2}%
  {\put(1,0){\line(1,0){10}}
  \put(5,-1){x}
  \put(1,16){\line(1,0){10}}
  \put(5,15){x}
  \put(13,32){\line(1,0){10}}
  \put(17,31){x}
  \put(13,48){\line(1,0){10}}
  \put(17,47){x}
  \put(12,65){\line(0,1){4}}
  \put(11.2,66){x}
  \put(13,6){\line(1,0){10}}
  \put(1,22){\line(1,0){10}}
  \put(1,38){\line(1,0){10}}
  \put(13,54){\line(1,0){10}}
  \put(24,17){\line(0,1){4}}
  \put(0,49){\line(0,1){4}}
  \put(0,65){\line(0,1){4}}
  \put(24,65){\line(0,1){4}}
  \put(1,6){\line(4,-1){22}}
  \put(1,32){\line(4,1){22}}}
  \put(65,3){\vector(1,0){12}}
  \put(65,3){\vector(-1,0){12}}
  \put(65,19){\vector(1,0){12}}
  \put(65,19){\vector(-1,0){12}}
  \put(65,35){\vector(1,0){12}}
  \put(65,35){\vector(-1,0){12}}
  \put(65,51){\vector(1,0){12}}
  \put(65,51){\vector(-1,0){12}}
  \put(65,67){\vector(1,0){12}}
  \put(65,67){\vector(-1,0){12}}
  \end{picture}
   \end{center}


\section{\Large Necessity of the balancedness}  \label{sec:(i)-(ii)}

This section is devoted to the other direction in Theorem~\ref{tm:main}.
Moreover, we show a sharper property. It says that if a pair $\Ascr_0,\Bscr_0$
is not balanced, then for \emph{any} choice of appropriate consistent sets
$X,Y,X',Y'$ and for $\mathfrak{S}:=\Zset$, there exist (and can be explicitly
constructed) a planar network and a weighting such that the corresponding
flow-generated function $f$ violates relation~\refeq{gen_sq}.

As before, for subsets $C\subseteq Y$ and $C'\subseteq Y'$, we write $\bar C$
for $Y-C$, and $\bar C\,'$ for $Y'-C'$, and call $(C,C')$ a \emph{proper} pair
for $(Y,Y')$ if it satisfies~\refeq{proper}(ii).

  \begin{theorem} \label{tm:nonbalance}
Fix disjoint sets $X,Y\subseteq [n]$ and disjoint sets $X',Y'\subseteq [n']$
satisfying~\refeq{proper}(i). Let $\Ascr,\Bscr\Subset\Pi_{Y,Y'}$. Suppose that
$\Ascr,\Bscr$ are not balanced. Then~\refeq{gen_sq} does not hold for some
$(G,w)$ and $\mathfrak{S}=\Zset$. More precisely, there exists a planar network
$G=(V,E)$ with $n$ sources and $n'$ sinks such that for the all-unit weighting
$w\equiv 1$ on $V$, the flow-generated function $f=f_w$ on $\Escr^{n,n'}$ gives
  \begin{equation} \label{eq:non-balan}
  \sum\nolimits_{(A,A')\in\Ascr} f(XA|X'A')\,f(X\bar A|X'\bar A\,')
  \ne  \sum\nolimits_{(B,B')\in\Bscr} f(XB|X'B')\,f(X\bar B|X\bar B\,').
  \end{equation}
  \end{theorem}

  \begin{proof}
~Since $\Ascr,\Bscr$ are not balanced, there exists a planar perfect matching
$M$ on $Y\sqcup Y'$ such that
   \begin{equation} \label{eq:diff_M}
   |\Ascr_M|\ne |\Bscr_M|,
   \end{equation}
where $\Ascr_M$ denotes the set of pairs $(A,A')\in\Ascr$ having $M$ as a
feasible matching: $M\in\Mscr(A,A')$, and similarly for $\Bscr$.

We fix such an $M$, and our aim is to construct a planar network $G=(V,E)$ that
satisfies the following properties: for each proper pair $(C,C')$ for $(Y,Y')$,
  \begin{itemize}
\item[(P1)] If $M\in\Mscr(C,C')$, then ~$G$ has a
unique $(XC|X'C')$-flow and a unique $(X\bar C|X'\bar C\,')$-flow, i.e.,
$|\Phi_{XC|X'C'}|=|\Phi_{X\bar C|X'\bar C\,'}|=1$;
\item[(P2)] If $M\notin\Mscr(C,C')$, then at least one
of $\Phi_{XC|X'C'}$ and $\Phi_{X\bar C|X'\bar C\,'}$ is empty.
  \end{itemize}

Assuming that such a $G$ does exist, assign the weight $w(v):=1$ to each vertex
$v$. By~(P1) and~(P2), for the function $f=f_w$ and a proper pair
$(C,C')\subseteq (Y,Y')$, each of the values $f(XC|X'C')$ and $f(X\bar C|X'\bar
C\,')$ is equal to one if $M\in \Mscr(C,C')$, and at least one of them is zero
otherwise. This implies
 $$
  \sum_{(A,A')\in\Ascr} f(XA|X'A')\,f(X\bar A|X'\bar A\,')=|\Ascr_M|,\;\;
  \sum_{(B,B')\in\Bscr} f(XB|X'B')\,f(X\bar B|X'\bar B\,')=|\Bscr_M|,
 $$
and now the required inequality~\refeq{non-balan} follows from~\refeq{diff_M}.

It suffices to construct the desired network $G$ in case $n=|X|+|Y|$ and
$n'=|X'|+|Y'|$ (for we can add a source $s_i$ for $i\in[n]-(X\cup Y)$ (if
exists) as an isolated vertex, and can do similarly for sinks). So we may
assume, w.l.o.g., that $X,Y$ form a partition of $[n]$, and $X',Y'$ do that of
$[n']$.

We first describe the construction when $X=X'=\emptyset$, which is the crucial
special case. Subsequently we will explain that this construction can be easily
extended to arbitrary $X,X'$.

Thus, we deal with $n=|Y|=|Y'|$ sources $s_1,\ldots,s_n$ and $n$ sinks
$t_1,\ldots,t_n$. As usual, the sources (sinks) lie in the lower (resp., upper)
half of a circumference $O$ in the plane, and their indices grow from left to
right. The other vertices of $G$ lie inside $O$. All edges will be represented
by directed straight-line segments. The graph $G$ is constructed in five steps.
\smallskip

\emph{Step 1.} ~For each couple $\pi\in M$, we draw the segment between
corresponding terminals, denoted by $L_\pi$. Namely: (a) if $\pi=ij\in\Mlh$,
~$L_\pi$ connects the sources $s_i,s_j$ (a \emph{lower horizontal} segment);
~(b) if $\pi=ij\in\Muh$, ~$L_\pi$ connects the sinks $t_i,t_j$ (an \emph{upper
horizontal} segment); ~and (c) if $\pi=ij\in\Mver$, ~$L_\pi$ connects the
source $s_i$ and sink $t_j$ (a \emph{vertical} segment). In case~(c), we direct
$L_\pi$ from $s_i$ to $t_j$. These segments are pairwise disjoint (since $M$ is
planar).
\smallskip

\emph{Step 2.} ~For each $\pi=ij\in\Mlh$, the lower horizontal segment $L_\pi$
is transformed into a graph whose vertices are $s_i,s_j$ and $j-i$ distinct
points in the interior of the segment. The edges are the $j-i+1$ subsegments
connecting consecutive vertices. We distinguish between two sorts of vertices,
called \emph{odd} and \emph{even} ones, so that $s_i,s_j$ are regarded as odd,
and the odd and even vertices alternate along $L_\pi$. Each edge is directed
from the odd to even vertex. So $L_\pi$ becomes a path with alternating edge
directions, and its end vertices $s_i,s_j$ have leaving edges.

Each upper horizontal segment $L_{\pi=ij}$ is transformed into a path in a
similar fashion, but now we direct each edge from the even to odd vertex. So
the end vertices $t_i,t_j$ of $L_\pi$ are odd and have entering edges.
\smallskip

\emph{Step 3.} ~The horizontal segments (regarded as graphs) are connected by
additional edges. To define them, let us say that a couple $\pi=ij\in\Mlh$ is a
\emph{predecessor} of another couple $\pi'\in\Mlh$ if $[\pi]\supset[\pi']$. If,
in addition, there is no $\pi''\in\Mlh$ between $\pi$ and $\pi'$ (i.e.,
$[\pi]\supset[\pi'']\supset[\pi']$), $\pi$ is called the \emph{immediate
predecessor} of $\pi'$. Accordingly, $\pi'$ is called a \emph{successor} of
$\pi$ in the former case, and an \emph{immediate successor} in the latter case.
A couple is \emph{maximal} (\emph{minimal}) if it has no predecessor (resp., no
successor). The set of successors (immediate successors) of $\pi$ is denoted by
$\Succ(\pi)$ (resp., by $\ISucc(\pi)$); moreover, we order the couples in
$\ISucc(\pi)$, say, $\pi_1=i_1j_1,\ldots,\pi_r=i_rj_r$, so that $j_d<i_{d+1}$
for $d=1,\ldots,r-1$. It is easy to see that $i_1=i+1$, $j_r=j-1$ and
$i_{d+1}=j_d+1$ for each $d$ (when $r\ge 1$).

Note that each path $L_{\pi_d}$ has exactly $(j_d-i_d+1)/2$ ~even vertices.
Also $L_\pi$ has exactly $(j-i-1)/2$ nonterminal odd vertices $v$ (i.e., $v\ne
s_i,s_j$). Then
   $$
   \frac12\sum\nolimits_{d=1}^{r}(j_d-i_d+1)=\frac12(j_r-i_1+1)=\frac12(j-i-1),
   $$
yielding the equality
    $$
    |\Vodd(\pi)|=\sum(|\Veven(\pi')|\,\colon \pi'\in\ISucc(\pi)),
    $$
where $\Vodd(\pi'')$ ~($\Veven(\pi'')$) denotes the set of \emph{nonterminal}
odd vertices (resp., of even vertices) in a path $L_{\pi''}$. Observe that
within the circle $O^*$ surrounded by $O$, the region confined by the segments
for $\{\pi\}\cup\ISucc(\pi)$ is convex and does not meet any other segment for
$M$. Ordering the vertices in each of the two equally sized sets
$W(\pi):=\cup(\Veven(\pi')\,\colon \pi'\in\ISucc(\pi))$ and $\Vodd(\pi)$ from
left to right, we draw a directed edge (segment) from each vertex of the former
to the corresponding vertex of the latter. These edges are pairwise disjoint;
we call them \emph{lower bridges} for $\pi$. See the picture where $\pi=16$ and
the dark and light circles indicate even and nonterminal odd vertices in
$L_\pi$, respectively.

 \begin{center}
  \unitlength=1mm
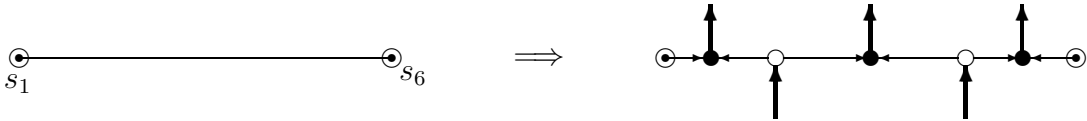
  \begin{picture}(145,17)(0,2)
 \put(0,10){\circle{2.5}}
 \put(0,10){\circle*{1}}
 \put(49,10){\circle{2.5}}
 \put(49,10){\circle*{1}}
 \put(-2,6){$s_1$}
 \put(50,7){$s_6$}
 \put(0,10){\line(1,0){49}}
 \put(65,9){$\Longrightarrow$}

 \put(85,10){\circle{2.5}}
 \put(85,10){\circle*{1}}
 \put(139,10){\circle{2.5}}
 \put(139,10){\circle*{1}}
 \put(91,10){\circle*{2}}
 \put(112,10){\circle*{2}}
 \put(132,10){\circle*{2}}
 \put(99.5,10){\circle{2}}
 \put(124.5,10){\circle{2}}
 \put(85,10){\vector(1,0){5}}
 \put(98.5,10){\vector(-1,0){6.5}}
 \put(100.5,10){\vector(1,0){11}}
 \put(123.5,10){\vector(-1,0){10.5}}
 \put(125.5,10){\vector(1,0){6}}
 \put(139,10){\vector(-1,0){6}}
{\thicklines
 \put(99.4,2){\vector(0,1){7}}
 \put(124.4,2){\vector(0,1){7}}
 \put(99.6,2){\vector(0,1){7}}
 \put(124.6,2){\vector(0,1){7}}
 \put(90.8,10){\vector(0,1){7}}
 \put(131.8,10){\vector(0,1){7}}
 \put(111.8,10){\vector(0,1){7}}
 \put(91.0,10){\vector(0,1){7}}
 \put(132.0,10){\vector(0,1){7}}
 \put(112.0,10){\vector(0,1){7}}
}
  \end{picture}
   \end{center}

A similar edge set is constructed for each non-minimal upper couple
$\pi\in\Muh$, connecting the set $\Vodd(\pi)$ of nonterminal odd vertices in
$L_\pi$ and the set $W(\pi)$ of even vertices in $L_{\pi'}$ among $\pi'\in
\ISucc(\pi)$. The only difference is that such edges, called \emph{upper
bridges} for $\pi$, are now directed from odd  to even vertices.

The following observation is useful:
  \begin{numitem}
for each $\pi\in\Mlh$, ~$|\Veven(\pi)|=|\{\pi\}\cup\Succ(\pi)|$, and similarly
for $\pi\in\Muh$.
  \label{eq:Vqpi}
  \end{numitem}
(This follows from the equality $|\Veven(\pi)|=(j-i+1)/2$, where $\pi=ij$, and
the fact that the successors of $\pi$ form a perfect matching on
$\{i+1,\ldots,j-1\}$, implying $|\Succ(\pi)|=(j-i-1)/2$.)
\smallskip

\emph{Step 4.} ~Let $\Mlhmax$ ($\Muhmax$) be the set of maximal couples in
$\Mlh$ (resp., $\Muh$). The segments of couples in $\Mlhmax\cup\Muhmax$ confine
a convex region $\Omega$ within the circle $O^*$. Consider the sets
$Q:=\cup(\Veven(\pi)\,\colon \pi\in \Mlhmax)$ and $Q':=\cup(\Veven(\pi)\,\colon
\pi\in \Muhmax)$; we order the vertices in each of them from left to right.
Property~\refeq{Vqpi} and the equalities $|Y|=|Y'|=2|M|$ lead to the following
relations:
   \begin{equation} \label{eq:QQp}
|Q|=|\Mlh|=\frac12(|Y|-|\Mver|)=\frac12(|Y'|-|\Mver|)=|\Muh|=|Q'|.
   \end{equation}

So $|Q|=|Q'|$. We draw a directed edge from each vertex of the sequence $Q$ to
the corresponding vertex of $Q'$. These (pairwise non-crossing) edges are
called \emph{middle bridges}.
\smallskip

\emph{Step 5.} ~When $\Mver\ne \emptyset$, the graph $G'$ constructed during
the previous steps need not be planar since some middle bridges may intersect
vertical segments. The final step transforms $G'$ within small neighborhoods of
such intersection points.

More precisely, for $\pi=ij\in\Mver$, the (directed) vertical segment $L_\pi$
goes from the source $s_i$ to the sink $t_j$ and lies in the convex region
$\Omega$ (defined above). The set $B_\pi$ of edges of $G'$ intersecting $L_\pi$
consists of some middle bridges. (One can see that $i-j$ is even and that
$B_\pi=\emptyset$ if $i=j$.) Let $z_{\pi,b}$ denote the intersection point of
$L_\pi$ and $b\in B_\pi$. Also for a middle bridge $b$, we denote by $R_b$ the
set of vertical segments intersecting $b$.

To transform $G'$ into the desired graph $G$, we first turn each vertical
segment $L_{\pi=ij}$ into the directed path (going from $s_i$ to $t_j$) whose
inner vertices are the points $z_{\pi,b}$ for $b\in B_\pi$, and similarly turn
each middle bridge $b$ (directed ``upwards'') into the directed path whose
inner vertices are the points $z_{\pi,b}$ for $L_\pi\in R_b$. Next we
iteratively modify the graph as follows. At each iteration, choose a vertex
$z=z_{\pi,b}$ in the current graph, split $z$ into two vertices $z'$ and $z''$,
and connect them by edge $e_{\pi,b}$ from $z'$ to $z''$, called the \emph{extra
edge} generated by $\pi,b$.

Geometrically, we choose $z',z''$ to be two points in the segment $b$ within a
small neighborhood of $z$ so that $z'$ lies below $z''$. Then $b$ (regarded as
path) is modified in a natural way: if $b$ is of the form $\ldots,e,z,\tilde
e,\ldots$ (where $e$ and $\tilde e$ are the edges in $b$ entering and leaving
$z$, respectively), then we make $e$ enter $z'$ and make $\tilde e$ leave
$z''$; this turns $b$ into the directed path $\ldots,e,z',e_{\pi,b}, z'',\tilde
e,\ldots$ . The local transformation of the path $L:=L_\pi$ at $z$ is
different: if $L$ is of the form $\ldots,e,z,\tilde e,\ldots$, we make $e$
entering $z''$, and $\tilde e$ leaving $z'$, obtaining the non-directed path
$\ldots,e,z'',e_{\pi,b},z',\tilde e,\ldots$ (in which $e_{\pi,b}$ has the
backward direction). Geometrically, $L$ turns into a zigzag-shaped line. The
transformation at $z=z_{\pi,b}$ is illustrated in the picture:
 \begin{center}
  \unitlength=1mm
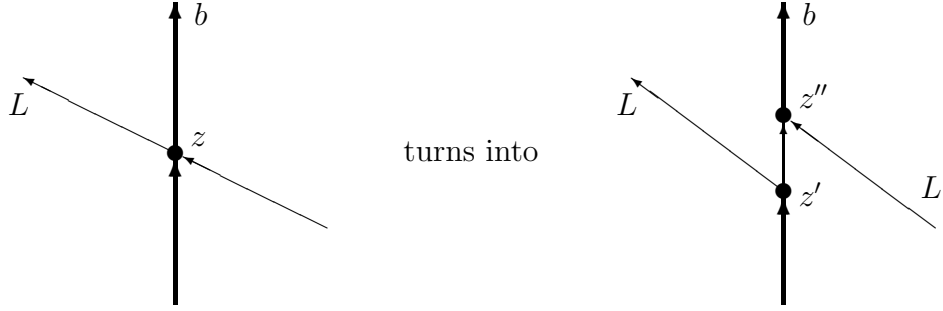
  \begin{picture}(120,38)
 \put(20,20){\circle*{2}}
 \put(20,20){\vector(-2,1){20}}
 \put(40,10){\vector(-2,1){19}}
 \put(-2,25){$L$}
 \put(22.5,37){$b$}
 \put(22,21){$z$}
{\thicklines
 \put(19.9,0){\vector(0,1){19}}
 \put(19.9,20){\vector(0,1){20}}
 \put(20.1,0){\vector(0,1){19}}
 \put(20.1,20){\vector(0,1){20}}
}
 \put(50,19){turns into}
 \put(100,15){\circle*{2}}
 \put(100,25){\circle*{2}}
 \put(100,15){\vector(-4,3){20}}
 \put(120,10){\vector(-4,3){19}}
 \put(100,15){\vector(0,1){9}}
 \put(78,25){$L$}
 \put(118,14){$L$}
 \put(102.5,37){$b$}
 \put(102,26){$z''$}
 \put(102,13){$z'$}
{\thicklines
 \put(99.9,0){\vector(0,1){14}}
 \put(99.9,25){\vector(0,1){15}}
 \put(100.1,0){\vector(0,1){14}}
 \put(100.1,25){\vector(0,1){15}}
}
  \end{picture}
   \end{center}

Eventually we obtain the desired graph $G$. We refer to the edges of $G$
generated by vertical segments (resp., middle bridges) of $G'$ and different
from extra edges as \emph{v-edges} (resp., \emph{b-edges}). Thus, under the
transformation $G'\mapsto G$, each middle bridge $b$ turns into a directed path
with $|R_b|+1$ ~b-edges and $|R_b|$ extra edges which alternate. In its turn,
each vertical segment $L=L_{\pi}$ turns into a ``zigzag'' path with $|B_\pi|+1$
~v-edges and $|B_\pi|$ extra edges; these edges alternate and are,
respectively, the forward and backward edges in the path.

We will distinguish between two sorts of edges in $G$, referring to the lower
and upper bridges and b-edges as \emph{thick} edges, and to the remaining edges
as \emph{thin} ones. The picture below illustrates the construction for an
instance of $M$. Here $n=7$, $\Mlh=\{16,23,45\}$, $\Muh=\{12,47,56\}$ and
$\Mver=\{73\}$, and the left fragment shows the segment representation of $M$
after Step~1. The graph $G$ is drawn in the right fragment where the dark
circles indicate even vertices and those formed by splitting, the light circles
indicate nonterminal odd vertices, and the thin and thick edges are as defined
above.
 \begin{center}
  \unitlength=1mm
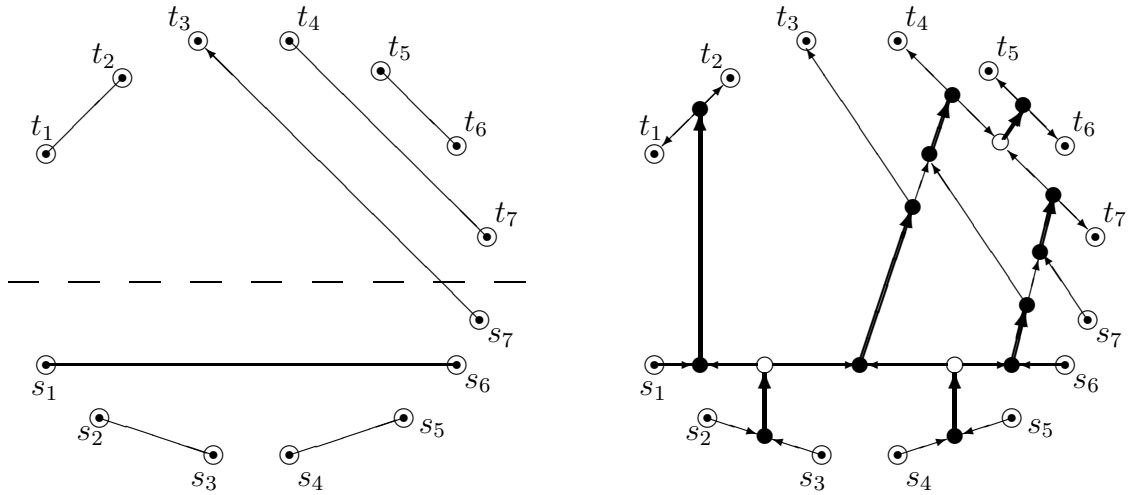
  \begin{picture}(140,65)
 \put(0,17){\circle{2.5}}
 \put(0,17){\circle*{1}}
 \put(7,10){\circle{2.5}}
 \put(7,10){\circle*{1}}
 \put(22,5){\circle{2.5}}
 \put(22,5){\circle*{1}}
 \put(32,5){\circle{2.5}}
 \put(32,5){\circle*{1}}
 \put(47,10){\circle{2.5}}
 \put(47,10){\circle*{1}}
 \put(54,17){\circle{2.5}}
 \put(54,17){\circle*{1}}
 \put(57,23){\circle{2.5}}
 \put(57,23){\circle*{1}}
 \put(0,45){\circle{2.5}}
 \put(0,45){\circle*{1}}
 \put(10,55){\circle{2.5}}
 \put(10,55){\circle*{1}}
 \put(20,60){\circle{2.5}}
 \put(20,60){\circle*{1}}
 \put(32,60){\circle{2.5}}
 \put(32,60){\circle*{1}}
 \put(44,56){\circle{2.5}}
 \put(44,56){\circle*{1}}
 \put(54,46){\circle{2.5}}
 \put(54,46){\circle*{1}}
 \put(58,34){\circle{2.5}}
 \put(58,34){\circle*{1}}
 \put(-2,13){$s_1$}
 \put(4,7){$s_2$}
 \put(19,1){$s_3$}
 \put(33,1){$s_4$}
 \put(49,8){$s_5$}
 \put(55,14){$s_6$}
 \put(58,20){$s_7$}
 \put(-2,47.5){$t_1$}
 \put(6,57){$t_2$}
 \put(16,62){$t_3$}
 \put(33,62){$t_4$}
 \put(45,58){$t_5$}
 \put(55,48){$t_6$}
 \put(59,36){$t_7$}
 \put(0,17){\line(1,0){54}}
 \put(7,10){\line(3,-1){15}}
 \put(32,5){\line(3,1){15}}
 \put(0,45){\line(1,1){10}}
 \put(32,60){\line(1,-1){26}}
 \put(44,56){\line(1,-1){10}}
 \put(57,23){\vector(-1,1){36}}
 \multiput(-5,28)(8,0){9}%
 {\line(1,0){4}
 }
 \put(80,17){\circle{2.5}}
 \put(80,17){\circle*{1}}
 \put(134,17){\circle{2.5}}
 \put(134,17){\circle*{1}}
 \put(86,17){\circle*{2}}
 \put(107,17){\circle*{2}}
 \put(127,17){\circle*{2}}
 \put(94.5,17){\circle{2}}
 \put(119.5,17){\circle{2}}
 \put(80,17){\vector(1,0){5}}
 \put(93.5,17){\vector(-1,0){6.5}}
 \put(95.5,17){\vector(1,0){11}}
 \put(118.5,17){\vector(-1,0){10.5}}
 \put(120.5,17){\vector(1,0){6}}
 \put(134,17){\vector(-1,0){6}}
 \put(87,10){\circle{2.5}}
 \put(87,10){\circle*{1}}
 \put(102,5){\circle{2.5}}
 \put(102,5){\circle*{1}}
 \put(94.5,7.5){\circle*{2}}
 \put(87,10){\vector(3,-1){6.5}}
 \put(102,5){\vector(-3,1){6.5}}
 \put(112,5){\circle{2.5}}
 \put(112,5){\circle*{1}}
 \put(127,10){\circle{2.5}}
 \put(127,10){\circle*{1}}
 \put(119.5,7.5){\circle*{2}}
 \put(112,5){\vector(3,1){6.5}}
 \put(127,10){\vector(-3,-1){6.5}}
 \put(80,45){\circle{2.5}}
 \put(80,45){\circle*{1}}
 \put(90,55){\circle{2.5}}
 \put(90,55){\circle*{1}}
 \put(86,51){\circle*{2}}
 \put(86,51){\vector(-1,-1){5}}
 \put(85,50){\vector(1,1){4}}
 \put(124,56){\circle{2.5}}
 \put(124,56){\circle*{1}}
 \put(134,46){\circle{2.5}}
 \put(134,46){\circle*{1}}
 \put(129,51){\vector(-1,1){4}}
 \put(128,52){\vector(1,-1){5}}
 \put(128.5,51.5){\circle*{2}}
 \put(112,60){\circle{2.5}}
 \put(112,60){\circle*{1}}
 \put(138,34){\circle{2.5}}
 \put(138,34){\circle*{1}}
 \put(119.2,52.8){\circle*{2}}
 \put(125.5,46.5){\circle{2}}
 \put(132.5,39.5){\circle*{2}}
 \put(119.2,52.8){\vector(-1,1){6}}
 \put(119.2,52.8){\vector(1,-1){5.5}}
 \put(132.5,39.5){\vector(-1,1){6}}
 \put(132,40){\vector(1,-1){5}}
 \put(137,23){\circle{2.5}}
 \put(137,23){\circle*{1}}
 \put(100,60){\circle{2.5}}
 \put(100,60){\circle*{1}}
 \put(114,38){\circle*{2}}
 \put(116.2,45){\circle*{2}}
 \put(129,25){\circle*{2}}
 \put(130.7,32){\circle*{2}}
 \put(137,23){\vector(-2,3){5.8}}
 \put(129,25){\vector(-2,3){12.5}}
 \put(114,38){\vector(-2,3){14}}
 \put(129,25){\vector(1,4){1.5}}
 \put(114,38){\vector(1,3){2}}
{\thicklines
 \put(85.9,17){\vector(0,1){33.2}}
 \put(94.4,7.5){\vector(0,1){8.5}}
 \put(119.4,7.5){\vector(0,1){8.5}}
 \put(126.8,17){\vector(1,4){1.8}}
 \put(106.8,17){\vector(1,3){6.8}}
 \put(130.5,32){\vector(1,4){1.8}}
 \put(116.2,45){\vector(1,3){2.3}}
 \put(125.8,47){\vector(2,3){2.5}}
 \put(86.1,17){\vector(0,1){33.2}}
 \put(94.6,7.5){\vector(0,1){8.5}}
 \put(119.6,7.5){\vector(0,1){8.5}}
 \put(127.1,17){\vector(1,4){1.8}}
 \put(107.1,17){\vector(1,3){6.8}}
 \put(130.8,32){\vector(1,4){1.8}}
 \put(116.5,45){\vector(1,3){2.3}}
 \put(126.1,47){\vector(2,3){2.5}}
}
 \put(78,13){$s_1$}
 \put(84,7){$s_2$}
 \put(99,1){$s_3$}
 \put(113,1){$s_4$}
 \put(129,8){$s_5$}
 \put(135,14){$s_6$}
 \put(138,20){$s_7$}
 \put(78,47.5){$t_1$}
 \put(86,57){$t_2$}
 \put(96,62){$t_3$}
 \put(113,62){$t_4$}
 \put(125,58){$t_5$}
 \put(135,48){$t_6$}
 \put(139,36){$t_7$}
  \end{picture}
   \end{center}

Note that the obtained $G$ is acyclic (as all edges not contained in
``horizontal segments'' are ``directed upwards''). Also we will take advantages
from the following features of $G$ which can be seen from the above
construction:
  \begin{numitem1}
  \begin{itemize}
\item[(i)] Each source has one leaving edge and no entering edge, whereas each
sink has one entering edge and no leaving edge;
\item[(ii)] Each inner (i.e., nonterminal) vertex is of degree 3, and it has
either two thin entering edges and one thick leaving edge, or two thin leaving
edges and one thick entering edge;
\item[(iii)] The connected components of the subgraph of $G$ induced by the
thin edges correspond to the lower horizontal paths $L_\pi$ for $\pi\in\Mlh$,
the upper horizontal paths $L_\pi$ for $\pi\in\Muh$, and the (straight or
zigzag) paths $L_\pi$ for $\pi\in\Mver$, each of these paths having alternately
directed edges.
  \end{itemize}
  \label{eq:propG}
  \end{numitem1}
It will be convenient to represent each thin path $L_\pi$ as the union of two
matchings (one being formed by the forward edges, and the other by the backward
edges), denoted by $N^1_\pi, N^2_\pi$. Also we denote the set of thick edges
entering (leaving) vertices of $L_\pi$ by $\Zin_\pi$ (resp., $\Zout_\pi$). In
particular, $\Zin_\pi$ is the set of lower bridges for $\pi$ when $\pi\in
\Mlh$, and $\Zout_\pi$ is the set of upper bridges for $\pi$ when $\pi\in
\Muh$.

We assert that $G$ satisfies properties~(P1) and~(P2) for the given $M$. To
show this, we consider a proper pair $(C,C')$ for $(Y,Y')$ and argue as
follows. \smallskip

(a) Suppose $M\in\Mscr(C,C')$. Take the subgraph $F$ of $G$ induced by the edge
set $U$ consisting of all thick edges and the following thin edges. For each
$\pi\in \Mlh$, ~$U$ includes exactly one of the matchings $N^1_\pi,N^2_\pi$ in
$L_\pi$, namely, the one containing the edge leaving the source $s_i$, where
$i$ is the element of $\pi\cap C$ (which is unique since $M$ is feasible for
$(C,C')$). Similarly, for each $\pi\in\Muh$, ~$U$ includes the matching in
$L_\pi$ that contains the edge entering the sink $t_j$, where $\{j\}=\pi\cap
C'$. And for each $\pi=ij=\Mver$, if $i\in C$ (and therefore, $j\in C'$), then
$U$ includes the matching in $L_\pi$ covering both $s_i$ and $t_j$ (which is
formed by v-edges), whereas if $i\notin C$ (and $j\notin C'$), then $U$
includes the matching formed by extra edges (which may be empty).

Using~\refeq{propG} and the fact that $G$ is acyclic, one can conclude that $F$
consists of pairwise disjoint directed paths going from $S_C$ to $T_{C'}$,
i.e., $F$ is a $(C|C')$-flow in $G$. Acting similarly w.r.t. $\bar C$ and $\bar
C\,'$, we construct a $(\bar C|\bar C\,')$-flow $F'$ in $G$. \smallskip

(b) Next we show that in case $M\in\Mscr(C,C')$ the flows $F$ and $F'$ as above
are unique. Consider an arbitrary flow $\tilde F$ from some sources to some
sinks in $G$. From~\refeq{propG} it easily follows that for each $\pi\in M$,
   \begin{numitem1}
 \begin{itemize}
\item[(i)] If $\tilde F\cap L_\pi$ is a matching $N^\alpha_\pi$, $\alpha\in\{1,2\}$,
then $\tilde F$ contains both $\Zin_\pi,\Zout_\pi$;
\item[(ii)] Conversely, if $\tilde F$ contains a set $Z\in\{\Zin_\pi,\Zout_\pi\}$
and if $Z\ne\emptyset$, then $\tilde F\cap L_\pi$ is exactly one of
$N^1_\pi,N^2_\pi$ (regarding these objects as edge sets).
  \end{itemize}
  \label{eq:thin-thick}
  \end{numitem1}

We explain that~\refeq{thin-thick} determines $\tilde F$ uniquely if $\tilde F$
is a $(C|C')$-flow. Indeed, from the construction of $G$ it easily follows that
there is an ordering $\pi(1),\ldots,\pi(m)$ ($m=|Y|$) of the couples in $M$
such that for $k=1,\ldots,m$, at least one of the set $\Zin_{\pi(k)}$ and
$\Zout_{\pi(k)}$ is entirely contained in $\cup_{d=1}^{k-1}(\Zin_{\pi(d)}\cup
\Zout_{\pi(d)})$ (which is automatically holds when $\pi(k)$ is a minimal
couple in $\Mlh\cup\Muh$ since some of $\Zin_{\pi(k)},\Zout_{\pi(k)}$ is
empty).

Now we argue as follows. If $\pi(k)$ is a minimal couple in $\Mlh$ and if
$\{i\}=\pi(k)\cap C$, then each of $N^1_{\pi(k)},N^2_{\pi(k)}$ consists of a
single edge and, obviously, $\tilde F$ contains exactly one of them, namely,
the edge incident to $s_i$. Applying~\refeq{thin-thick}(i) to this $\pi(k)$, we
obtain that $\tilde F$ contains $\Zout_{\pi(k)}$ (as well as
$\Zin_{\pi(k)}=\emptyset$). Similarly, if $\pi(k)$ is a minimal couple in
$\Muh$, then $\tilde F$ is determined within $L_{\pi(k)}$ and contains
$\Zin_{\pi(k)}$ (and $\Zout_{\pi(k)}=\emptyset$). In a general case, assume by
induction that for $d=1,\ldots,k-1$, ~$\tilde F$ is determined on $L_{\pi(d)}$
and contains $\Zin_{\pi(d)}\cup \Zout_{\pi(d)}$. Then, due to the above
ordering, $\tilde F$ contains at least one of $\Zin_{\pi(k)},\Zout_{\pi(k)}$.
Hence, by~\refeq{thin-thick}(ii), ~$\tilde F\cap L_{\pi(k)}$ is
$N^\alpha_{\pi(k)}$ for some $\alpha\in\{1,2\}$. (This remains true when
$\Zin_{\pi(k)}\cup \Zout_{\pi(k)}=\emptyset$, which is possible only if
$\pi(k)=ij\in\Mver$ and $i=j$.) Moreover, $\alpha$ is determined by considering
the end vertices (terminals) of $L_{\pi(k)}$ and checking which of them (or
none, or both) belongs to $S_C\cup T_{C'}$ (since such a terminal must be
covered by $N^\alpha_{\pi(k)}$). Now~\refeq{thin-thick}(i) enables us to
conclude with $\Zin_{\pi(k)}\cup\Zout_{\pi(k)}\subseteq \tilde F$, justifying
the induction.

So $\tilde F=F$. The uniqueness of a $(\bar C|\bar C\,')$-flow is shown
similarly. This yields~(P1). \smallskip

(c) To check~(P2), consider a proper pair $(C,C')$ for $(Y,Y')$ such that there
exist both a $(C|C')$-flow $F$ and a $(\bar C|\bar C\,')$-flow $F'$ in $G$. Our
goal is to show that $M\in\Mscr(C,C')$, i.e., the following hold: (c1)
~$|\pi\cap C|=1$ for $\pi\in\Mlh$; (c2) ~$|\pi\cap C'|=1$ for $\pi\in\Muh$; and
(c3) ~$(i\in C) \Leftrightarrow (j\in C')$ for $\pi=ij\in\Mver$. We consider
the above ordering $\pi(1),\ldots,\pi(r)$ on $M$ and use induction on $k$,
assuming that the corresponding relation among (c1)--(c3) holds for each
$\pi(d)$ with $d<k$.

When $\pi(k)$ is a minimal couple in $\Mlh$, ~$\pi(k)\subseteq C$ would imply
that $F$ contains both edges of the 2-edge path $L_{\pi(k)}$, which is
impossible (since these edges enter the same vertex). For a similar reason,
$\pi(k)\cap C=\emptyset$ would imply the nonexistence of $F'$. So $\pi(k)$
satisfies~(c1). Similarly a minimal couple in $\Muh$ satisfies~(c2). Also if
$\pi(k)=ij\in\Mver$ and $i=j$, then $L_{\pi(k)}$ consists of a single edge,
and~(c3) is trivial. For a general $k$, using induction and arguing as in
part~(b), we may assume that there is a nonempty set among
$\Zin_{\pi(k)},\Zout_{\pi(k)}$ which is contained in both $F,F'$. Then there
are $\alpha,\beta\in\{1,2\}$ such that $F\cap L_{\pi(k)}=N^\alpha_{\pi(k)}$ and
$F'\cap L_{\pi(k)}=N^\beta_{\pi(k)}$. The matchings $N^\alpha_{\pi(k)},
N^\beta_{\pi(k)}$ determine the location of the end vertices (terminals) $u,v$
of $L_{\pi(k)}$ w.r.t. $C,C'$ and their complements to $Y,Y'$, and each of
$u,v$ is related to exactly one of $C\cup C'$ and $\bar C\cup\bar C\,'$. This
implies $\alpha\ne \beta$, and validity of (c1)--(c3) for $\pi(k)$ follows.
Finally,~\refeq{thin-thick}(i) provides that each of $F,F'$ contains both
$\Zin_{\pi(k)},\Zout_{\pi(k)}$, completing the proof of~(P2).
\medskip

It remains to consider the situation when some of $X,X'$ or both are nonempty.
It reduces to the previous case by replacing each element of $X$ ($X'$) by a
couple of elements in $Y$ (resp., $Y'$) and adding such couples to the matching
in question. More precisely, an element $i\in X$ is replaced by consecutive
elements $i',i''$ added to $Y$ (which are inserted into the linearly ordered
set $X\cup Y$ in place of $i$). Similarly, an element $j\in X'$ is replaced by
consecutive elements $j',j''$ added to $Y'$. Then the resulting sets $\tilde Y$
and $\tilde Y'$ have the same size, equal to $|Y|+2|X|=|Y'|+2|X'|$.
Accordingly, we extend each planar perfect matching $M$ on $Y\sqcup Y'$ to a
planar perfect matching $\tilde M$ on $\tilde Y\sqcup \tilde Y'$ by adding the
lower (upper) horizontal couple $\pi^i=i'i''$ for each $i\in X$ (resp.,
$\pi^j=j'j''$ for each $j\in X'$). Note that the added couples are minimal for
the corresponding partial orders, and the pairs $(\tilde C\subseteq \tilde
Y,\tilde C'\subseteq \tilde Y')$ having $\tilde M$ as a feasible matching are
exactly those obtained from the pairs $(C\subseteq Y,C'\subseteq Y')$
satisfying $M\in\Mscr(C,C')$ by adding to $C$ one element from $\{i',i''\}$ for
each $i\in X$, and adding to $C'$ one element from $\{j',j''\}$ for each $j\in
X'$.

Let $\tilde G$ be the graph obtained by applying the previous construction to
such an $\tilde M$. Then each couple $\pi^i$, $i\in X$, generates the 2-edge
path $L_{\pi^i}$ connecting the sources $s_{i'},s_{i''}$, and similarly for
$X'$ and sinks. Shrinking each $L_{\pi^i}$ into one point, regarded as the
source $s_i$ when $i\in X$ and as the sink $t_i$ when $i\in X'$, we obtain the
desired graph $G$ for $X,X',Y,Y'$ and $M$. It is straightforward to verify that
properties~(P1),(P2) for $\tilde G,\tilde Y,\tilde Y',\tilde M$ imply those for
$G,X,X',Y,Y',M$.

This completes the proof of Theorem~\ref{tm:nonbalance}, and
Theorem~\ref{tm:main} follows. \hfill$\qed$
 \end{proof}

Note that any FG-function $f=f_w$ on $\Escr^{n,n'}$ obtained by the
construction in the above proof takes nonnegative integer values (since the
weighting $w$ is nonnegative). This together with the fact that the function of
minors of a totally nonnegative matrix is an FG-function gives the following:
  \begin{corollary} \label{cor:nn_matr}
~2-patterns $\Ascr_0,\Bscr_0\Subset\Pi_{m,m'}$ are balanced if and only if the
corresponding quadratic relations (concerning any $n,n',X,X',Y,Y'$ as above)
hold for any function $f:\Escr^{n,n'}\to\Zset_{\ge 0}$ which is the function of
minors of a totally nonnegative $n'\times n$ matrix. Furthermore, when
$\Ascr_0,\Bscr_0\Subset\Pi_{m,m'}$ are not balanced, for any corresponding
$n,n',X,X',Y,Y'$, there exists, and can be explicitly constructed, a totally
nonnegative $n'\times n$ matrix  such that the function $f$ of its minors obeys
inequality~\refeq{non-balan} (where $\Ascr:=\gamma_{Y,Y'}(\Ascr_0)$ and
$\Bscr:=\gamma_{Y,Y'}(\Bscr_0)$).
   \Xcomment{

For $Y\subseteq [n]$ and $Y'\subseteq [n']$, 2-families
$\Ascr,\Bscr\Subset\Pi_{Y,Y'}$ are balanced if and only if the corresponding
quadratic relations associated to $\Ascr,\Bscr$ hold for any function
$f:\Escr^{n,n'}\to\Zset_{\ge 0}$ which is the function of minors of a totally
nonnegative $n'\times n$ matrix.
  }
  \end{corollary}


\section{\Large Applications to Schur functions}  \label{sec:Schur}

It is well known that Schur functions (polynomials) are expressed as minors of
a certain matrix, by Jacobi--Trudi's formula. Therefore, these functions
satisfy many quadratic relations, in particular, ones of Plucker type.
In~\cite{Flm,FK} and some other works (see a discussion in~\cite{Flm}) one
shows how to establish quadratic relations for ordinary and skew Schur
functions by use of a lattice paths method based on the Gessel--Viennot
interpretation of semistandard Young tableaux~\cite{GV}. This lattice path
method is, in fact, a specialization to a particular planar network of the flow
approach that we described in Sections~\ref{sec:flow},\ref{sec:balan}. Below we
give a brief discussion on this subject.

Recall that a \emph{partition} of length $r$ is an $r$-tuple $\lambda$ of
weakly decreasing nonnegative integers $\lambda_1\ge\lambda_2\ge \ldots\ge
\lambda_r$. The \emph{Ferrers diagram} of $\lambda$ is meant to be the array
$F_\lambda$ of cells with $r$ left-aligned rows containing $\lambda_i$ cells in
$i$th row (the row indices will grow from the bottom to the top). For
$N\in\Nset$, an $N$-\emph{semistandard Young tableau} of shape $\lambda$ is a
filling $T$ of $F_\lambda$ with numbers from $[N]$ so that the numbers weakly
increase in each row and strictly increase in each column. We associate to $T$
the monomial $x^T$ that is the product of variables $x_1,\ldots,x_N$, each
$x_k$ being taken in the degree equal to the number of occurrences of $k$ in
$T$. Then the \emph{Schur function} for $\lambda$ and $N$ is the polynomial
   $$
   s_\lambda=s_\lambda(x_1,\ldots,x_N):=\sum\nolimits_T x^T,
   $$
where the sum is over all $N$-semistandard Young tableaux of shape $\lambda$.
Besides, one often considers a \emph{skew} Schur function $s_{\lambda/\mu}$,
where $\mu$ is an $r$-partition with $\mu_i\le\lambda_i$; it is defined in a
similar way w.r.t. the skew Ferrers diagram $F_{\lambda/\mu}$ obtained by
removing from $F_\lambda$ the cells of $F_\mu$, along with its semistandard
fillings. When needed, an ``ordinary'' diagram $F_\lambda$ is regarded as the
skew one $F_{\lambda/\mu}$, where $\mu=(0,\ldots,0)$, and similarly for
tableaux.

There is an important one-to-one correspondence between the $r$-partitions
$\lambda$ and the $r$-element subsets $A_\lambda$ of the set $\Zset_{>0}$ of
positive integers (or a set $[n]$ for $n\ge \lambda_1+r$). This is given by
  \begin{equation} \label{eq:A_lambda}
  \lambda=(\lambda_1\ge\ldots\ge\lambda_r) \Longleftrightarrow
  A_\lambda:=\{\lambda_r+1,\lambda_{r-1}+2,\ldots,\lambda_1+r\}.
  \end{equation}

Let us form the directed square grid $\Gamma=\Gamma(N)$ whose vertices are the
points $(i,j)$ for $i\in\Zset_{>0}$ and $j\in[N]$ and whose edges $e$ are
directed upwards or to the right, i.e., $e=((i,j),(i,j+1))$ or
$((i,j),(i+1,j))$ (instead, one can take a finite truncation of this grid). The
vertices $s_i:=(i,1)$ and $t_i:=(i,N)$ are regarded as the sources and sinks in
$\Gamma$, respectively, and we assign to each horizontal edge $e$ at level $h$
the weight to be the indeterminate $x_h$:
  \begin{equation} \label{eq:w-x}
  w(e):=x_h \qquad\mbox{for ~$e=((i,h),(i+1,h))$, ~$i\in\Zset_{>0}$,
  ~$h=1,\ldots,N$.}
  \end{equation}

Now using the Gessel--Viennot model~\cite{GV} (in a slightly different form),
one can associate to an $N$-semistandard skew Young tableau $T$ with shape
$\lambda/\mu$ the system $\Pscr_T=(P_1,\ldots,P_r)$ of directed paths in
$\Gamma$, where for $k=1,\ldots,r$:
  \begin{numitem1}
~$P_k$ corresponds to $(r+1-k)$th row of $T$; it goes from the source
$s_{k+\mu_{r+1-k}}$ to the sink $t_{k+\lambda_{r+1-k}}$; and for
$h=1,\ldots,N$, the number of horizontal edges of $P_k$ at level $h$ is equal
to the number of occurrences of $h$ in the $k$th row of $T$.
  \label{eq:shur-paths}
  \end{numitem1}

So the sources used in $\Pscr_T$ are the $s_i$ for $i\in A_\mu$, and the sinks
are the $t_j$ for $j\in A_\lambda$. Observe that the semistandardness of $T$
implies that these paths are pairwise disjoint, i.e., $\Pscr_T$ is an
$(A_\mu|A_\lambda)$-flow in $\Gamma$. One can see the converse as well: if
$\Pscr$ is an $(A_\mu|A_\lambda)$-flow in $\Gamma$, then the filling $T$ of
$F_{\lambda/\mu}$ determined, in a due way, by the horizontal edges of paths in
$\Pscr$ is just a semistandard skew Young tableau, and one has $\Pscr_T=\Pscr$.
This gives a nice bijection between corresponding flows and tableaux. The next
picture illustrates an example of a semistandard Young tableau $T$ with $N=6$,
$r=5$, $\lambda=(6,5,3,3,2)$ and $\mu=(2,2,1,1,0)$, and its corresponding flow
$\Pscr_T=(P_1,\ldots,P_5)$.

 \begin{center}
  \unitlength=1mm
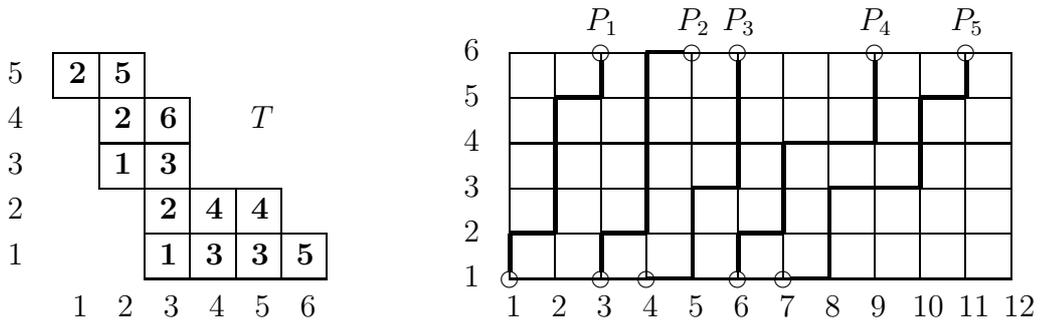
  \begin{picture}(140,40)
  \put(22,4){\line(1,0){24}}
  \put(22,10){\line(1,0){24}}
  \put(16,16){\line(1,0){24}}
  \put(16,22){\line(1,0){12}}
  \put(10,28){\line(1,0){18}}
  \put(10,34){\line(1,0){12}}
  \put(10,28){\line(0,1){6}}
  \put(16,16){\line(0,1){18}}
  \put(22,4){\line(0,1){30}}
  \put(28,4){\line(0,1){24}}
  \put(34,4){\line(0,1){12}}
  \put(40,4){\line(0,1){12}}
  \put(46,4){\line(0,1){6}}
  \put(12.5,-1){1}
  \put(18.5,-1){2}
  \put(24.5,-1){3}
  \put(30.5,-1){4}
  \put(36.5,-1){5}
  \put(42.5,-1){6}
  \put(4,6){1}
  \put(4,12){2}
  \put(4,18){3}
  \put(4,24){4}
  \put(4,30){5}
  \put(24,6){{\bf 1}}
  \put(30,6){{\bf 3}}
  \put(36,6){{\bf 3}}
  \put(42,6){{\bf 5}}
  \put(24,12){{\bf 2}}
  \put(30,12){{\bf 4}}
  \put(36,12){{\bf 4}}
  \put(18,18){{\bf 1}}
  \put(24,18){{\bf 3}}
  \put(18,24){{\bf 2}}
  \put(24,24){{\bf 6}}
  \put(12,30){{\bf 2}}
  \put(18,30){{\bf 5}}
  \put(36,24){$T$}
  \put(70,4){\line(1,0){66}}
  \put(70,10){\line(1,0){66}}
  \put(70,16){\line(1,0){66}}
  \put(70,22){\line(1,0){66}}
  \put(70,28){\line(1,0){66}}
  \put(70,34){\line(1,0){66}}
  \put(70,4){\line(0,1){30}}
  \put(76,4){\line(0,1){30}}
  \put(82,4){\line(0,1){30}}
  \put(88,4){\line(0,1){30}}
  \put(94,4){\line(0,1){30}}
  \put(100,4){\line(0,1){30}}
  \put(106,4){\line(0,1){30}}
  \put(112,4){\line(0,1){30}}
  \put(118,4){\line(0,1){30}}
  \put(124,4){\line(0,1){30}}
  \put(130,4){\line(0,1){30}}
  \put(136,4){\line(0,1){30}}
  \put(70,4){\circle{2}}
  \put(82,4){\circle{2}}
  \put(88,4){\circle{2}}
  \put(100,4){\circle{2}}
  \put(106,4){\circle{2}}
  \put(82,34){\circle{2}}
  \put(94,34){\circle{2}}
  \put(100,34){\circle{2}}
  \put(118,34){\circle{2}}
  \put(130,34){\circle{2}}
  \put(69.5,-1){1}
  \put(75.5,-1){2}
  \put(81.5,-1){3}
  \put(87.5,-1){4}
  \put(93.5,-1){5}
  \put(99.5,-1){6}
  \put(105.5,-1){7}
  \put(111.5,-1){8}
  \put(117.5,-1){9}
  \put(123,-1){10}
  \put(129,-1){11}
  \put(135,-1){12}
  \put(64,3){1}
  \put(64,9){2}
  \put(64,15){3}
  \put(64,21){4}
  \put(64,27){5}
  \put(64,33){6}
  \put(80,37){$P_1$}
  \put(92,37){$P_2$}
  \put(98,37){$P_3$}
  \put(116,37){$P_4$}
  \put(128,37){$P_5$}
\thicklines{
  \put(88,4){\line(1,0){6}}
  \put(107,4){\line(1,0){5}}
  \put(70,10){\line(1,0){6}}
  \put(82,10){\line(1,0){6}}
  \put(100,10){\line(1,0){6}}
  \put(94,16){\line(1,0){6}}
  \put(112,16){\line(1,0){12}}
  \put(106,22){\line(1,0){12}}
  \put(76,28){\line(1,0){6}}
  \put(124,28){\line(1,0){6}}
  \put(88,34){\line(1,0){5}}
  \put(88,4.2){\line(1,0){6}}
  \put(107,4.2){\line(1,0){5}}
  \put(70,10.2){\line(1,0){6}}
  \put(82,10.2){\line(1,0){6}}
  \put(100,10.2){\line(1,0){6}}
  \put(94,16.2){\line(1,0){6}}
  \put(112,16.2){\line(1,0){12}}
  \put(106,22.2){\line(1,0){12}}
  \put(76,28.2){\line(1,0){6}}
  \put(124,28.2){\line(1,0){6}}
  \put(88,34.2){\line(1,0){5}}
  \put(70,5){\line(0,1){5}}
  \put(76,10){\line(0,1){18}}
  \put(82,28){\line(0,1){5}}
  \put(82,5){\line(0,1){5}}
  \put(88,10){\line(0,1){24}}
  \put(94,4){\line(0,1){12}}
  \put(100,16){\line(0,1){17}}
  \put(100,5){\line(0,1){5}}
  \put(106,10){\line(0,1){12}}
  \put(118,22){\line(0,1){11}}
  \put(112,4){\line(0,1){12}}
  \put(124,16){\line(0,1){12}}
  \put(130,28){\line(0,1){5}}
  \put(70.2,5){\line(0,1){5}}
  \put(76.2,10){\line(0,1){18}}
  \put(82.2,28){\line(0,1){5}}
  \put(82.2,5){\line(0,1){5}}
  \put(88.2,10){\line(0,1){24}}
  \put(94.2,4){\line(0,1){12}}
  \put(100.2,16){\line(0,1){17}}
  \put(100.2,5){\line(0,1){5}}
  \put(106.2,10){\line(0,1){12}}
  \put(118.2,22){\line(0,1){11}}
  \put(112.2,4){\line(0,1){12}}
  \put(124.2,16){\line(0,1){12}}
  \put(130.2,28){\line(0,1){5}}
 }
  \end{picture}
   \end{center}

Note that when $T$ is ``ordinary'' (i.e., $\mu={\bf 0}$), the sources used in
$\Pscr_T$ are $s_1,s_2,\ldots,s_r$. We say that this $\Pscr_T$ is a
\emph{co-flag flow} (it becomes a flag flow if we reverse the edges of $\Gamma$
and swap the sources and sinks).

The above bijection between the $N$-semistandard skew Young tableaux with shape
$\lambda/\mu$ and the $(A_\mu|A_\lambda)$-flows in $\Gamma(N)$ implies that
(ordinary of skew) Schur functions are ``values'' of the flow-generated
function $f_w$ for $\Gamma$ and the weighting $w$ as in~\refeq{w-x}. (It leads
to no confusion that the weights are given on the horizontal edges of $\Gamma$
and belong to a polynomial ring.) This gives rise to establishing quadratic
relations on Schur functions, by properly translating SQ-relations on
FG-functions. Below we give two examples (the reader may try to extend the list
of examples by using SQ-relations from Section~\ref{sec:relat}).
\smallskip

1) A particular relation on ordinary Schur functions with $r=2$ can be derived
from PSQ-relations on quadruples. This reads as
   \begin{equation} \label{eq:schur-ijkl}
s_{(k,i)}s_{(\ell,j)}=s_{(\ell,i)}s_{(k,j)}+s_{(j-1,i)}s_{(\ell,k+1)},
  \end{equation}
where $i<j\le k<\ell$. Letting $(i',j',k',\ell'):=(i+1,j+1,k+2,\ell+2)$ and
$f:=f_w$, one can see that~\refeq{schur-ijkl} turns into
  $$
f([2]|\,i'k')\,f([2]|\,j'\ell')=f([2]|\,i'\ell')\,f([2]|\,j'k')
                             +f([2]|\,i'j')\,f([2]|\,k'\ell'),
  $$
which, in view of $i'<j'<k'<\ell'$, is nothing else than the co-flag
counterpart of the AP4-relation~\refeq{AP4} in case $X=\emptyset$. (Note that
relation~\refeq{schur-ijkl} can be generalized by adding to each 2-component
partition a fixed partition $(\lambda_1,\ldots,\lambda_{r'})$ such that either
$\lambda_{r'}\ge \ell$ or $i\ge\lambda_1$.) \smallskip

2) The next example is shown in~\cite{FK} by use of Dodgson's condensation
formula for matrix minors. It says that a partition
$(\lambda_1,\ldots,\lambda_{r})$ with $\lambda_r>0$ yields
  \begin{equation} \label{eq:schur-dodg}
s_{(\lambda_1,\ldots,\lambda_{r-1})}\,s_{(\lambda_2,\ldots,\lambda_{r})}
=s_{(\lambda_2,\ldots,\lambda_{r-1})}\,s_{(\lambda_1,\ldots,\lambda_{r})}
+s_{(\lambda_2-1,\ldots,\lambda_{r}-1)}\,s_{(\lambda_1+1,\ldots,\lambda_{r-1}+1)}.
  \end{equation}
For each of the six partitions $\lambda^{(i)}$ in this relation, $i=1,\ldots,6$
(from left to right), we take the set $A_{\lambda^{(i)}}$ as
in~\refeq{A_lambda} and form the corresponding subsets $S^{(i)},T^{(i)}$ of
sources and sinks in $\Gamma$, respectively. In addition, for $i=1,3,5$, we
shift each of the sets $S^{(i)},T^{(i)}$ by one position to the right (which
leads to equivalent sets of flows, as well as their weights, in $\Gamma$). Then
we obtain the following six source-sink index pairs (from left to right, as
before), denoting $X:=\{2,\ldots,r-1\}$ and
$X':=\{\lambda_2+r-1,\lambda_3+r-2,\ldots,\lambda_{r-1}+2\}$:
   \begin{gather*}
   (Xr|X'(\lambda_1+r)),\;\; (1X|\,(\lambda_r+1)X'),\;\; (X|X'), \;\;
(1Xr|\,(\lambda_r+1)X'(\lambda_1+r)),\\
     \;\;(Xr|\,(\lambda_r+1)X'),\;\;(1X|X'(\lambda_1+r)).
   \end{gather*}
Now define $i:=1$, ~$k:=r$, ~$i':=\lambda_r+1$, ~$k':=\lambda_1+r$. Then $i<k$,
~$i'<k'$, ~$X\cap\{i,k\}=\emptyset$, ~$X'\cap\{i',k'\}=\emptyset$,
and~\refeq{schur-dodg} turns into the following relation for $f=f_w$:
  $$
  f(Xk|X'k')\,f(iX|\,i'X')=f(X|X')\,f(iXk|\,i'X'k')+f(Xk|\,i'X')\,f(iX|X'k'),
    $$
which is just Dodgson's condensation formula (cf.~\refeq{Dodg_orig}).


\section{\Large FG-functions over a semiring with division}  \label{sec:laurent}

In this section we assume that $\frakS$ is a commutative semiring with
division, i.e., $\frakS$ contains $\underline 1$ and the operation $\odot$ is
invertible (i.e., $(\frakS,\odot)$ is an abelian group). Two important special
cases, mentioned in the Introduction, are: the set $\Rset_{>0}$ of positive
reals, and the tropicalization $\frakL^{\rm trop}$ of a totally ordered abelian
group $\frakL$, in particular, the set $\Rset_{\max}$ of reals with the
operations $\oplus=\max$ and $\odot=+$. It turns out that for such a $\frakS$,
the set $\bf{FG}=\bf{FG}_n(\frakS)$ of flag-flow-generated functions on
$2^{[n]}$ possesses the following nice properties:
   \begin{numitem1}  \label{eq:bas-Laur}
   \begin{itemize}
\item[(i)] All these functions $f$ can be generated by flows in one planar network,
namely, in the half-grid $\Gamma^\triangle_n$ (see Fig.~\ref{fig:Gamma});
\item[(ii)] $\bf{FG}$ coincides with the set of functions $f:2^{[n]}\to \frakS$
satisfying P3-relation~\refeq{SP3} for all $i,j,k,X$ (so~\refeq{SP3} provides
the other PSQ-relations);
\item[(iii)] $\bf{FG}$ has as a basis the set $\Iscr_n$ of intervals in $[n]$
(including the ``empty interval'' $\emptyset$), called the \emph{standard}
basis for $\bf{FG}$;
\item[(iv)] The values of $f$ are expressed as (algebraic or tropical) Laurent
polynomials in its values on $\Iscr_n$.
  \end{itemize}
   \end{numitem1}

\noindent Here a \emph{basis} for this $\bf{FG}$ is meant to be a collection
$\Iscr'\subset 2^{[n]}$ such that the restriction map $f\mapsto
{f}\rest{\Iscr'}$ gives a bijection between $\bf{FG}$ and
$\mathfrak{S}^{\Iscr'}$; in other words, any function in $\bf{FG}$ is
determined by its values on $\Iscr'$, and the latter values can be chosen
arbitrarily in $\mathfrak{S}$.

The facts exhibited in~\refeq{bas-Laur} are discussed in~\cite{DKK1} (mostly
for $\frakS:=\Rset_{\max}$) and in~\cite{BFZ,FZ} (concerning~(iv)). The
arguments given there can be directly extended to an arbitrary $\frakS$ as
above, and below we give a brief outline (which is sufficient to restore the
details with help of~\cite{DKK1}). As before, an interval $\{p,p+1,\ldots,q\}$
in $[n]$ is denoted by $[p..q]$. \medskip

\noindent \textbf{A.} ~An important feature of $\Gamma_n^\triangle=(V,E)$ is
that for any nonempty interval $I=[q..r]$ in $[n]$, there exists exactly one
feasible flow $\phi_I$ from $S_I$ to the sinks $t_1,\ldots,t_{|I|}$; namely,
$\phi_I$ goes through the vertices $(i,j)$ occurring in the rectangle
$[r]\times[r-q+1]$ (more precisely, satisfying $i\le r$, ~$j\le r-q+1$ and
$i\ge j$). Therefore, given a weighting $w:V\to \frakS$, the values of $f=f_w$
on the nonempty intervals $[q..r]$ are viewed as
  \begin{equation} \label{eq:fpq}
  f([q..r])=\bigodot_{j\le i\le r,\; 1\le j\le r-q+1} w(i,j).
  \end{equation}

Note that the number $\frac{n(n+1)}{2}$ of vertices of $\Gamma_n^\triangle$ is
equal to the number of nonempty intervals in $[n]$ and system~\refeq{fpq} is
non-degenerate. So, using the division in $\frakS$, denoted as $\odivide$, we
can in turn express the weights of vertices via the values of $f$ on the
intervals. This is computed as
   \begin{equation}\label{eq:wij}
   w(i,j)= \left\{
    \begin{array}{ll}
  (f(I_{i,j})\odot f(I_{i-1,j-1}))\odivide (f(I_{i-1,j})\odot f(I_{i,j-1})) & \mbox{for $i>j$},\\
   f(I_{i,j})\odivide f(I_{i,j-1}) & \mbox{for $i=j$},
  \end{array}
        \right.
   \end{equation}
denoting by $I_{i',j'}$ the interval $[(i'-j'+1).. i']$ and letting
$f(I_{i',0}):=\underline{1}$.

Thus, the correspondence $w\mapsto f_w$ gives a bijection between the set of
weightings $w:V\to\frakS$ and the set $\frakS^{\Iscr_n^+}$, where $\Iscr_n^+$
denotes the set of nonempty intervals in $[n]$. \medskip

\noindent\textbf{B.} ~We know that for a weighting $w$, the value of $f=f_w$ on
any nonempty subset $A\subseteq [n]$ is represented by a ``polynomial'' in
variables $w(v)$, $v\in V$, namely, by an $\oplus$-sum of products
$\odot(w(v)\,\colon v\in V')$ for some subsets $V'\subseteq V$. Substituting
into this polynomial the corresponding terms from~\refeq{wij}, we obtain an
expression of the form
   $$
   f(A)=\oplus\left( \Pscr_k\,\colon k=1,\ldots,N\right),
   $$
where each $\Pscr_k$ is a ``monomial'' $\odot(f(I)^{\odot\sigma_k(I)} \,\colon
I\in\Iscr_n^+)$ with integer (possibly negative) degrees $\sigma_k(I)$. This
means that $f(A)$ is a Laurent polynomial (w.r.t. the addition $\oplus$ and
multiplication $\odot$) in variables $f(I)$, $I\in\Iscr_n^+$.

(Analyzing possible flows in $\Gamma_n^\triangle$, one can show that the
degrees $\sigma_k(I)$ are bounded and, moreover, belong to $\{-1,0,1,2\}$. This
is proved in~\cite{DKK1} for the tropical case and can be directly extended to
an arbitrary commutative semiring $\frakS$ with division.) \medskip

\noindent \textbf{C.} ~A simple fact (cf.~\cite{DKK1}) is that any function
$f:2^{[n]}\to \Rset$ obeying TP3-relation~\refeq{TP3} is determined by its
values on $\Iscr_n$. The proof of this fact is directly extended to $\frakS$ in
question, as follows (a sketch). If $S\subseteq[n]$ is not an interval, define
$i:=\min(S)$, $k:=\max(S)$, $X:=S-\{i,k\}$, and choose an element  $j$ in
$[i..k]-S$. Then for a function $f$ on $2^{[n]}$ obeying
P3-relation~\refeq{SP3}, the value $f(S)$ is expressed via the values $f(S')$
on five sets $S'=Xi,Xj,Xk,Xij,Xjk$. Since $\max(S')-\min(S')< \max(S)-\min(S)$,
we can apply induction on $\max(S)-\min(S)$.

Using this fact and reasonings above, we obtain that $\Iscr_n$ is indeed a
basis for the functions in $\bf{FG}_n(\frakS)$ and that all these functions are
generated by flows in $\Gamma_n^\triangle$ (so they are bijective to weightings
$w:V\to\frakS$, up to their values on $\emptyset$, and possess the Laurentness
property as above). \medskip

Next we explain how to extend the properties exhibited in~\refeq{bas-Laur} to
the set $\bf{FG}_{n,n'}(\frakS)$ of flow-generated functions on $\Escr^{n,n'}$
taking values in a commutative semiring $\frakS$ with division. Instead of
P3-relation~\refeq{SP3} which has shown its importance in the flag flow case, a
central role will now be played by three special SQ-relations. The first one is
   \begin{multline} \label{eq:genP3}
    f(Xik|X'k')\odot f(Xj|X')\\
   =(f(Xij|X'k')\odot f(Xk|X'))\oplus(f(Xjk|X'k')\odot f(Xi|X')),
    \end{multline}
where $i<j<k$ and $X$ are as before, $k'\in[n']$ and $X'\subseteq[n']-\{k'\}$.
We refer to~\refeq{genP3} as the \emph{generalized P3-relation}. (The pair of
2-patterns for it is equivalent to the pair of 1-patterns for~\refeq{SP3}. In
fact, for our purposes it suffices to assume that $k'>\max(X')$.) The second
one is the SQ-relation symmetric to~\refeq{genP3}:
   \begin{multline} \label{eq:gen-coP3}
    f(Xk|X'i'k')\odot f(X|X'j')\\
   =(f(Xk|X'i'j')\odot f(X|X'k'))\oplus(f(Xk|X'j'k')\odot f(X|X'i')).
    \end{multline}
And the third one is Dodgson's type relation~\refeq{dodgson} for
$i,k,X,i',k',X'$ such that (cf.~\refeq{Dodg_orig}):
  \begin{equation} \label{eq:spec_dodg}
k-i=k'-i',\quad X=[i+1..k-1],\quad\mbox{and}\;\; X'=[i'+1..k'-1].
  \end{equation}

Let $\bf{K}_{n,n'}(\frakS)$ be the set of functions $f:\Escr^{n,n'}\to\frakS$
satisfying~\refeq{genP3}, \refeq{gen-coP3}, and~\refeq{dodgson}
with~\refeq{spec_dodg}. Besides, define $\Iscr_{n,n'}$ to be the set of pairs
$(I\subseteq[n], I'\subseteq[n'])$ such that both $I$ and $I'$ are intervals
and $|I|=|I'|$; we refer to $(I,I')$ as a (\emph{consistent}) \emph{double
interval}. Two subsets of double intervals are distinguished: let
$\Dscr^1_{n,n'}$ consist of those $(I,I')$ that the first interval $I$ is
initial (i.e., contains 1), and $\Dscr^2_{n,n'}$ of those $(I,I')$ that the
second interval $I'$ is initial; we say that such an $(I,I')$ is a
\emph{pressed} double interval.
 \begin{theorem} \label{tm:gen_bas-Laur}
For $\bf{F}:=\bf{FG}_{n,n'}(\frakS)$, ~$\bf{K}:=\bf{K}_{n,n'}(\frakS)$, and
$\Dscr:=\Dscr^1_{n,n'}\cup \Dscr^2_{n,n'}$, the following properties hold:
  \begin{itemize}
\item[{\rm(i)}] $\bf{K}$ coincides with $\bf{F}$;
\item[{\rm(ii)}] $\Dscr$ is a basis for the functions in $\bf{F}$;
\item[{\rm(iii)}] For each $f\in \bf{F}$, the values of $f$ are Laurent
polynomials (over $\frakS$) in the values $f(I|I')$, $(I,I')\in \Dscr$.
  \end{itemize}
  \end{theorem}
  \begin{proof}
~Instead of the half-grid $\Gamma_n^\triangle$ used in the flag flow case, we now
work with the grid $\Gamma=\Gamma_{n,n'}=(V,E)$ (see Fig.~\ref{fig:Gamma}). Let
us associate to each vertex $(k,k')$ of $\Gamma$ the integer rectangle
$R(k,k'):=[k]\times[k']$ and the pressed double interval
$D(k,k'):=(I,I')\in\Dscr$, where $I=[i..k]$ and $I'=[i'..k']$ (then
$\min\{i,i'\}=1$). ($D(k,k')$ is well-defined due to $|I|=|I'|$.) Observe that
$\Gamma$ has a unique $(I|I')$-flow $\phi$: its vertices are exactly those in
$R(k,k')$. Therefore, for a weighting $w:V\to\frakS$, the FG-function $f=f_w$
satisfies
   $$
   f(D(k,k'))=\odot(w(p,q)\,\colon (p,q)\in R(k,k')).
   $$
Then $w$ is expressed via the values of $f$ on $\Dscr$ as
  $$
w(k,k')=(f(D(k,k'))\odot f(D(k-1,k'-1))) \, / \,(f(D(k-1,k'))\odot
f(D(k,k'-1))),
  $$
letting $f(D(p,q)):=\underline{1}$ if $p=0$ or $q=0$. Thus, $w\mapsto f_w$ gives
a bijection between the set of weightings $w:V\to \frakS$ and $\frakS^{\Dscr^+}$,
where $\Dscr^+:=\Dscr-\{(\emptyset,\emptyset)\}$.

Next, let $f\in\bf{K}$ and consider a pair $(S,S')\in\Escr^{n,n'}$. Let
$i:=\min(S)$, $k:=\max(S)$, $i':=\min(S')$, $k':=\max(S')$. We show that
$f(S|S')$ is determined by the values of $f$ on $\Dscr$, by considering three
cases. \smallskip

(a) ~Suppose that $S$ is not an interval. Letting $X:=S-\{i,k\}$ and
$X':=S'-\{k'\}$, choosing an element $j\in[i..k]-S$, and using~\refeq{genP3},
we express $f(S|S')$ via the values of $f$ on the other five pairs occurring
there. Since for each of those five pairs $(\tilde S|\tilde S')$, the number
$\max(\tilde S)+\max(\tilde S')-\min(\tilde S)-\min(\tilde S')$ is strictly
less than $k+k'-i-i'$, we can apply induction and conclude that $f(S|S')$ is
determined by the values of $f$ on those pairs in $\Escr^{n,n'}$ where the
first term is an interval. \smallskip

(b) ~Suppose that $S$ is an interval but $S'$ is not. Acting symmetrically to
the previous case and using~\refeq{gen-coP3}, we conclude that $f(S|S')$ is
determined by the values of $f$ on double intervals. \smallskip

(c) ~Suppose that $(S,S')$ is a double interval but not a pressed one. Set
$\tilde i:=i-1$ and $\tilde i':=i'-1$; then $\tilde i\ge 1$ and $\tilde i'\ge
1$. Let $X:=S-\{k\}$ and $X':=S'-\{k'\}$ and apply~\refeq{dodgson} to $\tilde
i,k,X,\tilde i',k',X'$. Then $f(S|S')$ is expressed via the values of $f$ on
five double intervals $(I|I')$ such that $\max(I)+\max(I')+\min(I)+\min(I')$ is
strictly less than $k+k'+i+i'$. So we can apply induction. \smallskip

As a result, we obtain that any function $f\in\bf{K}$ is determined by its values
on $\Dscr$. On the other hand, we have seen that any choice of
$f_0:\Dscr^+\to\frakS$ determines a (unique) weighting $w$ in $\Gamma$, which in
turn determines a function $f\in \bf{F}$ with $f\rest{\Dscr^+}=f_0$. These
observations imply that $\bf{K}=\bf{F}$ and that $\Dscr$ is a basis for $\bf{F}$.
The Laurentness concerning $\bf{F}$ and $\Dscr$ is clear.
  \end{proof}

In light of this theorem, when we deal with a commutative semiring with
division, any SQ-relation is a consequence of the SQ-relations of three types:
the generalized P3-relation, its symmetric one, and Dodgson's type relation. In
particular, this is so when we deal with SQ-relations on minors of totally
positive matrices. \medskip

We conclude this paper with extending the SQ-relations to the functions $f$ of
minors for a wide class of matrices (where a priori $f$ need not be
flow-generated).
  \begin{prop} \label{pr:matrix-SQ}
Let $A$ be an $n\times n'$ matrix over a commutative ring $\Rscr$. Then the
function $f^A$ of minors of $A$ obeys all SQ-relations for $n,n'$.
  \end{prop}
  \begin{proof}
~Assuming $\Rscr=\Rset$, consider the parameterized matrix $P_{(t)}=A+tB$,
where $B$ is an arbitrary totally positive $n'\times n$ matrix and $t\in
\Rset$. (By the way, such a $B$ can be generated by use of the grid
$\Gamma_{n,n'}$ with a weighting in $\Rset_{>0}$.)  When $t$ is large enough,
$P_{(t)}$ becomes totally positive, and therefore the function
$f_{(t)}:=f^{P_{(t)}}$ of its minors becomes an FG-function, implying that such
an $f$ obeys all SQ-relations $S$. Substituting $f$ into $S$ gives a polynomial
$Q$ in $t$. Since $Q$ turns into zero when $t$ is large, $Q$ is the zero
polynomial. Hence $f_{(0)}=f^{A}$ obeys $S$ as well (when $\Rscr=\Rset$).
Moreover, $Q$ at $0$ is a polynomial, with integer coefficients, in the entries
$a_{ij}$ of the matrix $A$ (each being regarded as indeterminate). So it is the
zero polynomial in $a_{ij}$, and we can take an arbitrary commutative ring for
$\Rscr$, obtaining the result.
  \end{proof}

 \section*{Appendix. Matrix minors and FG-functions}

\refstepcounter{section}

In this section we prove the following

  \begin{theorem} \label{tm:minor-FG}
Let $M=(m_{ji})$ be an $n'\times n$ matrix with the entries in a field
$\frakF$. For $(I,I')\in\Escr^{n,n'}$, let $g_M(I|I')$ denote the minor of $M$
with the column set $I$ and the row set $I'$. Then $g_M$ is a flow generated
function, i.e., there exist a planar acyclic graph $G=(V,E)$ with $n$ sources
and $n'$ sinks (under the assumptions in Section~\ref{sec:intr}) and a
weighting $w:E\to\frakF$ such that $f_{G,w}=g_M$ (where $f_{G,w}$ is the
function on $\Escr^{n,n'}$ determined by $G,w$). Such $G,w$ can be explicitly
constructed.
  \end{theorem}

\noindent (For technical reasons, we assign a weighting on the edges rather
than vertices.) \smallskip

 \begin{proof}
~Due to Lindstr\"om's theorem~\cite{Li}, it suffices to construct $G,w$ so as
to provide the equality $f_{G,w}(i\,|j)=m_{ji}$ for all $i\in[n]$ and
$j\in[n']$ (where $f_{G,w}(i\,|j)$ is the sum of weights $\prod_{e\in E_{\phi}}
w(e)$ over all flows (paths) $\phi$ from the source $s_i$ to the sink $t_j$).

The desired $G$ will be constructed by concatenating a sequence of (planar
acyclic) graphs. For a better visualization, we assume that each graph $G'$ we
deal with is located within a (virtual) rectangle $R'$, with horizontal and
vertical sides, so that the sources lie (in the increasing order from left to
right) on the lower horizontal side, and the sinks lie on the upper horizontal
side of $R'$.

Suppose we are given a graph $G'=(V',E')$ with sources $s'_1,\ldots,s'_{n'}$
and sinks $t'_1,\ldots,t'_r$, and a graph $G''=(V'',E'')$ with sources
$s''_1,\ldots,s''_r$ and sinks $t''_1,\ldots,t''_{n''}$. The
\emph{concatenation} $G'\circ G''$ is obtained by mounting (the rectangle of)
$G''$ on the top of $G'$ and identifying each $t'_i$ with $s''_i$ for
$i=1,\ldots,r$. This is again a planar acyclic graph in which
$s'_1,\ldots,s'_{n'}$ and $t''_1,\ldots,t''_{n''}$ are regarded as the sources
and sinks, respectively (these lie on the lower and upper sides of an
appropriate rectangle).

Consider weightings $w':E'\to\frakF$ and $w'':E''\to\frakF$. They combine into
a weighting $w$ on the edges of the graph $G:=G'\circ G''$; namely, $w$
coincides with $w'$ on $E'$, and with $w''$ on $E''$. Let $F'$ denote the
$r\times n'$ matrix whose entries are $f_{G',w'}(i\,|j)$ for $i\in [n']$ and
$j\in [r]$, called the \emph{flow matrix} for $G',w'$. Similarly, form the
$n''\times r$ flow matrix $F''$ for $G'',w''$, and the $n''\times n'$ flow
matrix $F$ for $G,w$. By the construction, any path in $G$ beginning at a
source $s'_i$ and ending at a sink $t''_j$ is the concatenation of a path from
$s'_i$ to $t'_k$ in $G'$ and a path from $s''_k$ to $t''_j$ in $G'$, for some
$k\in[r]$. Also the corresponding converse property takes place. This implies
that
   $$
   F=F''F'.
   $$

Now we use the fact that any matrix over a field can be reduced to a
``quasi-diagonal'' matrix by applying a sequence of elementary operations, such
as: a permutation of two neighboring rows or two neighboring columns; or adding
to some row (column) the previous row (resp., column) multiplied by a factor
from the field. This implies that our $n'\times n$ matrix $M$ can be
(explicitly) represented as a product $\Pscr$ of matrices over $\frakF$ such
that: one member of $\Pscr$ is an $n'\times n$ matrix $D=(d_{ji})$ with
$d_{ji}=0$ for $i\ne j$ (a ``quasi-diagonal'' matrix), and each of the other
members of $\Pscr$ is either

(i) the matrix $\Pi_{r,i}$ obtained from the identity (over $\frakF$) matrix of
order $r$ by exchanging the columns $i$ and $i+1$, where $r\in\{n,n'\}$ and $i
\in[r-1]$; or

(ii) the matrix $A^x_{r,i}$ obtained from the identity matrix of order $r$ by
inserting the element $x$ in the intersection of column $i$ and row $i+1$,
where $r\in\{n,n'\}$, ~$i\in[r-1]$ and $x\in\frakF$.
\smallskip

In light of the above discussion, it remains to devise corresponding weighted
graphs (gadgets) for the above particular matrices.
\smallskip

1. A gadget $(G,w)$ for $D$ as above is trivial: $G$ consists of sources
$s_1,\ldots,s_n$ and sinks $t_1,\ldots,t_{n'}$ which are connected by edges
$e_i=(s_i,t_i)$ of weight $w(e_i)=d_{ii}$, $i=1,\ldots,\min\{n,n'\}$.
\smallskip

2. A gadget $(G=(V,E),w)$ for $\Pi_{r,i}$ is illustrated in the left fragment
of the picture below. Here $V=\{s_1,\ldots,s_r,t_1,\ldots,t_r\}\cup \{v\}$ and
$E=\{e_j=(s_j,t_j)\colon j=1,\ldots,r\}\cup U$, where $U:=\{(s_i,v),
(s_{i+1},v), (v,t_i),(v,t_{i+1})\}$. The weights  are: $w(e_j)=1$ for
$j=1,\ldots,i-1,i+2,\ldots,r$; ~$w(e_i)=w(e_{i+1})=-1$; and $w(e)=1$ for each
$e\in U$ (where 1 means the unit in $\frakF$). Observe that each of the pairs
$(s_i,t_{i+1})$ and $(s_{i+1},t_i)$ is connected by exactly one path and this
path has weight $1\cdot 1=1$, whereas each of the pairs $(s_i,t_i)$ and
$(s_{i+1},t_{i+1})$ is connected by two paths, one of weight --1, and the other
of weight $1\cdot 1=1$. Then the flow matrix for $(G,w)$ coincides with
$\Pi_{r,i}$.

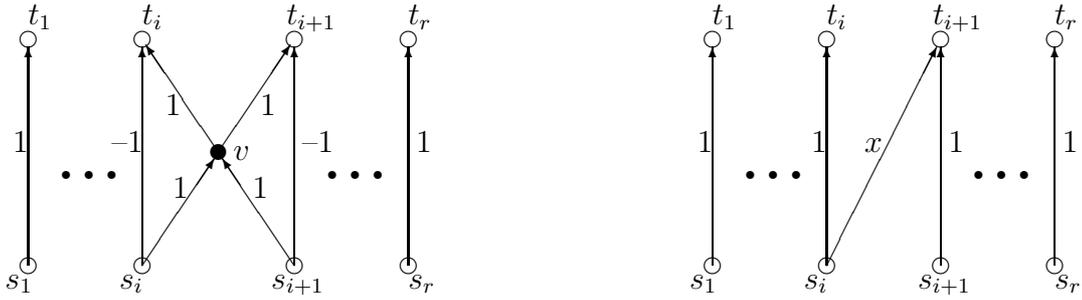
 \begin{center}
 \unitlength=1.0mm
  \begin{picture}(135,40)
   \put(0,-2){\begin{picture}(60,37)
 \put(0,5){\circle{2}}
 \put(15,5){\circle{2}}
 \put(35,5){\circle{2}}
 \put(50,5){\circle{2}}
 \put(0,35){\circle{2}}
 \put(15,35){\circle{2}}
 \put(35,35){\circle{2}}
 \put(50,35){\circle{2}}
 \put(25,20){\circle*{2}}
 \put(0,5){\vector(0,1){29}}
 \put(15,5){\vector(0,1){29}}
 \put(35,5){\vector(0,1){29}}
 \put(50,5){\vector(0,1){29}}
 \put(15,5){\vector(2,3){9.5}}
 \put(35,5){\vector(-2,3){9.5}}
 \put(25,20){\vector(2,3){9.5}}
 \put(25,20){\vector(-2,3){9.5}}
 \put(-3,2){$s_1$}
  \put(12,2){$s_i$}
  \put(32,2){$s_{i+1}$}
  \put(50,2){$s_r$}
 \put(0,37){$t_1$}
 \put(15,37){$t_i$}
 \put(34,37){$t_{i+1}$}
 \put(50,37){$t_r$}
 \put(27,19){$v$}
 \put(-2,20){1}
 \put(51,20){1}
 \put(18,25){1}
 \put(30.5,25){1}
 \put(19,14){1}
 \put(29.5,14){1}
 \put(11,20){--1}
 \put(36,20){--1}
 \put(5,17){\circle*{1}}
 \put(8,17){\circle*{1}}
 \put(11,17){\circle*{1}}
 \put(40,17){\circle*{1}}
 \put(43,17){\circle*{1}}
 \put(46,17){\circle*{1}}
  \end{picture}}
%
   \put(90,-2){\begin{picture}(60,37)
 \put(0,5){\circle{2}}
 \put(15,5){\circle{2}}
 \put(30,5){\circle{2}}
 \put(45,5){\circle{2}}
 \put(0,35){\circle{2}}
 \put(15,35){\circle{2}}
 \put(30,35){\circle{2}}
 \put(45,35){\circle{2}}
 \put(0,5){\vector(0,1){29}}
 \put(15,5){\vector(0,1){29}}
 \put(30,5){\vector(0,1){29}}
 \put(45,5){\vector(0,1){29}}
 \put(15,5){\vector(1,2){14.5}}
 \put(-3,2){$s_1$}
  \put(12,2){$s_i$}
  \put(27,2){$s_{i+1}$}
  \put(45,2){$s_r$}
 \put(0,37){$t_1$}
 \put(15,37){$t_i$}
 \put(29,37){$t_{i+1}$}
 \put(45,37){$t_r$}
 \put(-2,20){1}
 \put(46,20){1}
 \put(20,20){$x$}
 \put(13,20){1}
 \put(31,20){1}
 \put(5,17){\circle*{1}}
 \put(8,17){\circle*{1}}
 \put(11,17){\circle*{1}}
 \put(35,17){\circle*{1}}
 \put(38,17){\circle*{1}}
 \put(41,17){\circle*{1}}
  \end{picture}}
 \end{picture}
  \end{center}

3. A gadget $(G,w)$ for $A^x_{r,i}$ is illustrated in the right fragment of the
above picture. Here $G$ has $r+1$ edges: the edges $(s_j,t_j)$ of weight 1 for
$j=1,\ldots,r$, and the edge $(s_i,t_{i+1})$ of weight $x$. Then the flow
matrix for $(G,w)$ coincides with $A^x_{r,i}$.

This completes the proof of the theorem.
\end{proof}

 \medskip
\noindent \textbf{Acknowledgement.} We are thankful to C.~Krattenthaler for
pointing out to us papers~\cite{Flm,FK}. Also we thank the anonymous referees
who gave a meticulous analysis of the original text, revealed inaccuracies
there, and suggested many stylistic improvements.

\end{document}